\documentclass[a4paper,10pt]{amsart}
\usepackage[utf8]{inputenc}
\usepackage{lmodern}
\usepackage[T1]{fontenc}
\usepackage{microtype} 	
\usepackage[all]{xy}
\usepackage{amsmath,amssymb,amsfonts,amsthm,latexsym,mathrsfs}
\usepackage{mathtools}
\usepackage{ stmaryrd } 
\usepackage{xcolor}
\usepackage{hyperref}
\usepackage{cleveref}
\hypersetup{colorlinks=true,linkcolor=gray,citecolor=gray}
\usepackage[colorinlistoftodos]{todonotes}
\usepackage[hyperpageref]{backref} 
\newtheorem{lemma}{Lemma}[section]
\newtheorem{proposition}[lemma]{Proposition}
\newtheorem{theorem}[lemma]{Theorem}
\newtheorem{corollary}[lemma]{Corollary}

\newtheorem{maintheorem}{Theorem}
\newtheorem{mainproposition}[maintheorem]{Proposition}
\newtheorem{maincorollary}[maintheorem]{Corollary}

\theoremstyle{definition}

\newtheorem{definition}[lemma]{Definition}
\newtheorem{conjecture}[lemma]{Conjecture}

\theoremstyle{remark}
\newtheorem{remark}[lemma]{Remark}
\newtheorem{example}[lemma]{Example}

\newtheorem*{remark*}{Remark}
\newtheorem*{problem*}{Problem}
\newtheorem*{lremark*}{Literature remark}

\newcommand\Ima{{\rm Im}}

\newcommand{\Ma}{\operatorname{M}}
\newcommand\End{\rm End}

\newcommand\Hom{{\rm Hom}}
\DeclareMathOperator{\id}{id}
\newcommand{\Supp}{\mathrm{supp}}
\newcommand{\Span}{\mathrm{span}}
\newcommand{\fin}{\operatorname{fin}}

\DeclareMathOperator{\SPAN}{span}
\newcommand\D{\Delta}
\newcommand\eps{\varepsilon}
\newcommand\triangler{\triangleright}
\newcommand\trianglel{\triangleleft}

\newcommand\Char{{\rm char}}

\newcommand\ot{\otimes}

\newcommand\mc{\mathcal}

\newcommand{\N}{{\mathbb N}}
\newcommand{\Z}{{\mathbb Z}}

\newcommand{\C}{{\mathbb C}}

 \newcommand\restr[2]{{
   \left.\kern-\nulldelimiterspace 
   #1 
   \right|_{#2} 
   }}

\makeatletter
\providecommand\@dotsep{5}
\renewcommand{\listoftodos}[1][\@todonotes@todolistname]{%
  \@starttoc{tdo}{#1}}
\makeatother

\let\oldtocsubsection=\tocsubsection

\renewcommand{\tocsubsection}[2]{\hspace{2em}\oldtocsubsection{#1}{#2}}

\usepackage[margin=1.4in]{geometry}

\begin{document}

\title[On combinatorial solutions of PE and positive basis Hopf algebras]{On set-theoretic solutions of pentagon equation and positive basis Hopf algebras}
\author{Ilaria Colazzo}
\author{Geoffrey Janssens}

\address{(Ilaria Colazzo)\newline University of Leeds
School of Mathematics
Department of Pure Mathematics
Leeds, UK
\newline E-mail address: {\tt I.Colazzo@leeds.ac.uk}\newline
(Geoffrey Janssens) \newline Departement Wiskunde, Vrije Universiteit Brussel,
Pleinlaan $2$, 1050 Elsene, Belgium \newline E-mail address: {\tt geofjans@vub.ac.be}}

\begin{abstract}
We investigate the connection between bijective, not necessarily finite, set-theoretic solutions of the pentagon equation and Hopf algebras. Firstly, we prove that finite solutions correspond to Hopf algebras with the positive basis property. As a corollary we generalise Lu-Yan-Zhu classification to arbitrary characteristic $0$ fields $k$. 
Secondly, we study the general problem of when a Hopf algebra has a basis yielding a set-theoretic solution. Finally, we classify all (co)commutative bijective solutions. This result requires to obtain a description of all bases of a group algebra $k[G]$ yielding a set-theoretic solution. We namely show   that such bases correspond, through a Fourier transform, to splittings $A \rtimes N$ of $G$ with $A$ a finite abelian group.
\end{abstract}
\maketitle

\newcommand\blfootnote[1]{%
  \begingroup
  \renewcommand\thefootnote{}\footnote{#1}%
  \addtocounter{footnote}{-1}%
  \endgroup
}


\blfootnote{The first author acknowledges the partial support of Fonds voor Wetenschappelijk
Onderzoek (Flanders) – Krediet voor wetenschappelijk verblijf in Vlaanderen (grant
V512223N) for supporting her research visit at VUB. The second author is grateful to Fonds Wetenschappelijk Onderzoek vlaanderen - FWO (grant 88258), and le Fonds de la Recherche Scientifique - FNRS (grant 1.B.239.22) for financial support.}

\tableofcontents

\section{Introduction}
\subsection{Background}\addtocontents{toc}{\protect\setcounter{tocdepth}{1}}
A linear map $f \in \End_k(V\ot V)$ on a $k$-vector space $V$ is a solution of the pentagon equation if it satisfies the equation $f_{23} f_{13} f_{12} = f_{12} f_{23}$ in $V^{\ot 3}$. 
The pentagon equation appears naturally across several areas ranging from theory of quantum groups to mathematical physics.
In integrable systems, the pentagon equation can act as a building block for higher-dimensional analogues of the Yang–Baxter equation: Maillet constructs spectral-parameter solutions of the tetrahedron equation from solutions of the pentagon equation \cite{zbMATH00721651}.

In operator algebra, unitary solutions to the pentagon equation are known as multiplicative unitary operators and play a central role in the theory of (locally compact) quantum groups. Baaj and Skandalis showed that, in finite dimension, a multiplicative unitary operator yields a Hopf $C^*$-algebra which is in fact a finite-dimensional Kac algebra \cite[theorem 4.10]{BaSk}. Since multiplicative unitary operators have been an important tool in the theory of quantum groups, e.g. see \cite{Wo96, zbMATH01594092}.

Although this formulation is analytic, the content of the equation is fundamentally a coherence condition for associativity.
Namely, Mac Lane's pentagon axiom \cite{maclane} is precisely the coherence constraint for the associator in a monoidal category. Moreover, in a monoidal category one finds morphisms satisfying the same formal pentagon equation: Street 
\cite{St98} proves that fusion operators satisfies the pentagon equation. 
Moreover, Kashaev \cite{Ka96} proves that the pentagon equation plays for the Heisenberg double the same role that the Yang-Baxter equation plays for the Drinfeld double. Most importantly for this article Militaru \cite{Mi04} and Davydov \cite{Dav} have shown that any finite-dimensional Hopf algebra corresponds to an invertible solution of the pentagon equation.

Lately there has been considerable interest in solutions where the underlying vector space $V$ has a $k$-basis $\mc{B}$ which is $f$-invariant in the sense that $f$ restricts to a map between pure tensors in $\mc{B} \ot \mc{B}$. In that case $f$ induces a map $\restr{f}{\mc{B}} \in \End(\mc{B}\times \mc{B})$ which satisfies the set-theoretic pentagon equation.
A set-theoretic solution to the pentagon equation is often denoted by the pair $(S,s)$, where $s:S\times S\to S\times S$. 
In \cite{zbMATH05238963}, set-theoretic solutions to the pentagon equation were used in symmetrically factorizable Lie groups, while in \cite{zbMATH05984366} they were considered in the context of the discrete Liouville equation.
Involutive set-theoretic solutions $(S,s)$ of the pentagon equation, i.e. those with $s^2=\id$,
were completely classified by Colazzo--Jespers--Kubat \cite{CJK}, while Castelli \cite{castelli} recently studied further special families of bijective solutions. 
Moreover, Colazzo--Okni\'nski--Van Antwerpen  \cite{COvA} show that finite set-theoretic solution induces on $S$ a semigroup structure of the form $S\cong E\times G$,
with $E$ a left-zero semigroup and $G$ a group, and then refine this to a decomposition
$S\cong X\times A\times G$ together with a hidden group structure on $A$ and a matched product
structure on $(A,G)$ that determines the solution.\smallskip

In this article we investigate the connections between, potentially infinite, set-theoretic solutions and Hopf algebras and conversely to extract information on the Hopf algebras related to combinatorial solutions. There are two types of difficulties one is confronted with when describing (infinite) set-theoretic solutions of the PE.
\begin{enumerate}
\item[(i)] Most constructions and dualities at a operator of Hopf algebra level do not preserve bases;
\item[(ii)] There is no canonical (multiplier) Hopf algebra associated to an infinite solution.
\end{enumerate}

One aim of this article is to contribute on circumventing the first difficulty. Another is to understand the finite dimensional Hopf algebras related to finite bijective set-theoretic solutions of the pentagon equation. One of our main results states that such Hopf algebras have very nice combinatorial properties: they are exactly those admitting a positive basis. Hopf algebras with a positive basis were defined by Lu–Yan–Zhu in \cite{zbMATH01616308} and classified over the complex numbers. Later the same authors constructed set-theoretic solutions of the Yang–Baxter equation in \cite{Lu2000} from positive basis Hopf algebras. A third aim of this work is to investigate bases of potentially infinite dimensional Hopf algebras that canonically provide set-theoretic pentagon solutions. In particular, we classify such bases for the group algebra and also (infinite) cocommutative solutions.

In upcoming work, the framework of this paper, combined with representation theory, will be used for the study of combinatorial solutions of the Yang-Baxter equation. Among others, it will be used to construct an action of finite injective set-theoretic solutions of the YBE on the finite bijective set-theoretic solutions of the RPE.\

In the following subsections we will now explain in more detail the main results of this paper.

\subsection{Connection set-theoretic solutions and positive basis property}\label{subsec:connection}
Given a bialgebra $(H, \eps, 1, m, \D)$ there is a canonical associated RPE solution, see \cite{Dav} and \Cref{from bialg to PE}, as following:
\begin{equation}\label{phi constr intro}
\Phi_H := (1 \ot m)(\D \ot 1) \in \End(H \ot H).
\end{equation}
We call a basis $\mc{B}$ of the Hopf algebra \emph{$\Phi$-set theoretic} if $\Phi_H$ sends pure tensors $b \ot c \in \mc{B} \ot \mc{B}$ to a pure tensor in $\mc{B} \ot \mc{B}$. In other words if the restriction of $\Phi_H$ to $\mc{B} \ot \mc{B}$ yields a set-theoretic solution on $\mc{B} \times \mc{B}.$ If $H$ is a Hopf algebra, then the solution $\Phi_H$ is bijective.\smallskip

In \Cref{section generalitis exotic basis} we investigate some general properties of $\Phi$-set theoretic bases. For instance in \Cref{nearly positive} we show that if $\eps(b) \geq 0$ for all $b \in \mc{B}$ and also the unit element $1_H$ decomposes into $\mc{B}$ with positive coefficients, then $H$ has the \emph{positive basis property}. The latter means that $H$ has a basis for which all structural constants of all structural maps are positive.\smallskip 

If $H$ is finite dimensional, then there is somehow an inverse construction to $\Phi_H$. More precisely, \cite{Dav, Mi04} associated to any finite dimensional bijective RPE solution two finite dimensional Hopf algebras, called right (resp. left) coefficient Hopf algebra. This construction is an analogue of Baaj-Skandalis construction \cite{BaSk93} of slices in the $C^*$-algebra setting. This construction yields a two-way construction between finite bijective RPE solutions and finite dimensional Hopf algebras \cite[Theorem 3.1]{Mi04}. Given a finite bijective set-theoretic solution $(S,s)$, these Hopf algebras will be denoted $H_{r}(s)$ and $H_{\ell}(s)$. Our first main result is that in the finite setting, set-theoretic solutions correspond to Hopf algebras having a positive basis.

\begin{maintheorem}[\Cref{pos basis prop theorem}]\label{main thm positive basis}
Let $(S,s)$ be a finite bijective solution to RPE. Then $H_{\ell}(s)$ and $H_{r}(s)$ have a basis which is both positive and $\Phi$-set theoretic.
\end{maintheorem}

If the ground field $k$ is the complex numbers, then finite dimensional Hopf algebras with the positive basis property have been classified by Lu-Yan-Zhu \cite{zbMATH01616308}. They showed that they are all isomorphic to a bicrossed product Hopf algebra $k[B]^* \bowtie k[N]$ for some mashed pair of groups $(B,N)$. In \Cref{subsecti bicrossed} we recall the necessary background on mashed pairs. More precisely, they show that given a positive basis $\mc{B}$ of $H$, one can rescale the basis elements so that $\mc{B}$ corresponds to a group $G$ together with a factorization $G = B.N$, uniquely determined by $H$. The existence of such rescaling, obtained in \cite[Section 4]{zbMATH01616308}, required some analysis on the spaces of bialgebra structures on a given vector space. As a consequence of \Cref{main thm positive basis} and \cite{COvA}, we obtain a new proof of Lu-Yan-Zhu's classification which has the advantage to work for any characteristic $0$ ground field. More precisely, we obtain the following corollary.

\begin{maincorollary}\label{coro intro set basis iff positive}
    Let $H$ be a finite dimensional Hopf $k$-algebra with $\Char(k)=0$. Then the following are equivalent:
    \begin{itemize}
    \item $H$ has a $\Phi$-set theoretic basis,
    \item $H$ has the positive basis property,
    \item $H \cong k[B]^* \bowtie k[N]$ for a mashed pair of finite groups $(B,N, \trianglel,\triangler)$.
    \end{itemize}
    Moreover, if above holds, then the mashed pair is uniquely determined.
\end{maincorollary}

We expect that for any Hopf algebra, the existence of a $\Phi$-set theoretic basis implies the existence of a (potentially different) basis which is positive, see \Cref{conjecture on positive basis}.

\begin{remark*}
Conversely, one may attempt from the positive basis Hopf algebras classification in \cite{zbMATH01616308} and the methods above to obtain a Hopf theoretical proof of the classification in \cite{COvA} of all finite bijective set-theoretic solutions $(S,s)$. However, this is more subtle as it first looks like. Indeed, if we denote by $(k[S],s^v)$ the associated vector space solution, then \cite[Theorem 5.7]{Dav} combined with \cite{zbMATH01616308} yields that $s^v$ is equivalent to a solution $\Phi_{k[B]^* \bowtie k[N]} \ot 1_X$ for some finite dimensional vector space $X$. However, it is not clear to what the starting basis $S$ of $k[S]$ would correspond to in the space $k[B]^* \ot k[N] \ot X$. This for two reasons: (1) the tensor product $\Phi_{k[B]^* \bowtie k[N]} \ot 1_X$ arises after a use of the fundamental theorem of Hopf modules, which do not preserve the canonical bases and (2) the groups $B$ and $N$ are obtained after a non-explicit rescaling process. In order to solve this obstacles, one should describe all $\Phi$-set theoretical bases of $k[B]^* \bowtie k[N]$. As explained below, we will be do so in the cocommutative case.
\end{remark*}

\subsection{Infinite dimensional Hopf algebras yielding set-theoretic solutions}\label{subsec:infinite}
In case of infinite solutions there is no well-behaved analogue of the left coefficient Hopf algebras. In particular, there is no associated Hopf algebra in general. For instance it will follow from the main results below, \Cref{main th cocomm} and \Cref{reconstructino theorem intro}, that the solution from \Cref{sol from grp alg} on an infinite abelian group $A$ is not induced from a $\Phi$-set theoretic basis of a Hopf algebra. However, the construction \eqref{phi constr intro} also works for multiplier Hopf algebras and the aforementioned solution originates so via the multiplier Hopf algebra $k_{\fin}^A$ of finitely supported functions on $A$. We expect that this phenomenon holds for any bijective set-theoretic solution. A conceptual reason is that set-theoretic solutions, in contrary of vector space solutions, enjoy duality (cf. \Cref{dual in set } and \Cref{connection linear and pullback sol}).

In the remainder of the article we investigate which infinite solutions can be obtained from Hopf algebras. We call such solutions \emph{reachable}. In practice this amount to the problem of determining whether a given Hopf algebra $H$ has a $\Phi$-set theoretic basis.

\subsubsection*{Cocommutative Hopf algebras} 

One of our aims is to classify non-finite bijective cocommutative solutions of the RPE. To do so we relate in \Cref{cocomm from sol to hopf}  such solutions  as following to Hopf algebras.

\begin{maintheorem}[\Cref{classif cocomm sol}]\label{main th cocomm}
Let $(S,s)$ be a reachable cocommutative bijective solution of RPE on a set $S$. Then there exists a group $G$ and $\Phi$-set theoretic basis $\mc{B}$ of $k[G]$ such that 
$$ s = \phi_{\mc{B}} \times 1_{X}$$ 
for some set $X$, where $\phi_{\mc{B}}$ is the set-theoretic solution on $\mc{B} \times \mc{B}$ associated to $\restr{\Phi_{k[G]}}{\mc{B} \ot \mc{B}}$.
\end{maintheorem}

It is tempting to believe that \Cref{main th cocomm} follows from the work of Baaj-Skandalis \cite{BaSk93, BaSk03} by lifting the linearised solution $(k[S],s^v)$ to a multiplicative unitary on the Hilbert space $\ell^2(S)$. Their result would provide a unitary to $L^2(G \times X)$ for some locally compact group $G$ and discrete space $X$. However, it is not clear where the basis $S$ is send to in $L^2(G \times X)$. The aforementioned \Cref{sol from grp alg}, illustrates this. This problem will be considered in upcoming work.

\subsubsection*{Non-canonical bases of group algebras}

\Cref{main th cocomm} reduces the problem of classifying all (reachable) cocommutative set-theoretic solutions to describing $\Phi$-set theoretic bases of group algebras $k[G]$. At first surprising, such bases do not need to be multiplicativly closed. For instance any mashed pair decomposition $A \bowtie N$ of $G$ with $A$ a finite abelian group acting trivially on $N$, yield a $\Phi$-set theoretic basis of $k[G]$ using Fourier analysis. 

Concretely, when $A$ is finite, for each $\chi$ in the character group $A^{\vee}$ one has the idempotent $e_{\chi} := \frac{1}{|A|}\sum_{a\in A}\chi(a^{-1})\,a .$ The set  $\mathcal{B}_{A^\vee} := \{\, e_\chi\,u \mid \chi\in A^\vee,\ u\in N\,\}$  is a $k$-basis of $k[G]$. Furthermore the right action of $N$ on $A$ induces a left action of $N$ on $A^\vee$ which permutes the idempotents $e_{\chi}$.

\begin{mainproposition}[\Cref{thm:dual}]\label{main th sol from mashed}
Let $(A,N)$ a mashed pair with $A$ a finite abelian group acting trivially on $N$ and let $G = A \rtimes N$. With notations as above, the basis $\mathcal{B}_{A^\vee}$ is $\Phi$-set theoretic. Moreover, $\Phi_{k[G]}$ restricts to following set-theoretic solution of RPE:
\[
\phi_{A^\vee\bowtie N}:\ (A^\vee\times N)\times(A^\vee\times N)\to (A^\vee\times N)\times(A^\vee\times N),
\]
given by
\begin{equation}\label{set sol on mashed intro}
\phi_{A^\vee\bowtie N}\bigl((\alpha,u),(\beta,v)\bigr)
=
\bigl((\alpha(u\cdot\beta)^{-1},u),\ (u\cdot\beta,uv)\bigr).
\end{equation}
\end{mainproposition}

Constructing the non-trivial $\Phi$-set theoretic basis $\mc{B}$ in \Cref{main th sol from mashed} crucially relies on $A$ being finite and for $G$ to have torsion elements. In \Cref{thm:two-point-obstruction} we show how non-torsion elements restrict the possible support of an element in $\mc{B}$. Furthermore, it is shown in \Cref{no non-triv sol of torsion-fre} that in the torsion-free case there is only the trivial $\Phi$-set theoretic basis. More generally, our final main theorem shows that all $\Phi$-set theoretic bases of $k[G]$ must be of the form as in \Cref{main th sol from mashed}. As a consequence all cocommutative solutions would be equivalent to one of the form \eqref{set sol on mashed intro}.

\begin{maintheorem}[\Cref{Classification theorem basis grp alg}]\label{reconstructino theorem intro}
Let $G$ be a group and $\mc{B}$ a $\Phi$-set theoretic basis of $k[G]$. Then there exists $A,N \leq G$ such that
\begin{enumerate}
    \item $G \cong A \rtimes N$ with $A$ a finite abelian group,
    \item $\mc{B}$ is a scalar multiple of the basis $\mc{B}_{A^{\vee}}= \{ e_{\chi} u \mid  \chi \in A^{\vee}, u \in N\}$,
    \item $A= \{ b\in \mc{B} \mid 1 \in \Supp(b) \}$ are the idempotents in $\mc{B}$,
    \item $\phi_{\mc{B}} := \restr{\Phi}{\mc{B} \ot \mc{B}}$ is equivalent to the solution \eqref{eq:Phi-dual}.
\end{enumerate}
\end{maintheorem}

\subsubsection*{Non-existence for domains}

A first, non-difficult but important, fact for \Cref{main th cocomm} is that cocommutative solutions can only arise as the restriction to a $\Phi$-set theoretic basis of a solution $\Phi_H$ if $H$ is a cocommutative Hopf algebra. This allows to invoke Cartier-Konstant-Milnor-Moore classification saying that $H \cong U(\mathfrak{g}) \rtimes k[G]$ with $\mathfrak{g} := P(H)$ the Lie algebra consisting of primitive elements in $H$ and $G := G(H)$ the group of group-like elements. For universal envelopings, and more generally domains, we show that no $\Phi$-set-theoretic basis exists.

\begin{maintheorem}[\Cref{domain no set solution} \& \Cref{no set for Lie}]
Let $H$ be a Hopf algebra which is a domain. If $H$ has a $\Phi$-set-theoretic basis $\mc{B}$, then for each $b \in \mc{B}$ we have that $\D_H(b) = \eps_H(b)^{-1}\,  b \otimes b$. Consequently, if $\mathfrak{g}$ is a finite dimenisonal Lie algebra, then $U(\mathfrak{g})$ has no $\Phi$-set-theoretic basis.
\end{maintheorem} 

In fact, in general by \Cref{thm:groupalgebra-under-coalgebasis}, having a basis $\mc{B}$ such that $\D(\mc{B}) \subseteq k^* \ \mc{B} \ot \mc{B}$  implies that the Hopf algebra must be isomorphic to a group algebra. By one of the notorious Kaplansky conjectures, a group algebra $k[G]$ would be a domain if and only if $G$ is torsion-free.

\subsection{Outline}\label{subsec:outline}

In \Cref{section overview sols} we review different notions of solutions of the PE that can be found in the literature. The main differences being in which category they live, i.e. are they built from a set, vector space or algebra. We recall how to explicitly pass from one setting to another. Thereafter in \Cref{section background coeff alg} we recall Davydov \cite{Dav} and Militaru's \cite{Mi04} construction of a finite dimensional Hopf algebra associated to a RPE solution. These Hopf algebras are called the coefficient algebras and are inspired by Baaj-Skandalis construction \cite{BaSk93} in the operator algebra setting. Thereafter we specialize the construction to the case of set-theoretic solutions and describe a spanning set.

This is followed by \cref{section From hopf to sol} where we recall the converse construction. In other words, how to canonically associate a solution to any Hopf algebra $H$. This section contains several instrumental examples and introduces the concept of a $\Phi$-set theoretic basis $\mc{B}$, a crucial concept in the rest of the paper. It is also explained that in the finite dimensional setting the Davydov-Militaru Hopf algebra determines a vector space solution up to a trivial factor, but do not behave well with remembering specific bases.

The aim of \Cref{section positive basis} is to relate the class of finite dimensional Hopf algebras with the positive basis property and finite bijective solutions of the pentagon equation. This is done by constructing an explicit basis of the Hopf algebra attached to a set-theoretic solution. This basis is shown to have structural constants $\{ 0,1 \}$. 
This section fully clarifies the finite dimensional setting. Hence the remainder of paper focuses on Hopf algebras of arbitrary dimension. Firstly, in \Cref{cocomm classif section}, we show that a Hopf algebra which is a domain or is cocommutative can only yield a set-theoretic solution if it is a group algebra. Subsequently,
in \Cref{reconstruction theorem section} we classify all set-theoretic bases of a group algebra. This classification depends on a construction introduced in \Cref{solutions from grp alg}. \smallskip

\vspace{0,2cm}
\noindent \textbf{Acknowledgment.} We would like to thank Leandro Vendramin for interesting discussions and especially for suggesting to consider the positive basis property. The authors are grateful to CIRM for supporting this work by providing ideal working conditions during their research in residence. the  The second author thanks Kenny De Commer for useful conversations on multiplier Hopf algebras and work of Baaj-Skandalis. He would also like to express his gratitude to \v{S}pela \v{S}penko for all her support during the writing of this paper.

\section{Overview and connection on the different types of Pentagon Solutions}\label{section overview sols}

Independent of the type of algebraic object one is dealing with we will study equations of the form

\begin{align}
Z_{12}Z_{13}Z_{23} = Z_{23} Z_{12} \tag{RPE}\\
Z_{23} Z_{13} Z_{12} = Z_{12} Z_{23} \tag{PE}
\end{align}
where PE stands for \emph{pentagon equation} and RPE for the \emph{reversed pentagon equation}. In some parts of the literature the PE is called the \emph{Hopf equation}, in which case the RPE is referred to simply as the pentagon equation. The meaning of the operator $Z_{ij}$, and hence of a solution to such equations, depends on the category in which we are working.

In this section we recall the necessary background on the various notions of solutions of the pentagon equation (for sets, vector spaces and algebras) that appear in the literature, and we explain how these notions are related. The content of this section is well-known to the experts, since we are not aware of a  single reference that collects all these facts, we include it here.

\subsection{Background on pentagon equation}

Let $S$ be a set. In the set-theoretic setting, the (R)PE equation in this case is an identity in $\End(S\times S\times S)$.

\begin{definition}\label{def set PE}
A set-theoretic solution to the RPE (resp. PE) is a pair $(S,s)$ where 
\begin{enumerate}
    \item $S$ is a set
    \item $s \in \End (S\times S)$ satisfying RPE (resp. PE)
\end{enumerate}
with the standard notation $s_{12} = s \times \id_S, s_{23} = \id_S \times s$ and $s_{13} = (\id_s \times \tau)(s \times \id_s)(\id_s \times \tau)$ where $\tau(u,v) = (v,u)$ denotes the flip map. 
\end{definition}

If $(S,s)$ is a set-theoretic solution to the RPE (resp. PE), it is convenient to denote it by
$$s(x,y) = (\psi_y(x), y\circ x) \quad (\text{resp. } s(x,y) =(xy, \theta_x(y))).$$
Expanding the RPE (resp. PE) for $(x,y,z)\in S^3$ yields 
\begin{align}
    z\circ(y\circ x) = (z\circ y) \circ x\quad &(\text{resp. } x\cdot(y\cdot z))\label{eq1}\\
    \psi_z(y\circ x) = \psi_z(y)\circ \psi_{z\circ y}(x)\quad&(\text{resp. } \theta_x(y\cdot z) = \theta_x(y)\cdot \theta_{x\cdot y}(z))\label{eq2}\\
    \psi_{\psi_z(y)}\psi_{z\circ y}(x) =\psi_y(x)\quad &(\text{resp. }\theta_{\theta_x(y)}\theta_{x\cdot y}(z) =\theta_y(z)).\label{eq3}
\end{align}
In particular, \eqref{eq1}, shows that $\circ$ is an associative binary operation on $S$.

From the above identities one checks immediately that, given a set-theoretic solution $(S,s)$ to the RPE, the map
$$t(x,y) = (x\circ y, \psi_x(y))$$
defines a set-theoretic solution $(S,t)$ to the PE. 
Moreover, since $t = \tau s \tau$ (with $\tau(x,y)=(yx)$), it follows that $t$ is bijective if and only if $s$ is bijective. 

Finally, let $(S,s)$ be a solution of the RPE and write $s(x,y)=(\psi_y(x),y\circ x)$ and $s^{-1}(x,y) = (xy, \theta_x(y))$. A direct computation shows that $ss^{-1}=\id_{S\times S}= s^{-1}s$  is equivalent to the following relations:
\begin{align}
    \psi_y(x)(y\circ x) = x,\\
    \theta_{\psi_y(x)}(y\circ x) = y,\\
    \psi_{\theta_x(y)}(xy) = x, \label{eq:psitheta}\\ 
    \theta_{x}(y) \circ (xy)= y.
\end{align}
In \Cref{section generalitis exotic basis} we refer to theses as the \emph{inverse relations}.\medskip

Given a set-theoretic solution, one can associate to it solutions in linear settings (vector spaces or algebras). For instance, linearising yields an operator on $V=k[S]$ and hence an equation in $\End(V^{\otimes 3})$. 

\begin{definition}\label{def VS PE}
A \emph{vector space solution} to the PE (resp. RPE) is a pair $(V,f)$ where 
\begin{enumerate}
    \item $V$ is a $k$-vector space
    \item $f \in \End_k(V\otimes V)$ satisfying PE (resp. RPE)
\end{enumerate} 
with $f_{12} = f \otimes \id_V, f_{23}= \id_V \otimes f$ and $f_{13} = (\id_V \otimes \tau)(f \otimes \id_V)(\id_V \otimes \tau)$
where $\tau(a\otimes b) = b\otimes a $ is the bilinear extension of the flip map.
\end{definition}

Some authors prefer to view the PE inside $\End_k(V)^{\ot 2}$, or, in other words, to work in a $k$-algebra. 

\begin{definition}\label{def algebra PE}
An \emph{algebra solution} to the PE (resp. RPE) is a tuple $(A,R)$ with 
\begin{enumerate}
    \item $A$ a $k$-algebra
    \item $R := \sum_i R_i^{(1)} \otimes R_i^{(2)} \in A \otimes A$ satisfying the PE (resp. RPE)
\end{enumerate}
where $R_{12}= R \times 1_A, R_{23} = 1_A \otimes R$ and $R_{13} = \sum_i R_i^{(1)}\otimes 1_A \otimes R_i^{(2)}.$
\end{definition}

Note that in \Cref{def algebra PE} the PE is an equation in $A \ot A \ot A$.

\subsection{Connections between set-theoretic and vector space solutions}

Starting from a set $S$ there are two canonical ways to associate to it a $k$-vector space:
\begin{enumerate}
    \item $k[S]$, the free $k$-vector space on $S$;
    \item $k^S$, is the vector space of functions $S \to k$.
\end{enumerate}

\subsubsection{Via linearisation.}

We consider the canonical $k$-linear map
$$\psi_{(n)} : k[S]^{\ot n} \rightarrow k[S^n]: s_1 \ot \cdots \ot s_n \mapsto (s_1,\cdots, s_n).$$
Note that $\psi_{(n)}$ is an isomorphism for any $n\in \N$ with inverse induced by the universal property (this holds even when $S$ infinite). Moreover, if $S$ is a semigroup, then $\psi_{(n)}$ is a $k$-algebra morphism.
Next given $s\in \End(S^{n})$, let $\overline{s} \in \End_k(k[S^{n}])$ denote its $k$-linear extension. We then define from a set-theoretic the $k$-linear map $s^{v} := \psi^{-1}_{(n)}\circ \overline{s} \circ \psi_{(n)}$.

\begin{proposition}
    The pair $(S,s)$ is a set-theoretic solution of PE (resp. RPE) if and only $(kS, s^v)$ is a vector space solution of the PE (resp. RPE).
\end{proposition}
\begin{proof}
    For $(i,j)\in\{(1,2),(1,3),(2,3)\}$ one has $(s^{v})_{ij}=\psi_{(3)}^{-1}\circ \overline{s_{ij}}\circ \psi_{(3)}$.
    The claim follows considering the conjugation by $\psi_{(3)}$.
\end{proof}

\subsubsection{Via the space of functions}

Alternatively, one can also work with $k^S$. For $f \in \Hom(X,Y)$ we write $f^*\in \Hom_k(k^Y,k^X)$ for the pullback, defined by $f^*\circ\varphi = \varphi\circ f$ for $\varphi \in k^Y$. In particular, if $s \in \End(S^n)$, then $s^* \in \End(k^{S^n})$. Since $(\cdot)^*$ is a contravariant functor from \underline{Set} to \underline{Vec}$_k$, it exchange the PE and the RPE. To express the pullback in a suitable way we introduce the follwing maps.

\begin{definition}\label{definition fs}
Let $S$ be a set, $k$ be a field, and $s\in \End(S^n)$. Define $f_s = \theta_n^{-1}\circ s^*\circ \theta_n$ where $\theta_n: k^S \ot \cdots \ot k^S \rightarrow k^{S^n}$ is given by 
$$
     \theta_n(f_1\ot\cdots\ot f_n): S^n \longrightarrow k: (x_1,\cdots,x_n) \longmapsto \prod_{i=1}^n f_i(x_i).
$$
\end{definition}

If $S$ is finite, then $\theta_n$ is an isomorphism; however its inverse is not canonical. 

\begin{proposition}\label{dual in set }
The pair $(S,s)$ is a solution of PE (resp. RPE) if and only if $(k^S, f_s)$ is a solution of the RPE (resp. PE).
\end{proposition}
\begin{proof}
    A direct verification shows, for $(i,j)\in\{(1,2),(1,3),(2,3)\}$, that $(f_s)_{ij} = \theta_3^{-1}\circ(s_{ij})^*\circ\theta_3$. Thus
    $$(f_s)_{12}(f_s)_{13}(f_s)_{23}=\theta_3^{-1}s_{12}^*s_{13}^*s_{23}^*\theta_3=\theta_3^{-1}(s_{23}s_{13}s_{12})^*\theta_3$$
    and
    $$(f_s)_{23}(f_s)_{12} = \theta_3^{-1}s_{23}^*s_{12}^*\theta_3= \theta_3^{-1}(s_{12}s_{23})^*\theta_3.$$
    Therefore the RPE for $f_s$ in $\End_k(k^S\ot k^S \ot k^S)$ is equivalent to the equation 
    $$(s_{23}s_{13}s_{12})^* = (s_{12}s_{23})^*$$ in $\End_k(k^{S\times S \times S})$. Now evaluating on delta-functions shows that this holds if and only if $s_{23}s_{13}s_{12} = s_{12}s_{23}$ in $\End(S^3)$, i.e. $(S,s)$ satisfies the PE. The reverse implication and the swap of PE and RPE are analogous.
\end{proof}

\begin{example}\label{dual PE group solution}
    Let $G$ be a group and consider the map $s \in \End(G^2)$ defined by $s(g,h)= (gh,h)$. This is a set-theoretic solution of the PE. We compute the pullback $s^* \in \End(k^{S\times S})$. For a basis element $\delta_{(g,h)}$ we have, by definition, that 
    $$s^*(\delta_{(g,h)}) (a,b) = \delta_{(g,h)}(s(a,b)) = \delta_{(g,h)}(ab,b) = \left\lbrace \begin{array}{ll}
        1 & \text{if } h = b \text{ and } gb^{-1}=a\\
        0 & \text{else} 
    \end{array} \right.$$
Therefore $s^*(\delta_{(g,h)})= \delta_{(gh^{-1},h)}$. Note that more generally $s^*(\delta_{(g,h)})=\delta_{s^{-1}(g,h)}$. 
Consequently, on basis elements $\delta_g\ot \delta_h$ of $k^S \ot k^S$ one obtains $f_s(\delta_g\ot \delta_h) = \delta_{gh^{-1}}\ot\delta_h$, which gives that $(k^G, f_s)$ is a vector space solution of the RPE.
\end{example}

\subsubsection{Connection between the two constructions}

For a \emph{bijctive} solution $(S,s)$ of the RPE the constructions above give two vector space solutions:  (i) $(k[S],s^v)$, which is a solution of RPE and (ii) $(k^S,f_s)$, which is a solution of PE. They are related as follows.

\begin{proposition}\label{connection linear and pullback sol}
Let $(S,s)$ be a set-theoretic solution of the PE. Then 
$$f_s = (\sigma \ot \sigma)(s^{-1})^{v} (\sigma^{-1} \ot \sigma^{-1})$$ with $\sigma : k[S] \rightarrow k^S : g \mapsto \delta_g$.
\end{proposition}
\begin{proof}
It is enough to check the identity on the basis $\{\delta_g \mid g \in S\}$ of $K^S$.
On the one hand,
$$f_s(\delta_g \ot \delta_h) = \theta^{-1}_{(2)}s^* (\delta_{(g,h)}) = \theta^{-1}_{(2)} \delta_{s^{-1}(g,h)}.$$
And on the other hand,
$$\left( (\sigma \ot \sigma)(s^{-1})^{v} (\sigma^{-1} \ot \sigma^{-1}) \right) (\delta_g \ot \delta_h) = (\sigma \ot \sigma) (s^{-1})^v (g\ot h) = (\sigma \ot \sigma) \psi^{-1}_{(2)}s^{-1}(g,h)$$
which equals the expression above. This proves the claim.
\end{proof}

\subsection{Relation between vector space and algebra solutions}

To compare \Cref{def VS PE} and \Cref{def algebra PE} we use the following algebra map. For $n \in \mathbb{N}$ define
$$\phi_{(n)}: \End_k(V)^{\ot n} \longrightarrow \End_k(V^{\ot n}): g_1 \otimes \cdots \ot g_n \longmapsto \phi_{(n)}(g_1 \otimes \cdots \ot g_n)$$
defined as $\phi_{(n)}(g_1 \otimes \cdots \ot g_n)(v_1 \ot \cdots \ot v_n) := g_1(v_1) \ot \cdots \ot g_n(v_n)$.

It is a standard exercise that the following holds.

\begin{lemma}
  For all $n\in \N$ the map $\phi_{(n)}$ is a $k$-algebra monomorphism and hence an isomorphism if $\dim_k V$ is finite. Therefore if $(\End_k(V),R)$ is an algebra solution of PE (resp. RPE), then $(V,\phi_{(2)}(R))$ is a vector space solution of PE (resp. RPE). Converse holds if $\dim_k(V) < \infty.$
\end{lemma}

In practice we will often pass from a vector space solution to an algebra solution, so we need an explicit description of $\phi_{(2)}^{-1}$. Assume that $\dim_kV< \infty$ and to fix a basis $\mathcal{B} = \{ v_i \mid 1 \leq  i\leq \dim V  \}$ of $V$. For $1\leq i,j\leq \dim_kV$ define
$$S_{ij} : V \longrightarrow V, \qquad S_{ij}(v_j)= v_i,\quad S_{i,j}(v_k)=0 \ (k\neq j).$$
Under the identification $\End_k(V) \cong \Ma_{\dim V}(k)$ the endomorphism $S_{ij}$ corresponds to the elementary matrix $e_{ij}$. Thus the $\{S_{ij}\}$ is a basis  of $\End_k(V)$ and consequently $\{ \phi_{(2)}(S_{ij}\ot S_{kl}) \}$ is basis of $\End_k(V\ot V)$.

Now let $T \in \End_k(V \ot V)$ and denote $T : V \ot V \rightarrow V \ot V : v_i \ot v_j \mapsto \sum_{k,l} \alpha_{kl}^{ij} v_k \ot v_l$. Then a direct computation yields
\begin{equation}\label{from VS to alg solution}
\phi_{(2)}\left( \sum_{k,l,t,h} \alpha_{kl}^{th} S_{kt} \ot S_{lh} \right) = T.
\end{equation}
Note that \eqref{from VS to alg solution} provides an explicit formula for $\phi_{(2)}^{-1}(T)$ and it uses that the index sets are finite, i.e. $\dim_k V$ finite.

\section{Hopf algebras from set-theoretic RPE solutions}\label{section background coeff alg}

We follow the constructions of Militaru and Davydov. 
Let $A$ be a $k$-algebra with $R=\sum R^{1}\otimes R^{2} \in A^{\otimes 2}$ a RPE algebra solution. Then the subspaces
\begin{align*}
    R_{(\ell)} & :=\left\{\left.\sum a^{\ast}(R^{2})R^{1}\ \right|\  a^{\ast}\in A^{\ast}\right\} \\
 R_{(r)} &:=\left\{\left.\sum a^{\ast}(R^{1})R^{2}\ \right|\  a^{\ast}\in A^{\ast}\right\}
\end{align*}
are called \emph{left}, respectively \emph{right coefficients of $R$}. If $A$ is finite dimensional, Militaru \cite{Mi04} and Davydov \cite{Dav} have shown that the following structure maps equip $R_{(\ell)}$ and $R_{(r)}$ with the structure of a Hopf algebra.

\begin{proposition}[Davydov, Militaru]\label{structure maps coefficients hopf algebras}
The subspace $R_{(r)}$ (resp. $R_{(\ell)}$) of $A$ is Hopf algebra for the following structural maps:
\begin{itemize}
    \item Unit and multiplication: $R_{(r)}$ (resp. $R_{(\ell)}$) is a unital subalgebra of $A$,
    \item Co-unit : 
    $$\epsilon_{r}: R_{(r)} \rightarrow k : \sum a^*(R^1) R^2 \mapsto a^*(1),$$
    resp. $\epsilon_{\ell}(\sum a^*(R^2) R^1)= a^*(1)$.
    \item Co-multiplication: 
    $$\D_{r}: R_{(r)} \rightarrow R_{(r)} \ot R_{(r)}: x \mapsto R(x\otimes 1)R^{-1}, $$ 
    resp. $\D_{\ell}(x) = R^{-1}(1 \otimes x)R $.
    \item Antipode: 
    $$S_{r}: R_{(r)} \rightarrow R_{(r)}: (a^* \ot 1)(R) \mapsto (a^* \ot 1)(R^{-1}),$$
    resp. $S_{\ell}((1\ot a^*)(R)) = (1 \ot a^*)(R^{-1}).$
\end{itemize}
\end{proposition}

\begin{remark}\label{militaru vs davydov}
    In \cite[Theorem 2.1]{Mi04} the co-unit and antipode were defined using a basis of $R_{(\ell)}$ (and $R_{(r)}$). The definitions given in \Cref{structure maps coefficients hopf algebras} are taken from \cite[Section 5]{Dav}. Indeed, $R_{(r)}$ in \cite{Mi04} is isomorphic to $\Ima (\lambda)$ in \cite{Dav}, where
    $$\lambda: A^* \rightarrow A: a^* \mapsto (a \otimes 1)(R) .$$
    Furthermore, $R_{(l)}$ is isomorphic to $\Ima(\rho)$ with $\rho: A^* \rightarrow A: \omega \mapsto (I \otimes \omega)(R)$. Through these identifications, it is shown in
    \cite[Proposition 5.4 \& 5.5]{Dav} that the co-unit and antipode from \cite{Mi04} can be expressed as in \Cref{structure maps coefficients hopf algebras}.
\end{remark}

 In case that $R = s^{A}$ originates from a set-theoretic solution $(S,s)$ to RPE we will denote the associated Hopf algebras by 
$$H_{\ell}(s) := R_{(\ell)} \text{ and } H_r(s) := R_{(r)}.$$ 
Concretely, if $s(x,y) = (\psi_y(x), y \circ x)$ is a set-theoretic solution of PE, then from \eqref{from VS to alg solution} we see that
$$s^{A} = \sum\limits_{x,y\in S} S_{\psi_y(x), x}\otimes S_{y \circ x,y}.$$
Furthermore, denoting $s^{-1}(x,y) =(xy,\theta_x(y))$ one can verify that 
\begin{equation}
    (s^{A})^{-1} = (s^{-1})^{A} = \sum\limits_{x,y\in S} S_{xy, x}\otimes S_{\theta_x(y),y}.
\end{equation}

We have following generating sets for the associated Hopf algebras.

\begin{proposition}\label{leftrightinv}
Let $s(x,y) =(\psi_y(x), y \circ x)$ be a solution of RPE. Then
\begin{align}
    H_{\ell}(s) &= \Span_k\left\{\left.\sum\limits_{x \in \mu(x',y')} S_{\psi_{y'}(x),x}\ \right|\ (x',y')\in S^{2}\right\} \label{leftinv}\\
    H_r(s) &= \Span_k\left\{\left.\sum\limits_{ y \in \nu(x',y')} S_{y \circ x',y}\ \right|\ (x',y')\in S^{2}\right\} \label{rightinv}
\end{align}
where $\nu(x',y')= \{ y\in S\mid  \psi_y(x') = \psi_{y'}(x')\}$ and $\mu(x',y') = \{ x \in S \mid y'\circ x = y'\circ x' \}.$
\end{proposition}
\begin{proof}
We start with the left coefficients. Consider a fixed tuple $(x',y')\in S^2$ and the associate Kronecker delta-function $ \delta_{S_{y'\circ x',y'}} \in \End_k(k[S])^*$. Applying $\id \ot \delta_{S_{y'\circ x',y'}}$ on $s^{A}$ yields
\begin{align*}
    (\id \ot \delta_{S_{y'\circ x',y'}})(s^{A}) & = (\id \ot \delta_{S_{y'\circ x',y}}) \left( \sum\limits_{x,y\in S} S_{\psi_y(x), x}\otimes S_{y \circ x,y} \right) \\
    & =  \sum\limits_{x,y \in S} S_{\psi_y(x), x}\,  \delta_{S_{y'\circ x',y'} \,,\, S_{y \circ x,y}} \\
    & = \sum\limits_{x \in \mu(x',y')} S_{\psi_{y'}(x),x}\
\end{align*}
The right coefficients follows in an analogue way.
\end{proof}

In \cite[Theorem 2.1]{Mi04} it was shown that $R_{(\ell)} \cong R_{(r)}^*$. In fact there is a bialgebra pairing between them \cite[Lemma 5.2]{Dav}. At a set-theoretic level this translates to the following. 

\begin{proposition}\label{coefficients PE vs RPE}
Let $(S,s)$ be a set-theoretic solution of RPE. We have that:
$$H_{\ell}(s) \cong H_r(\tau s^{-1}\tau).$$
Furthermore, if $s = s_1 \times s_2$, then $H_{\ell}(s) \cong H_{\ell}(s_1) \ot H_{\ell}(s_2)$.
\end{proposition}

\begin{example}\label{Hopf op grp and dual solution}
Let $G$ and $H$ be finite groups and consider the following map
$$s:(H \times G)^{\times 2} \rightarrow (H \times G)^{\times 2}: ((a,g),(b,h)) \mapsto ((ab^{-1},g), (b,hg)).$$
A direct verification shows that this is a solution of the RPE $s_{12}s_{13}s_{23} = s_{23}s_{12}.$ We claim that the associated Hopf algebra $H_l(s)$ has following nice form:
$$H_l(s) \cong k[H] \otimes_k (k[G^{op}])^*.$$
Indeed, in this case \Cref{leftrightinv} becomes
$$H_l(s) = \SPAN_k \{ \sum_{a\in H} S_{(ab^{-1},g), (a,g)} \mid b\in H, g \in G \}.$$
Using the '$g$-coordinate' one sees readily that the generating set is a $k$-basis. With a lengthy but direct computation one verifies that 
$$f : H_l(s) \rightarrow k[H] \otimes_k (k[G^{op}])^*: \sum_{a\in H} S_{(ab^{-1},g), (a,g)} \mapsto b \otimes \rho_g$$
is a Hopf-algebra isomorphism. 
\end{example}

\section{Combinatorial RPE solutions arising from Hopf algebras}\label{section From hopf to sol}

The aim of this section is to recall the canonical RPE solution associated to a Hopf module over a (not necessarily finite dimensional) Hopf algebra and investigate interesting examples. More precisely, \Cref{subsection sol from hopf} contains background and introduce the core concept of a $\Phi$-set theoretic basis. In \Cref{section bicrossed} we recall the construction of mashed pair of groups and Hopf algebras, yielding the bicrossed product Hopf algebra. Thereof we compute the associated RPE solution and propose a formal framework of mashed pair of RPE solutions.

\subsection{Background on multiplier Hopf algebra}\addtocontents{toc}{\protect\setcounter{tocdepth}{2}} \label{background multiplier}

We refer to \cite{Tim} for a good account on multiplier Hopf algebras. Here we only recall the bare minimum to understand statements later in the paper.

\noindent {\it Convention:} In the setting of multiplier algebras, we no longer assume that an algebra has an identity. Therefore in subsequent section, the terminology (Hopf) algebra will always mean unital, but when multiplier is added we do not assume unitality. 

To start,

\begin{definition}
An associative $k$-algebra $A$ is \emph{non-degenerate} if $\Span \{ab \mid a,b \in A\}=A$ and for all $a \in A$ one has 
$$
Aa=0 \;\Rightarrow\; a=0 \qquad\text{and}\qquad aA=0 \;\Rightarrow\; a=0.
$$
A \emph{multiplier} of $A$ is a pair $(L,R)$ of $k$-linear maps $L,R:A\to A$ such that for all $x,y \in A$ holds
$$
L(xy)=L(x)y,\qquad R(xy)=xR(y),\qquad xL(y)=R(x)y\quad(x,y\in A).
$$
The set of all multipliers is denoted $M(A)$. 
\end{definition}

The set $M(A)$ can be made into an algebra, called the \emph{multiplier algebra of $A$}.

\begin{proposition}
    The set $M(A)$ endowed with the operations
    \begin{align*}
    (L_1,R_1) + \lambda (L_2,R_2) & := (L_1 + \lambda L_2, R_1 + \lambda R_2) \\
    (L_1,R_1)\cdot (L_2,R_2) & := (L_1\circ L_2, R_2 \circ R_1)
    \end{align*}
    is an algebra. Furthermore, $A$ embeds into $M(A)$ via $a\mapsto (L_a,R_a)$ with $L_a(x)=ax$, $R_a(x)=xa$.
\end{proposition}

Note that a unital algebra is non-degenerate. In fact $A = M(A)$ if and only if $A$ is unital.

Give two non-degenerate algebras $A$ and $B$. A homomorphism $\phi: A \rightarrow M(B)$ is non-degenerate if $\Span \{\phi(A)\, B \} = \Span \{B \, \phi(A) \} = B$. 

\begin{definition}
A \emph{multiplier bialgebra} is an algebra $A$ equipped with a non-degeneratealgebra homomorphism $\Delta: A\to M(A\otimes A)$ such that
\begin{enumerate}
\item[(i)] $ \Delta(a)(1\otimes b),\ (1\otimes b)\Delta(a),\ \Delta(a)(b\otimes 1),\ (b\otimes 1)\Delta(a) \in A\otimes A$  for all $a,b\in A$.
\item[(ii)] $(\Delta\otimes \mathrm{id})\Delta = (\mathrm{id}\otimes \Delta)\Delta$  in $M(A\otimes A\otimes A)$.
\end{enumerate}
A \emph{multiplier Hopf algebra} is a multiplier bialgebra $(A,\Delta)$ such that the canonical maps $T_1,T_2: A\otimes A\to A\otimes A$ defined by $$ T_1(a\otimes b)=\Delta(a)(1\otimes b) \, \text{ and } \,  T_2(a\otimes b)=(a\otimes 1)\Delta(b)
$$
are bijective. 
\end{definition}

Any Hopf algebra is a multiplier Hopf algebra.

\begin{example}\label{finite support functions}
Let $G$ be a discrete group and denote by $k^G$ the algebra of functions $G \rightarrow k$ with the pointwise operations. We will write $k^{G}_{\fin}$ for the \emph{finitely supported functions.} Note that $k^{G}_{fin} = \bigoplus_{g\in G} k \, \delta_g.$ where $\delta_g(h) = \delta_{g,h}$ for any $g,h \in G.$ Note that $\delta_g . \delta_h = 0$ if $g\neq h$ and equal to $\delta_g$ otherwise. Given $f \in k^G$ one can consider $L_f = R_f : k^{G}_{\fin} \rightarrow k^{G}_{\fin}: g \mapsto f.$ Then $T_f = (L_f,R_f)$ is a multiplier on $k^{G}_{\fin}$. The associated map 
$$k^G \rightarrow M(k^{G}_{\fin}): f \mapsto T_f$$
can be verified to be an isomorphism. Furthermore, $k^{G\times G} \cong M(k^{G}_{\fin} \ot k^{G}_{\fin}).$

The algebra $k^{G}_{\fin}$ is a multiplier Hopf algebra. Concretely, using the aformentioned identification, the coproduct is given by
$$
\D : k^{G}_{\fin} \to M(k^{G}_{\fin} \ot k^{G}_{\fin})
\text{ with } (\D f)(x,y) := f(xy).
$$
Note that on the basis elements $\D (\delta_g) = \sum_{ab=g} \delta_a \ot \delta_b$ which possibily is an infinite sum, i.e. it does not lie in $k^{G}_{\fin} \ot k^{G}_{\fin}$. The counit and antipode are $\eps(f) = f(e)$ and $(Sf)(g) = f(g^{-1})$.

Note that if $G$ is finite, then the sum $\sum_{ab=g} \delta_a \otimes \delta_b$ is finite. Hence in that case $k^G_{\fin}$ is isomorphic to the linear dual $k[G]^*$.
\end{example}

\subsection{RPE solutions from a bialgebra}\label{subsection sol from hopf}
Given a Hopf algebra $H$, following Davydov \cite[Section 3]{Dav} one can associate to each Hopf $H$-module a bijective solution of the PE. Recall that a vector space $M$ is called a \emph{Hopf $H$-module} if it satisfies the following:
\begin{itemize}
    \item $M$ is a left $H$-module, given by $\mu_M: H \ot M \rightarrow M$. We denote $\mu_M(h\ot m)= : hm$.
    \item $M$ is a right $H$-comodule, given by $\D_M: M \rightarrow M \ot H$. We use Sweedler's notation $\D_M(m) = m_{(0)} \ot m_{(1)}$.
    \item Both structures are compatible via $\D_M(hm) = \D_H(h)\D_M(m)$, i.e. $(hm)_{(0)} \ot (hm)_{(1)} = h_{(1)}m_{(0)} \ot h_{(2)} m_{(1)}.$ 
\end{itemize}

Note that the definition of a Hopf module do not require the antipode of $H$ and hence is also defined for bialgebras. When $H$ is a bialgebra, then Davydov's construction still yields a solution, but not necessarily bijective. For convenience of the reader we include a proof of the latter fact.

\begin{proposition}[\cite{Dav}]\label{from bialg to PE}
Let $B$ be a bialgebra and $(M, \mu_M, \D_M)$ a Hopf $B$-module. Define
$$\Phi_M := (I_M \otimes \mu_M) \circ (\D_M \otimes I_M): M\ot M \rightarrow M \ot M .$$ 
Then $(\Phi_M, M)$ is a vector space solution of the RPE.   
\end{proposition}

\begin{remark}\label{sol for hopf multipliier}
Note that by definition if $H$ is a multiplier bialgebra, then $(I \ot m_H)\circ(\Delta_H \ot I)(x \ot y)  \in H \ot H$ for any $x,y\in H$. Furthermore the map $\Phi_H$ is still a RPE solution. More generally, Hopf modules for multiplier Hopf algebras have been introduced in \cite{KvDZ}. In that case the map $\Phi_M$ still makes sense and is a solution of the RPE.
\end{remark}

\begin{remark}\label{fund th hopf mod}
    If $B$ is a Hopf algebra, then $\Phi_M$ is bijective with inverse 
    $$\Psi : M\ot M \rightarrow M \ot M : m \ot n \mapsto m_{(0)} \ot S(m_{(1)})n.$$
    Furthermore, due to the fundamental theorem of Hopf modules, $H \ot M_H \rightarrow M: h\ot m \mapsto hm$ is an isomorphism of Hopf modules, where $$M_H :=\{ m \in M \mid \D_M(m)=m\ot1_H \}$$
    is the subspace of \emph{coinvariants}. The inverse is given by
    $$M \rightarrow H \ot M_H:  m \mapsto m_{(1)}\ot S(m_{(2)})m_{(0)}.$$
    This isomorphism induces an isomorphsim of RPE solutions between $\Phi_M$ and $\Phi_H \ot I_{M_H}$, where $H$ is viewed as Hopf module over itself, see \cite[Corollary 3.4]{Dav}.
\end{remark}

\begin{proof}[Proof of \Cref{from bialg to PE}]
The proof is a direct verification. Note that $\Phi_M(m\ot n) = m_{(0)} \ot m_{(1)}n$ and denote for simplicity $\Phi_M$ by $\Phi$. We need to verify that $\Phi_{12}\Phi_{13}\Phi_{23} = \Phi_{23}\Phi_{12}.$ For $m,n,l \in M$, the left hand side becomes:
$$\begin{array}{lcl}
  \Phi_{12}\Phi_{13}\Phi_{23}(m \ot n \ot l)   & =  & \Phi_{12}\Phi_{13}(m \ot n_{(0)} \ot n_{(1)}l )\\
     & =& \Phi_{12}(m_{(0)} \ot n_{(0)} \ot m_{(1)}n_{(1)}l)\\
     & =& m_{(0)} \ot m_{(1)}n_{(0)} \ot m_{(2)} n_{(1)}l \\
\end{array}$$
And the right hand side:
$$\begin{array}{lcl}
   \Phi_{23}\Phi_{12}(m \ot n \ot l)  & = & \Phi_{23}(m_{(0)} \ot m_{(1)}n \ot l)\\
     & =& m_{(0)} \ot m_{(1)}n_{(0)} \ot m_{(2)}n_{(1)}l
\end{array}$$
where in the last equality we used the compatibility of the structural maps of $M$.
\end{proof}

Suppose that $s \in \End_k(V\ot V)$ is a RPE solution on a finite dimensional vector space $V$. Then one can consider its right coefficients Hopf algebra $H_{r}(s)$. In \cite[Section 5]{Dav}, cf. \Cref{militaru vs davydov}, the space $V$ is equipped with the structure of Hopf module over $H_{r}(s)$. For this structure the following is shown in \cite[Theorem 5.7]{Dav} and implicitily in \cite[Theorem 2.1]{Mi04}.

\begin{theorem}[Davydov, Militaru]\label{reconstruction fd RPE}
    Let $(s,V)$ be a finite dimensional RPE solution. Then $V$ has a Hopf module structure over $H_{r}(s)$ such that $\Phi_V = s$. Therefore, $s \cong \Phi_{H_r(s)} \ot \id_{V_H}$ with $V_H = \{  x \in V \mid s(x \ot y) = x \ot y \text{ for all } y \in V\}.$
\end{theorem}

The `therefore-part' in \Cref{reconstruction fd RPE} follows from \Cref{fund th hopf mod} and the exact definition of the co-multiplication on $V$ constructed in \cite[Section 5]{Dav}. 

Up to the knowledge of the authors, if $(s,V)$ is an infinite dimensional RPE solution, then no alternative for the coefficient Hopf algebras is known. In particular, it is not known whether every infinite solution can be obtained via the construction in \Cref{from bialg to PE}. Therefore we introduce following terminology.

\begin{definition}\label{def reachable sol}
Let $(s,V)$ a RPE solution. Then it is called \emph{reachable} if there exists some Hopf algebra $H$ and Hopf $H$-module $M$ such that $s \cong \Phi_M.$
\end{definition}

Thus by \Cref{reconstruction fd RPE} every finite dimensional solution is reachable. Now consider a set-theoretic solution $(S,s)$. Recall that the associated linearisation is denoted $(s^v,k[S])$. An inconvenient aspect of the isomorphism $\Phi_M \cong \Phi_H \ot \id_{M_H}$ for a Hopf $H$-module $M$ is that it does not behave well with a fixed basis of $M$. In particular if $M = k[S]$, then it seems hard to detect the basis $S$ at the right hand side. As we will see in \Cref{solutions from grp alg}, this subtility hids the interesting fact that $\Phi_H$ for a fixed Hopf algebra $H$ can yield many set-theoretic solutions. More precisely, $H$ can have many bases as in the following definition.

\begin{definition}\label{phi set theoretic def}
Let $H$ be a Hopf algebra and $\Phi_H=(\mathrm{id}\otimes m)(\D \otimes \id)$ the associated RPE solution.
A basis $\mc{B}$ of $H$ is called \emph{$\Phi$-set theoretic} if for all $b,c\in \mc{B}$ the element $\Phi(b\otimes c)$ is a pure tensor in $\mc{B}\otimes \mc{B}$.
\end{definition}

Altough the set-theoretic level is quite subtile, at the vector space level one has following corollary of \Cref{reconstruction fd RPE}.

\begin{corollary}\label{Vs sol versus Hopf alg}
Let $(s_i,S_i)$, for $i=1,2$, be finite set-theoretic solutions of the RPE. Suppose that $|S_2|\geq |S_1|$. Then $H_r(s_1^{v}) \cong H_r(s_2^{v})$ if and only if $s_1^v \cong s_2^v\times 1_W$ for some vector space $W$.
\end{corollary}
\begin{proof}
    Suppose that $H_r(s_1^{V}) \cong H_r(s_2^{V})$, then by \Cref{reconstruction fd RPE} we have that 
    $$s_1^v \cong \Phi_{H_r(s_1^{v})} \ot \id_{k[S_1]} \text{ and } s_2^v \cong \Phi_{H_r(s_2^{v})} \ot \id_{k[S_2]}.$$
    Therefore $s_1^v \cong s_2^v\times 1_{k[X]}$ with $X$ a set of cardinality $|S_2| - |S_1|.$

    Conversely, suppose that  $s_1^v \cong s_2^v\times 1_W$. By \Cref{coefficients PE vs RPE} we have that $H_r(s_2^v\times 1_W) \cong H_r(s_2^v) \ot H_r(1_W)$. Now note that for the trivial solution $1_W$ the generating set in \Cref{leftrightinv} boils down to the single generator $\sum_{y \in W} S_{y,y}$ which is the unit of the right cofficient algebra. Thus $H_r(1_W) \cong k$ and therefore $H_r(s_1^v) \cong H_r(s_2^v\times 1_W) \cong H_r(s_2^v)$, as desired.
\end{proof}

\begin{example}\label{sol from grp alg}
Consider the semigroup algebra $k[S]$ which is a cocommutative bialgebra with coproduct $\D(g) = g \ot g$ for $g \in S$. Denote by $\Phi_S$ the PE solution associated via \Cref{from bialg to PE}. On $g,h \in k[S]$ it is given by:
$$\Phi_{S} (g,h) = (1 \ot m) (g \ot g \ot h)= g \ot gh$$
Hence restricting to the basis $S$ one obtains the set-theoretic solution $s(g,h) = (g,gh)$. If $S$ is a group and $p_2$ denotes the projection of $(g,h)$ on the second copy of $S$, then $s(g,h) = (g,gh)$ is the unique bijective set-theoretic RPE solution on $S$ for which $p_{2} \circ s$ coincides with the group structure of $S$  \cite{MR1637789,CMM19}.  
\end{example}

\begin{example}\label{sol from dual grp alg}
Let $G$ be a group and $k^G_{\fin}$ the finitely support functions on $G$. Its structure as mutliplier Hopf algebras was recalled in \Cref{finite support functions}. In particular, on the basis  $\{ \delta_g \}_{g \in G}$, the coproduct of $k^G_{\fin}$ is given by
$$\D(\delta_g) = \sum_{h \in G} \delta_h \ot \delta_{h^{-1}g}$$
and the co-unit by $\epsilon(\delta_g) = \delta_{1,g}.$
Hence on the canonical basis the associated RPE solution $\Phi$ takes the value $ \Phi(\delta_g \ot \delta_h)  =  \sum_{x} \delta_x \ot \delta_{x^{-1}g}.\delta_h$. Since
 $$
\delta_{x^{-1}g}.\delta_h = \left\lbrace \begin{array}{ll}
            0  & \text{if } x^{-1}g \neq h \\
            \delta_h  & \text{else }
         \end{array} \right.
$$
we obtain that $\Phi(\delta_g \ot \delta_h) = \delta_{gh^{-1}} \ot \delta_h$ preserves the basis $\{ \delta_g \}_{g\in G}$. Identifying sets $\{ \delta_g \}_{g\in G}$ and $\{ g \in G \}$, we see that $\Phi$ yields the set-theoretic map $s : G^{\times 2} \rightarrow G^{\times 2}: (g,h) \mapsto (gh^{-1},h).$ Following \Cref{dual PE group solution}, this the pullback of the unique PE group solution on $G$. \smallskip

In case that $G$ is finite abelian, then the solution obtained from the $k$-linear dual $k[G]^*= k^G_{\fin}$ coincides with the group solution from \Cref{sol from grp alg}, but on the pontryagin dual $G^{\vee}= \hom(G,k^*)$ of $G$, where $k$ is some splitting field of $G$. To see this, we need to recall an explicit Hopf algebra isomorphism, where $\{ e_{\chi} \mid \chi \in G^{\vee} \}$ denotes the basis of $k[G^{\vee}] $: 
\begin{equation}\label{map between two types of dual}
\Psi: k[G^{\vee}] \rightarrow k[G]^*: e_{\chi} \mapsto \sum_{g\in G} \chi(g) \delta_g.
\end{equation}
Its inversion is given via Fourier inversion, i.e. $\Psi^{-1}(\delta_g) = \sum_{\chi \in G^{\vee}}\chi(g^{-1})e_{\chi}.$\smallskip

\noindent {\it Claim:} $\Phi_{G^{\vee}}(e_{\chi_1} \ot e_{\chi_2}) = e_{\chi_1}\ot e_{\chi_1\chi_2}= (\Psi^{-1}\ot \Psi^{-1}) \circ \Phi \circ (\Psi \ot \Psi)$\medskip

We first verify that
$$\left( \Phi \circ (\Psi \ot \Psi) \right) (e_{\chi_1} \ot e_{\chi_2}) = \Phi(\sum_{g,h\in G} \chi_1(g) \chi_2(h) \delta_g \ot \delta_h ) = \sum_{g,h\in G} \chi_1(g) \chi_2(h) \delta_{gh^{-1}} \ot \delta_h.$$
Next we do the change of variables $u= gh^{-1}$ and use that the characters $\chi$ are mulitplicative to rewrite the latest sum as following:
$$\sum_{g,h\in G} \chi_1(g) \chi_2(h) \delta_{gh^{-1}} \ot \delta_h = \sum_{u,h\in G} \chi_1(u)\chi_1(h) \chi_2(h) \delta_{u} \ot \delta_h= \left( \sum_{u\in G} \chi_1(u)\delta_u \right) \ot \left( \sum_{h\in G} (\chi_1\chi_2)(h)\delta_h \right).$$
In other words, we computed that $\left( \Phi \circ (\Psi \ot \Psi) \right) (e_{\chi_1} \ot e_{\chi_2}) = \Psi(e_{\chi_1})\ot \Psi(e_{\chi_1\chi_2})$, which entails the claim.\medskip

An interesting feature of this is example is that we obtained on the vector space $k[G]^*$ two bases, namely $\mc{B}_1 = \{ \delta_g \}_{g\in G}$ and $\mc{B}_2= \{ \Psi(e_{\chi})\}_{\chi \in G^{\vee}}$, such that $\Phi$ preserves both bases. In particular $\restr{\Phi}{\mc{B}_i\ot \mc{B}_i}$ corresponds to set-theoretic solutions of RPE, but which are not isomorphic by \Cref{sol from dual grp alg}.
\end{example}

\begin{example}\label{sol from lie algebra}
 Consider $H = U(\mathfrak{g})$ for some Lie algebra $\mathfrak{g}$ of which we fix a totally ordered $k$-basis $\mathcal{B}$ of $\mathfrak{g}$. Recall that by the Poincar\'e-Birkhoff-Witt theorem the following set is a basis of $U(\mathfrak{g})$:  $$\{ y_1^{n_1}. \cdots . y_{\ell}^{n_{\ell}} \mid y_1 < \cdots < y_{\ell} \in \mathcal{B} \}.$$
Now note that for $x,y \in \mathcal{B}$,
$$\Phi_{\mathfrak{g}} (x,y) =(I \otimes \mu)(\D(x) \otimes y)= x\otimes y + 1 \otimes xy.$$
Hence restricting the solution $\Phi_{\mathfrak{g}}$ to the basis $\mathcal{B}$ does {\it not} yield a set-theoretic solution. We will prove in \Cref{no set for Lie} that there is also no other basis yielding one.
\end{example}

\subsection{Bicrossed product Hopf algebra and matched pair of solutions}\label{section bicrossed}

\subsubsection{Background on bicrossed products}\label{subsecti bicrossed}

Let $H$ and $K$ be bialgebras. Following \cite{Takeuchi} they are said to form a (left-right) matched pair if
\begin{itemize}
\item There is a left action $\triangleright:K\otimes H\to H$ making $H$ into a left $K$-module coalgebra.\smallskip
    \item There is a right co-action $\rho:K \to K \ot H : a \mapsto a_K \ot a_H$ making $K$ into a right $H$-comodule coalgebra.\smallskip
\item Compatibility identities, for all $a,b\in K$ and $h,h'\in H$:
\begin{align*}
a\triangler (hh') &= (a_{(1)}\triangler h)\bigl( a_{(2)}\triangler h'\bigr),\\
\D(a_K) \ot a_H &= (a_{(1)})_K \ot (a_{(2)})_K \ot (a_{(1)})_{H}(a_{(2)})_{H} \\
\rho(ab) & = \rho(a_{(1)})(b_K \ot (a_{(2)}\triangler b_H)) \\
\D_K(a \triangler h) & = (a_{(1)})_K \triangler h_{(1)} \ot (a_{(1)})_H \, (a_{(2)} \triangler h_{(2)})
\end{align*}
and
$a\triangler 1_H=\eps_K(a)1_H$, $1_K\triangler h=h$.
\end{itemize}
The compatibility will ensure that the multiplication below is associative and that the bialgebra axioms hold.

\begin{definition}
Let $H$ and $K$ be Hopf algebras. The (left-right) \emph{bicrossed product} Hopf algebra $H\bowtie K$ is the vector space $H\otimes K$ endowed with the following structural maps.
\begin{itemize}
\item unit: $1_{H\bowtie K}=1_H\ot 1_K$,
\item multiplication:
\begin{equation*}
(h\otimes a)(y\otimes b)
=
h\,(a_{(1)}\triangler y_{(1)})\ \ot a_{(2)}\,b,
\end{equation*}
\item counit: $\eps (h\ot k)=\eps_H(h)\eps_K(k)$,
\item coproduct:
\[
\D(h\ot a)=(h_{(1)}\ot (a_{(1)})_K)\ot (a_{(2)}(a_{(1)})_H\ot a_{(2)}),
\]
\item Antipode: $S(h \ot a) = (1 \ot S_K(a_K)) (S_H(h\,a_H) \ot 1)$
\end{itemize}
\end{definition}

It was proven in \cite{Takeuchi} that $H \bowtie K$ is again a Hopf algebra. To avoid confusion we will write the simple tensors as $h \# k := h \ot k.$

\begin{example}\label{example matched pair groups}
A main protagonist will be with  $H = k[B]^*$ the dual of a group algebra and $K= k[N]$ a group algebra where $B$ and $N$ form a matched pair of groups. Note that for $k[B]^*$ to be a co-algebra \emph{we need that $B$ is finite}. Recall that a matched pair of groups consists of maps
\begin{itemize}
\item  A right action of $N$ on $B$ : $\triangleleft: B\times N\to B$
\item A left action of $B$ on $N$: $\triangleright: B\times N\to N$
\end{itemize}
such that for all $b,b'\in B$ and $u,v\in N$,
\begin{align}
b\triangler (uv) &= (b\triangler u)\bigl((b\trianglel u)\triangler v\bigr),\label{eq:grpMP1}\\
(bb')\trianglel u &= \bigl(b\trianglel (b'\triangler u)\bigr)\,(b'\trianglel u).\label{eq:grpMP2}
\end{align}
and unital conditions $b\triangleright e_N=e_N$, $e_B\triangleleft n=e_B$, $e_B\triangleright n=n$, $b\triangleleft e_N=b$.
If $B$ and $N$ are subgroups of a group $G$, the above data is equivalent to a factorisation $G=BN$. From a matched pair $(B,N,\triangleleft,\triangleright)$ one can construct the following Hopf actions as follows.

\begin{itemize}
\item Left action of $k[B]^*$ on $k[N]$: for $u\in N$ we define 
$$
\delta_b\triangler u := b \, \triangler \,u$$ 
with $b \in B$ and $u \in N$. 

\item Right coaction of $k[N]$ on $k[B]^*$:
$$
\rho(u) = \sum_{b\in B} (b \triangler u) \ot \delta_b \in k[N] \ot k[B]^*
$$
\end{itemize}

A straightforward verification shows that the (co-)multiplication of $k[B]^*\bowtie k[N]$ takes the following form:

\begin{align*}
(\delta_s\# u)\,(\delta_t\# v)
&:=\delta_{\,s\trianglel h,\ t}\;(\delta_s\# uv),\\
\D (\delta_s\# u)
&:=\sum_{xy=s} (\delta_x\# (y\triangler u)) \otimes (\delta_y\# u) = \sum_{x \in B} (\delta_x\# (x^{-1}s\triangler u)) \otimes (\delta_{x^{-1}s}\# u), \\
1_{k[B]^*\bowtie k[N]}&:= 1_{k[B]^*} \# 1_{k[N]}= \sum_{s\in B}\delta_s\# 1_N \\
\eps (\delta_s\# v)&:=\delta_{s,1_B}
\end{align*}
for $s,t \in B$ and $u,v \in N$.
\end{example}

\subsubsection{RPE solution of matched pairs of Hopf algebras}

We now compute the RPE solutions associated via \Cref{from bialg to PE} to the above bicrossed product Hopf algebras. Note that if we take for the Hopf module $M$ a Hopf algebra $H$, then Davydov's constructions can be written as following:
\begin{equation}\label{concrete for Davydov}
    \Phi_H: H \ot H \rightarrow H \ot H : a\otimes b \mapsto a_{(1)}\otimes a_{(2)}b
\end{equation}

\begin{proposition} \label{solutions of hopf mashed pair} 
Let $(H,K, \triangler,\rho)$ a matched pair of Hopf algebras. then 
$$
\Phi_{H\bowtie K}\bigl((h\# a)\otimes (y\# b)\bigr)
=
(h_{(1)}\# (a_{(1)})_K)\otimes
\Bigl(
h_{(2)}(a_{(1)})_H\,(a_{(2)}\triangler y_{(1)})\# a_{(3)}b
\Bigr).
$$

In particular if the matched pair is trivial, i.e. $H \bowtie K \cong H \ot K$, then $\Phi_{H \bowtie K} = \Phi_{H} \times \Phi_{K}$. Moreover if $H$ and $K$ are as in \Cref{example matched pair groups}, then:
$$
\Phi_{k[B]^* \bowtie k[N]}\bigl((\delta_s\# u)\otimes(\delta_t\# v)\bigr)
=
\bigl(\delta_{\,s\,(t\triangleleft u^{-1})^{-1}}\# ((t\triangleleft u^{-1})\triangleright u)\bigr)
\ \otimes\
\bigl(\delta_{\,t\triangleleft u^{-1}}\# hv\bigr).
$$
preserves the basis $\{ \delta_b \ot u \mid b \in B, u \in N \}$.
\end{proposition}

Interpreting the basis element $\delta_b\otimes u$ as the tuple $(b,u)\in B\times N$,
this corresponds to following set-theoretic RPE solution on $B \times N$: 

\begin{equation}\label{set mashed pair}
\phi_{B\bowtie N}\bigl((s,u)\, ,\, (t,v)\bigr)
=
\Bigl( \bigl(s\,(t\triangleleft u^{-1})^{-1}, (t\triangleleft u^{-1})\triangleright u\bigr) \, , \,
(t\triangleleft u^{-1},  uv ) \Bigr)
\end{equation}
\begin{proof}
The expression for $\Phi_{k[B]^* \bowtie k[N]}$ follows directly from the definition and inserting the crossed coproduct and crossed multiplication of $H\bowtie K$, followed by expanding in Sweedler's notation:
\begin{align*}
\Phi_{H\bowtie K}\bigl((h\# a)\otimes (y\# b)\bigr)
 & =
(h_{(1)}\# (a_{(1)})_K)\otimes
\Bigl(
h_{(2)}(a_{(1)})_H\,(a_{(2)(1)}\triangler y_{(1)})\# a_{(2)(2)}b
\Bigr) \\
 & =
(h_{(1)}\# (a_{(1)})_K)\otimes
\Bigl(
h_{(2)}(a_{(1)})_H\,(a_{(2)}\triangler y_{(1)})\# a_{(3)}b
\Bigr) 
\end{align*}
From the obtained expression we see that if the both action are trivial, then we obtain the direct product of the solutions.

Now suppose that $H = k[B]^*$ and $K= k[N]$ with $(B,K,\trianglel, \triangler)$ a matched pair of groups. In that case the above expression can be further simplified using the form of the structural maps written in \Cref{example matched pair groups}. Filling in the above formula yields

$$
\Phi_{k[B]^* \bowtie k[N]}( (\delta_s\# u) \ot (\delta_t\# v))
=\sum_{xy=s} (\delta_x\# (y\triangleright u))\ \otimes\ \bigl((\delta_y\# u)(\delta_t\# v)\bigr).
$$
Now use that $ (\delta_y\# u)(\delta_t\# v)=\delta_{y\triangleleft u, t}\;(\delta_y\# uv)$. Hence only the unique term with $y\triangleleft u=t$, i.e. $y = t\triangleleft u^{-1}$, survives. For that $y$ we have that $x=s\,y^{-1} = s\,(t\triangleleft u^{-1})^{-1}.$ Therefore, substituting yields
\[
\Phi_{k[B]^* \bowtie k[N]}( (\delta_s\# u) \ot (\delta_t\# v))
=
(\delta_{\,s\,(t\triangleleft u^{-1})^{-1}}\# ((t\triangleleft u^{-1})\triangleright u))
\ \otimes\
(\delta_{\,t\triangleleft u^{-1}}\# uv).
\]
\end{proof}

\subsubsection{A formal notion of mashed pair of RPE solutions}
At this point we have seen how the RPE solutions $\Phi_H$ and $\Phi_K$ in \eqref{concrete for Davydov} and the solution 
$$(h\# a)\otimes (y\# b)\bigr)
\mapsto 
(h_{(1)}\# (a_{(1)})_K)\otimes
\Bigl(
h_{(2)}(a_{(1)})_H\,(a_{(2)}\triangler y_{(1)})\# a_{(3)}b
\Bigr)$$
from \Cref{solutions of hopf mashed pair}  are related at a Hopf theoretical level. We now introduce a formal setting that connects the maps $\Phi_H$ and $\Phi_K$ through an operation on solutions. 

\begin{definition}\label{def:mp_davydov}
Let $H$ and $K$ be bialgebras, and let 
\[
\Phi_H=(\id\otimes m_H)(\Delta_H\otimes\id),
\qquad
\Phi_K=(\id\otimes m_K)(\Delta_K\otimes\id)
\]
be their associated RPE solutions. A (left--right)\emph{ matched-pair datum of RPE solutions} consists of two linear maps
\[
\Xi_{\mathrm{mult}}:K\otimes H\to H\otimes K,
\qquad
\Xi_{\mathrm{cop}}:H\otimes K\to K\otimes H,
\]
satisfying following equations: 
\begin{enumerate}
\item Multiplicative mixed pentagon:
as maps $K\otimes H\otimes H\to H\otimes H\otimes K$,
\[
(\Phi_H\otimes \id_K)
(\id_H\otimes \Xi_{\mathrm{mult}})
(\Xi_{\mathrm{mult}}\otimes \id_H)
=
(\id_H\otimes \Xi_{\mathrm{mult}})
(\Xi_{\mathrm{mult}}\otimes \id_H)
(\id_K\otimes \Phi_H).
\]

\item Comultiplicative mixed pentagon:
as maps $H\otimes K\otimes K\to K\otimes H\otimes H$,
\[
(\id_K\otimes \Phi_H)
(\Xi_{\mathrm{cop}}\otimes \id_H)
(\id_H\otimes \Xi_{\mathrm{cop}})
=
(\Xi_{\mathrm{cop}}\otimes \id_H)
(\id_H\otimes \Xi_{\mathrm{cop}})
(\Phi_H\otimes \id_K).
\]

\item Action--coaction compatibility: 
as maps $K\otimes K\otimes H\to K\otimes H\otimes K$,
\[
(\Xi_{\mathrm{cop}}\otimes \id_K)
(\id_H\otimes m_K)
(\Xi_{\mathrm{mult}}\otimes \id_K)
=
(\id_K\otimes m_H)
(\id_K\otimes \Xi_{\mathrm{mult}})
(\Xi_{\mathrm{cop}}\otimes \id_H).
\]
\end{enumerate}
\end{definition}

The maps $\Xi_{\mathrm{mult}}$ and $\Xi_{\mathrm{cop}}$ will be called respectively the \emph{multiplicative} and \emph{comultiplicative interfaces}. The following can be verified with a direct computation

\begin{proposition}\label{thm:matched_construction}
Let $(H,K,\Xi_{\mathrm{mult}},\Xi_{\mathrm{cop}})$ be a matched-pair datum of RPE solutions. Define on $H \ot K$ the maps 
\begin{align*}
\Delta_{\bowtie }
&:=
(\id_H\otimes \Xi_{\mathrm{cop}}\otimes \id_K)
\circ
(\Delta_H\otimes \Delta_K),\\[2mm]
m_{\bowtie }
&:=
(m_H\otimes m_K)
\circ
(\id_H\otimes \Xi_{\mathrm{mult}}\otimes \id_K).
\end{align*}
Then 
\[
\Phi_{\bowtie }
:=
(\id\otimes m_{H\bowtie K})
(\Delta_{H\bowtie K}\otimes \id)
\]
is a solution of the RPE.
\end{proposition}

Note that if we one would write 
$$
\Xi_{\mathrm{cop}}(h_{(2)}\otimes a_{(1)})=a^{(0)}\otimes h_{(2)}^{\sharp} \,
\text{ and }\,
\Xi_{\mathrm{mult}}(a_{(2)}\otimes y)=y^{\flat}\otimes a'
$$
then $\Phi_{\bowtie}$ takes following form:
\begin{equation}\label{form mashed pair sol formal}
\Phi_{\bowtie}\bigl((h\otimes a)\otimes (y\otimes b)\bigr)
=
(h_{(1)}\otimes a^{(0)})
\otimes
\bigl(
(h_{(2)}^{\sharp}y^{\flat})\otimes (a'b)
\bigr).
\end{equation}

\begin{example}
If $(H,K, \triangler, \rho)$ is mashed pair of bialgebras, then one considers
\[
\Xi_{\mathrm{mult}}(a\otimes h)=(a_{(1)}\triangleright h_{(1)})\otimes a_{(2)}
\text{ and }
\Xi_{\mathrm{cop}}(h\otimes a)=a_K\otimes h\,a_H.
\]
In that case, the relations from \Cref{def:mp_davydov} correspond to the matched-pair axioms. Moreover the formula \eqref{form mashed pair sol formal} for $\Phi_{\bowtie}$ exactly becomes the solution $\Phi_{H \bowtie K}$ from \Cref{solutions of hopf mashed pair}.
\end{example}

\section{On relation between positive basis property and set-theoretic solutions}\label{section positive basis}

\noindent {\it Convention:} In this section $k$ will any field of characteristic $0$.\smallskip

For such fields one can speak about positivity of scalars, e.g by viewing $k$ as subfield of $\C$.  Recall, cf \cite{zbMATH01616308}, that a Hopf algebra $H$ is said to have {\it the positive basis property} if it has a basis $B$ for which the structure constants of all structure maps are positive, i.e.:
\begin{enumerate}
    \item the coordinates of the unit $1$ are nonnegative;
    \item the coordinates of the counit $\epsilon$ (with respect to the dual basis $B^{\ast}$) are nonnegative, i.e. $\epsilon(b)\geq 0$, for any $b\in B$;
    \item for any $b_1, b_2\in B$, the coordinates of $b_1 b_2$ are nonnegative;
    \item for any $b\in B$, the coordinates of $\D(b)$ with respect to the tensor product basis $B \otimes B$  are nonnegative;
    \item for any $b\in B$, the coordinates of the antipode $S(b)$ are nonnegative.
\end{enumerate}

If all structural constants are positive except those for the antipode, then $H$ is said to have the \emph{nearly positive basis property}.

In this section we study the relation between Hopf algebras with the positive basis property and set-theoretic solutions. In the finite dimensional setting we show in \Cref{pos basis prop theorem} that such Hopf algebras correspond to finite bijection solutions of the RPE. 

We start with \Cref{section generalitis exotic basis} where we show that any hopf algebra with a $\Phi$-set theoretic basis is close to be nearly positive.

\subsection{Generalities on $\Phi$-set theoretic bases}\label{section generalitis exotic basis}

Now suppose that $\mc{B}$ is a $\Phi$-set theoretic basis of some Hopf algebra H. In particular $\restr{\Phi}{\mc{B} \ot \mc{B}}$ yields a set-theoretic solution of the RPE. Thus it can be written in the form
\begin{equation}\label{eq:psi-circ-B}
\phi_{\mc{B}}(b,c)=\bigl(\psi_c(b),\ c\circ b\bigr).
\end{equation}
where for each $c\in \mc{B}$ the map $\psi_c:\mc{B}\to \mc{B}$ is a permutation and $\circ:\mc{B}\times \mc{B}\to \mc{B}$
is a binary operation. 
Thus, for all $b,c\in B$,
\begin{equation}\label{eq:Phi-psi-circ}
\Phi_H(b\otimes c)=\psi_c(b)\otimes c\circ b.
\end{equation}

\begin{lemma}\label{lem:right-monomial}
Let $\mc{B}$ be a $\Phi$-set theoretic basis of a Hopf algebra $H$.
Then for all $b,c\in \mc{B}$ one has the following:
\begin{enumerate}
    \item[(i)]\label{bc and circ} $bc = \eps(\psi_c(b))\,(c\circ b)$,
    \item[(ii)]\label{b and psi} $\eps(c)b= \eps(c \circ b)\psi_c(b)$,
    \item[(iii)] $\D(b) = \sum_{c \in \mc{B}} \lambda_c \, \eps(c)\eps(c \circ b)^{-1}b \ot c \circ b$ with $1_H = \sum_{c \in \mc{B}} \lambda_c c$ and if $\eps(c\circ b) \neq 0$ whenever $\lambda_c \neq0$,
    \item[(iv)]\label{expression antipode general} $S(\psi_c(b))\,(c\circ b) = \eps(b) c$ and $\eps(c) S(b) (c\circ b) = \eps(c\circ b) \eps( b) c.$. 
\end{enumerate}
\end{lemma}

The expression for the (co-)multiplication in \Cref{lem:right-monomial} directly imply the following.

\begin{corollary}\label{nearly positive}
Let $H$ be a Hopf algebra and $\mc{B}$ a $\Phi$-set theoretic basis of $H$. Suppose $\mc{B}$ satisfies positivity with respect to the unit and co-unit, i.e. $\eps(b) \geq 0$ for all $b \in \mc{B}$ and $\lambda_c \geq 0$ in the decomposition $1_H = \sum_{c \in \mc{B}} \lambda_c c$. Then $H$ has the nearly positive basis property.
\end{corollary}

It would be interesting to know whether under the conditions of \Cref{nearly positive} the expression \eqref{expression antipode general} also yield positivity for the antipode.

\begin{example}
\Cref{lem:right-monomial} implies that $\Phi_H(b\ot c) = \lambda_1 \lambda_2 b \ot bc $ for some scalars $\lambda_i \in k.$ Hence it is tempting to conclude that $\Phi_H$ is equivalent to the group solution. This is however not the case. As good illustration thereof consider $H = k[G]^*$ with $G$ a finite group and the associated basis $\{ \delta_g \mid g \in G\}.$ In \Cref{sol from dual grp alg} it was verified that $\Phi_{k[G]}(\delta_g \ot \delta_h) = \delta_{gh^{-1}} \ot \delta_h.$ 

The identity \eqref{b and psi} takes the form $\delta_g = \delta_{h,1}^{-1} \delta_{h,1} \delta_{gh^{-1}} = \delta_{h,1} \delta_{gh^{-1}}$. And the equation \eqref{bc and circ} is $\delta_{g,h}\delta_{h} = \delta_{gh^{-1},1} \delta_h$. Both statements are indeed correct and illustrate that the scalars $\lambda_i$ can rewrite expression in a non-expected way.
\end{example}

\begin{proof}[Proof of \Cref{lem:right-monomial}]
We start with the equation $bc = \eps(\psi_c(b))\,(c\circ b)$. For this apply $(\eps\otimes \id)$ to \eqref{eq:Phi-psi-circ}. On the one hand,
\begin{align*}
(\eps\ot \id)\Phi_H(b\ot c)
 & = (\eps\ot \id)(\id\ot m)(\D(b)\ot c) \\
 & = \eps(b_{(1)}) b_{(2)}c \\
&= m\bigl((\eps\ot \id)\D(b)\ot c\bigr)\\
&= bc,
\end{align*} 
using $(\eps\ot \id)\D=\id$.
On the other hand, $(\eps\ot \id)(\psi_c(b)\ot (c\circ b))=\eps(\psi_c(b))\,(c\circ b).$

Now we prove that $b= \eps(c)^{-1}\eps(c \circ b)\psi_c(b)$ in an analogue way. Namely compute $(\id \ot \eps)\Phi_H(b\ot c) = \eps(b_{(2)})\eps(c)b_{(1)}= \eps(c) b$ using that $(\id \ot \eps)\D = \id.$ On the other hand, $(\id \ot \eps)\Phi_H(b\ot c) = \eps(c \circ b) \psi_c(b)$, yielding the equation. \medskip

Next write $1_H = \sum_{c \in \mc{B}} \lambda_c c$. Then 
$$\D(b) = \Phi_H(b\ot 1_H) = \sum_{c \in \mc{B}} \lambda_c \,\psi_c(b)\otimes c\circ b = \sum_{c \in \mc{B}} \lambda_c \, \eps(c)\eps(c \circ b)^{-1}b \ot c \circ b.$$

Finally we consider the antipode. For this note that 
$$m(S \ot 1_H)\Phi_H(b \ot c) = S(b_{(1)})b_{(2)}c = \eps(b)c,$$ 
using the definition of $\Phi_H$ and the convolution rule $m(S\ot id)\D= \eps$. Now by \eqref{eq:Phi-psi-circ} and the second item, we also have $m(S \ot 1_H)\Phi_H(b \ot c)= S(\psi_c(b))\, (c \circ b).$ Therefore if we compare and use the second item,
$$\eps(c) S(b) (c\circ b) =  \eps(c\circ b) S(\psi_c(b)) (c\circ b) = \eps(c\circ b) \eps( b) c.$$
\end{proof}

\subsection{Finite dimensional Hopf algebras have positive basis}

\subsubsection{Main result and consequences}

The first main result of this section is the following.
\begin{theorem}\label{pos basis prop theorem}
Let $(S,s)$ be a finite bijective solution to RPE. Then $H_{\ell}(s)$ and $H_{r}(s)$ have a basis which is both positive and $\Phi$-set theoretic.
\end{theorem}

The basis alluded to in \Cref{pos basis prop theorem} will be refinement of the generating set in \Cref{leftrightinv}. 

Recall that  \Cref{reconstruction fd RPE} gives a correspondence between finite dimensional Hopf algebras and coefficient algebras of RPE solutions. Therefore we obtain as a consequence of \Cref{pos basis prop theorem} that having the positive basis property and having a $\Phi$-set theoretic basis are equivalent in the finite dimensional setting.

In \cite{zbMATH01616308} finite dimensional Hopf algebras over the complexes $\C$ with the positive basis property have been classified. It was namely shown that they are all isomorphic to a bicrossed product Hopf algebra $k[B]^* \bowtie k[N]$ for some smashed pair of finite groups $(B,N, \trianglel,\triangler)$. More precisely, they show that given a positive basis $\mc{B}$ of $H$, that one can rescale the basis elements so that $\mc{B}$ corresponds to a group $G$ with a factorisation $G = B.N$, uniquely determined by $H$. The existence of such rescaling, obtained in \cite[Section 4]{zbMATH01616308}, however really makes use of the complex numbers in order to be able to work analytically on the spaces of bialgebra structures on a fixed vector space. Now, using the above together with the classification of finite bijective set-theoretic solutions \cite{COvA} and \Cref{Hopf op grp and dual solution} we obtain that their classification also holds over arbitrary fields of characteristic $0$. 

In summary, we have following generalization of \cite[Theorem 1]{zbMATH01616308}.

\begin{corollary}\label{coro set basis iff positive}
    Let $H$ be a finite dimensional Hopf $k$-algebra with $\Char(k)=0$. Then the following are equivalent:
    \begin{itemize}
    \item $H$ has a $\Phi$-set theoretic basis,
    \item $H$ has the positive basis property,
    \item $H \cong k[B]^* \bowtie k[N]$ for a mashed pair of finite groups $(B,N, \trianglel,\triangler)$.
    \end{itemize}
    Moreover, if above holds, then the mashed pair is uniquely determined.
\end{corollary}

Combined with the content of \Cref{section generalitis exotic basis} we expect that the existence of a $\Phi$-set theoretic basis always imply the existence of a (potentially different) positive basis. However we do not expect the converse, because there exists examples of nearly positive Hopf algebras without a $\Phi$-set theoretic basis, such as the universal enveloping $U(\mathfrak{g})$ of a finite dimensional Lie algebra $\mathfrak{g}$, see \Cref{no set for Lie}. 

\begin{conjecture}\label{conjecture on positive basis}
Let $H$ be a Hopf $k$-algebra with $\Char(k)=0.$ If $H$ has a $\Phi$-set theoretic basis, then $H$ has the positive basis property.
\end{conjecture}

\subsubsection{Recollection of some structural results}

To obtain a precise basis of the coefficient algebras, we will use some recent result by Colazzo-Okninski-Van Antwerpen \cite[Proposition 2.4]{COvA} and Colazzo-Jespers-Kubat \cite[Lemma 2.4]{CJK}. 

\begin{proposition}\label{prop:structureSs} 
    Let $(S,s)$ be a finite bijective solution to the RPE. Then $(S, \circ)$ is a left group, i.e. there exist a set $E$ and a group $G$ such that $S=E\times G$ and 
    $$(e,g)\circ (f,h) = (e,g\circ h) \qquad \forall e,f\in E, g,h\in G.$$
    Moreover, for every $x\in S$, the map $\psi_x$ is bijective and  either $\psi_x=\id_S$ or $\psi_x$ is fixed-point free.
    Finally the set $$\Psi=\left\{\left. \psi_x \right| x\in S\right\} = \left\{\left. \psi_{(e,1)} \right| e\in E \right\}$$ is a group (under composition).
\end{proposition}
\begin{proof}
    Applying \cite[Proposition 2.4]{COvA} to the set-theoretic solution of the PE given by $t=\tau s \tau$, we obtain that $(S,\circ)$ is a left group.

    Next, by \cite[Proposition 2.10]{COvA} applied to $(S,t)$, the map $\psi_x$ is bijective for every $x\in S$. Combining \cite[Proposition 2.11]{COvA} with \cite[Lemma~2.4]{CJK}, we conclude that for each $x\in S$ either $\psi_x=\id_S$ or $\psi_x$ is fixed-point free.

    Finally, using \eqref{eq3} and the bijectivity of $s$, one checks that $\Psi$ is closed under composition. Moreover, $\Psi$ is finite, contains the identity (see \cite[Proposition 2.11]{COvA}), and each element is invertible; hence $\Psi$ is a group. The equality
    $$\Psi=\left\{\left. \psi_x \right| x\in S\right\} = \left\{\left. \psi_{(e,1)} \right| e\in E\right\}$$
    follows from \cite[Proposition 2.11]{COvA}.
\end{proof}

\subsubsection{Construction of the basis}

From now on, since in a left group $E\times G$ the set $E$ coincides with the set of idempotents $E(S,\circ)$, we will identify $e\in E$ with $(e,1)\in S$ and simply write $e$ for $(e,1)$.
Since $S$ is a left group, the group $G$ is isomorphic to $e\circ S$ for every $e \in E(S,\circ)$. As a consequence of the previous proposition, there exists an idempotent, which we denote by $1$, such that $\psi_1=\id_S$. In particular, we identify $G$ with $1\circ S$.

We use \Cref{prop:structureSs} to give a nicer description of the sets $H_{\ell}(s)$ and $H_r(s)$ from \Cref{leftrightinv}.

\begin{corollary}\label{the sets nice via left grp result}
    Let $(S, s)$ be a finite bijective solution to the RPE and $x',y' \in S$. Then:
    $$\nu(x',y')=\left\{y\in S \mid \psi_y=\psi_{y'}\right\} \qquad\text{and}\qquad
    \mu(x',y')=\left\{e\circ x' \mid e\in E(S,\circ)\right\}.$$
\end{corollary}
\begin{proof}
    Let $x',y'\in S$ and let $y\in \nu(x',y')$. Then $\psi_{y'}(x')=\psi_y(x')$, and since $\psi_y$ is bijective we obtain $\psi_y^{-1}\psi_{y'}(x') = x'$. 
    Since $\Psi=\{\psi_x\mid x\in S\}$ is a group, there exists $t\in S$ such that $\psi_y^{-1}\psi_{y'}=\psi_t$. Hence $\psi_t(x')=x'$, and by \cref{prop:structureSs} it follows that $\psi_t=\id_S$. 
    Therefore $\psi_y=\psi_{y'}$, $\nu(x',y') = \left\{ y \in S\left| \psi_y = \psi_{y'}\right.\right\}$.
    
   Now let $x\in \mu(x',y')$. Then $y'\circ x=y'\circ x'$. Since $(S,\circ)$ is a left group, we may write $S=E\times G$. Write $x=(e_x,g_x)$, $x'=(e_{x'},g_{x'})$, and $y'=(e_{y'},g_{y'})$. Then 
    $y'\circ x=(e_{y'},g_{y'})\circ (e_x,g_x)=(e_{y'},\,g_{y'}g_x)$
    and 
    $y'\circ x' = (e_{y'}, g_{y'}\circ g_{x'})$. Let $e=(e_x,1)\in E(S,\circ)$. Then $e\circ x' = x$. Therefore $x \in \left\{e\circ x' \mid e\in E(S,\circ)\right\}$. The reverse inclusion is immediate from
    $y'\circ (e\circ x')=y'\circ x'$, for all $e\in E(S,\circ)$.
    Hence $\mu(x',y') = \left\{e \circ x'\mid e \in E(S,\circ) \right\}$.
\end{proof}

Note that, by \Cref{the sets nice via left grp result}, the dependence on $x'$ in $\nu(x',y')$ disappears (and similarly the dependence on $y'$ in $\mu(x',y')$ disappears). Therefore, we introduce the notation
$$
\nu_{\psi}(y') = \left\{\, y\in S \mid \psi_y = \psi_{y'} \right\},
\qquad
\mu_{\circ}(x') = \left \{e \circ x' \mid e \in E(S,\circ) \right\}.
$$

Let $(S,s)$ be a finite bijective set-theoretic solution to the RPE. Now by \cite[Proposition 3.4]{COvA} applied to $t = \tau s \tau$,we obtain an equivalence relation $\sim$ called \emph{retraction} on $(S,\circ)$ defined by: $(e,g) \sim (f,h)$ if and only if $\psi_e = \psi_f$ and $g=h$. This clearly restricts to a congruence on the left zero semigroup $E(S,\circ)$ and simplify as $e\sim f$ if and if $\psi_e = \psi_f$.

Let $\bar{E}$ be a set of representatives of $E(S,\circ)/\sim$. Moreover, fix $1\in \bar{E}$ such that
$\psi_{1}=\id_S$, and identify $G$ with $1\circ S\subseteq S$.

We now show that the $\sim$-classes (equivalently, the retract classes) all have the same cardinality.

\begin{lemma}\label{cardinality_of_fibers}
    Let $(S,s)$ be a finite bijective set-theoretic solution of the RPE, written as
    $s(x,y)=(\psi_y(x),\,y\circ x)$. Then the equivalence classes with respect to the retract relation all have the same cardinality.
\end{lemma}
\begin{proof}
    For $g\in G=1\circ S$ and $e\in E$, we have
    $$[e\circ g]=\{\,f\circ g \mid f\in E,\ \psi_f=\psi_e\,\},$$
    and therefore
    $$|[e\circ g]|=\left|\{\,f\in E \mid \psi_f=\psi_e\,\}\right|=|\nu_{\psi}(e)\cap E|.$$
    Thus it suffices to show that $|\nu_{\psi}(e)\cap E|$ is independent of $e\in E$.
    
    Let $e,f\in E$. By \Cref{prop:structureSs}, the set $\Psi=\{\psi_x\mid x\in E\}$ is a group under composition. Hence
    $\psi_e^{-1}\psi_f\in \Psi$, so there exists $a\in E$ such that
    $$\psi_a=\psi_e^{-1}\psi_f.$$
    Define $\varphi:E\to E$ by $\varphi(b)=\psi_a(b)$.
    We first check that $\varphi$ is well defined. If $b\in E$, then $b=b\circ b$, and using \eqref{eq3} we obtain
    $$\psi_a(b)=\psi_a(b\circ b)=\psi_a(b)\circ \psi_{a\circ b}(b).$$
    Since $a, b\in E$, we have $a\circ b=a$, hence $\psi_{a\circ b}=\psi_a$, and so
    $\psi_a(b)=\psi_a(b)\circ \psi_a(b)$, which shows that $\psi_a(b)\in E$. Thus $\varphi(b)\in E$ and $\varphi$ is well defined. Moreover, $\varphi$ is bijective, with inverse $b\mapsto \psi_a^{-1}(b)$.
    
    Now let $b\in \nu_{\psi}(f)\cap E$, i.e. $\psi_b=\psi_f$. Using \eqref{eq3} and the fact that $a\circ b=a$, we get
    $\psi_{\psi_a(b)}\psi_a= \psi_{\psi_a(b)}\psi_{a\circ b}=\psi_b$.
    Hence
    $\psi_{\psi_a(b)}=\psi_b\psi_a^{-1}
    =\psi_f(\psi_e^{-1}\psi_f)^{-1}
    =\psi_f\psi_f^{-1}\psi_e
    =\psi_e$,
    so $\psi_{\varphi(b)}=\psi_e$, i.e.\ $\varphi(b)\in \nu_{\psi}(e)\cap E$.
    
    Therefore $\varphi$ restricts to a bijection from $\nu_{\psi}(f)\cap E$
    and $ \nu_{\psi}(e)\cap E$,
    and in particular $|\nu_{\psi}(e)\cap E|=|\nu_{\psi}(f)\cap E|$ for all $e,f\in E$. This proves that all retract classes have the same cardinality.
\end{proof}

With this at hand we can now describe a basis of the coefficient algebras.

\begin{theorem}\label{Th basis right coeff}
    Let $(S,s)$ be a finite bijective set-theoretic solution of the RPE. Then
    \begin{align*}
        \left\{\left.\sum_{y\in \nu_{\psi}(y')}S_{y\circ x', y} \ \right|\ y' \in \bar{E},\ x'\in G\right\}.
    \end{align*}
is a $k$-basis of $H_r(s)$, and 
    \begin{align*}
        \left\{\left. \sum_{x\in \mu_{\circ}(x')}S_{\psi_{y'}(x),x} \right|\ y' \in \bar{E},\ x'\in G\right\}
    \end{align*}
is a $k$-basis of $H_{\ell}(s)$.
\end{theorem}

To prove \Cref{Th basis right coeff} we need to understand first the associated ring of co-invariants. For a general algebra solution $(A,R)$ this is:
$$A^{co-inv} = \left\{a \in A\ \left|\ (1\otimes a) R = 1\otimes a\right.\right\}.$$

\begin{lemma}\label{lemma:co-invariants}
Let $(S,s)$ be a finite bijective set-theoretic solution of the RPE. Then the algebra of left co-invariant is given by 
$$A^{co-inv}=\operatorname{span}_{k}\left\{\left.\sum_{l \in d\circ S}S_{x,l}\ \right|\ x \in S, \, d \in \nu_{\psi}(1)\cap E\right\}.$$

In particular,
$$\dim_k A^{co-inv}\le |S| \frac{|E|}{\bar{E}}.$$
\end{lemma}
\begin{proof}
Write $a=\sum_{x,y\in S}\alpha_{x,y}S_{x,y}$. Then
\begin{align*}
    (1\otimes a)R
    &=\left(\sum_{u\in S}S_{u,u}\otimes \sum_{x,y\in S}\alpha_{x,y}S_{x,y}\right)
    \left(\sum_{c,d\in S}S_{\psi_d(c),c}\otimes S_{d\circ c,d}\right)\\
    &=\sum_{u,c,d,x,y\in S}\alpha_{x,y}\, S_{u,u}S_{\psi_d(c),c}\otimes S_{x,y}S_{d\circ c,d}\\
    &=\sum_{c,d,x\in S}\alpha_{x,d\circ c}\, S_{\psi_d(c),c}\otimes S_{x,d}.
\end{align*}
Indeed, $S_{u,u}S_{\psi_d(c),c}=S_{\psi_d(c),c}$ if $u=\psi_d(c)$ and $0$ otherwise, while
$S_{x,y}S_{d\circ c,d}=S_{x,d}$ if $y=d\circ c$ and $0$ otherwise.
Moreover, 
\begin{align*}
    1\otimes a=\left(\sum_{c\in S}S_{c,c}\right)\otimes \left(\sum_{x,d\in S}\alpha_{x,d}S_{x,d}\right)
    =\sum_{c,x,d\in S}\alpha_{x,d}\, S_{c,c}\otimes S_{x,d}.
\end{align*}
Since the elements $\{S_{p,q}\otimes S_{r,s}\}_{p,q,r,s\in S}$ are linearly independent, we compare coefficients. A term $S_{\psi_d(c),c}\otimes S_{x,d}$ can match a term $S_{c',c'}\otimes S_{x',d'}$ if and only if
$c=c'$, $x=x'$, $d=d'$ and $\psi_d(c)= c$. 
By \Cref{prop:structureSs}, for each $d\in S$ either $\psi_d=\id_S$ or $\psi_d$ is fixed-point free; hence
$\psi_d(c)=c$ forces $\psi_d=\id_S$. Therefore, if $\psi_d\neq\id_S$ then all coefficients
$\alpha_{x,d\circ c}$ must vanish (otherwise $(1\otimes a)R$ would contain terms not appearing in $1\otimes a$).
Thus $a$ is a $k$-linear combination of elements of the form
$$\sum_{l\in d\circ S} S_{x,l}
\qquad\text{with}\qquad d\in \nu_{\psi}(1)=\left\{d\in S\mid \psi_d=\id_S\right\}.$$
with $\alpha_{x,d\circ c}\neq 0$ if and only if $\psi_d =\operatorname{id}$, i.e. if and only if $d\in \nu_{\psi}(1)$. 
This yields
\begin{align*}
    A^{\mathrm{co\mbox{-}inv}}
    =\operatorname{span}_{k}\left\{
    \left.\sum_{l \in d\circ S} S_{x,l}\ \right|\ x \in S,\ d \in \nu_{\psi}(1)
    \right\}.
\end{align*}

Since $S=E\times G$ is a left group, if $d=(e,g)$ then $d\circ S=\{e\}\times G$. Let $E_0=\left\{\left. (e,1) \in E(S,\circ) \right| \exists g \in G \colon \psi_{(e,g)}=\id_S\right\}$. So 
\begin{align*}
    A^{\mathrm{co\mbox{-}inv}}
    =\operatorname{span}_{k}\left\{
    \left.\sum_{l \in d\circ S} S_{x,l}\ \right|\ x \in S,\ d \in E_0
    \right\}.
\end{align*}
We claim that $E_0=\nu_{\psi}(1)\cap E=\left\{\left.e \in E(S,\circ)\right|\psi_{e}=\id\right\}$. Clearly, $E_0$ contains $\nu_{\psi}(1)\cap E$ (indeed, if $(e,1)\in E(S,\circ)=E\times \{1\}$, then $\psi_(e,1)=\id_S$, then $e\in E_0$ by choosing $g=1$). Conversely, let $e\in E$ and assume that there exists $g\in G$ such that $\psi_{(e,g)}=\id_S$.
We want to prove that $\psi_{(e,1)}=\id_S$. By \eqref{eq3} applied to $z=(e,g)$ and $y=(e,g^{-1})$, we have $\psi_{\psi_{(e,g)}(e,g^{-1})}\,\psi_{(e,g)\circ (e,g^{-1})}=\psi_{(e,g^{-1})}$. Since $\psi_{(e,g)}=\id_S$, this becomes $\psi_{(e,g^{-1})}\,\psi_{(e,g)\circ (e,g^{-1})}=\psi_{(e,g^{-1})}$. Now, since $(e,g)\circ (e,g^{-1})=(e,1)$ we have $\psi_{(e,g^{-1})}\,\psi_{(e,1)}=\psi_{(e,g^{-1})}$. Finally, as $\psi_{(e,g^{-1})}$ is bijective, we have $\psi_{(e,1)}=\id$ as required. 

It follows that $\dim_k A^{co-inv}\leq |S||\nu_\psi(1)\cap E|$. Now, note that by \Cref{cardinality_of_fibers} (applied to the restriction of $\sim$ to $E$), all $\sim$-classes in $E$ have the same cardinality. Hence
$$|E|
=\sum_{[e]\in \bar{E}} |[e]\cap E|
=\sum_{[e]\in \bar{E}} |[1]\cap E|
=|\bar{E}|\cdot |[1]\cap E|
=|\bar E|\cdot |\nu_\psi(1)\cap E|.$$
Hence, $\dim_k A^{co-inv} \leq |S|\frac{|E|}{\bar{E}}$, as required. 
\end{proof}

Now, 
\begin{proof}[Proof of \Cref{Th basis right coeff}]

We know that $A^{co-inv}\otimes H_r(s)$ is isomorphic, as an $H_r(s)$-module, to $A$. Respectively, $H_{\ell}(s) \otimes A^{co-inv}$ is isomorphic, as $H_{\ell}$-module, $A$. Hence
$$\dim_k H_r(s)=\frac{\dim_k A}{\dim_k A^{co-inv}}
=\frac{|S|^2}{\dim_k A^{\mathrm{co-inv}}}.$$ 
By \Cref{lemma:co-invariants}, we have $\dim_k A^{co-inv}\le |S|\frac{|E|}{\bar{E}}$. Hence, 
$$\dim_k H_r(s)= \dfrac{|S|^2|\bar{E}|}{|S||E|}\leq |G||\bar{E}|$$

Furthermore, we have the upper bound
$$\dim_k H_r(s) \leq |G||\bar{E}|,$$ since by \Cref{the sets nice via left grp result} we have that the set $\left\{\left.\sum_{y\in \nu_{\psi}(y')}S_{y\circ x', y} \ \right|\ y' \in \bar{E},\ x'\in G\right\}$ is a generating set. This yields that a basis for $H_r(s)$ is given by $\left\{\left.\sum_{y\in \nu_{\psi}(y')}S_{y\circ x', y} \ \right|\ y' \in \bar{E},\ x'\in G\right\}$.

By \cite[Theorem 2.1]{Mi04} we know that $R_{(\ell)} \cong R_{(r)}^*$. In particular, $\dim H_{\ell}(s) = \dim H_r(s)$. The combination of \Cref{leftrightinv} and \Cref{the sets nice via left grp result} yields that the set  
$$\left\{\left. \sum_{x\in \mu_{\circ}(x')}S_{\psi_{y'}(x),x} \right| y'\in \bar{E}, x'\in G \right\}$$
is a generating set of $H_{\ell}(s)$. The latter set has cardinality $|\bar{E}|.|G|$ which equals $\dim H_r(s)$ by the just obtained basis for the right coefficients. Therefore the aforementioned generating set for $H_{\ell}(s)$ is a basis.
\end{proof}

\subsubsection{Proof that the basis is positive and $\Phi$-set theoretic}
With this basis at hand and the description of the structure maps in \Cref{structure maps coefficients hopf algebras}, now we can prove that $H_{\ell}(s)$ and $H_r(s)$ have a basis which is both positive and $\phi$-set-theoretic.

\begin{proof}[Proof of \Cref{pos basis prop theorem}]

Let us write $s(x,y)=(\psi_y(x),y\circ x)$ and $s^{-1}(x,y)= (xy, \theta_x(y))$. Since $H_r(s) \cong H_{\ell}(s)^*$ it is enough to prove that $H_r(s)$ has a positive basis. Recall that by \Cref{Th basis right coeff}, a basis for $H_r(s)$ is given by $\mc{B}_r := \left\{\left.\sum_{y\in \nu_{\psi}(y')}S_{y\circ x',y}\right| y'\in \bar{E}, x' \in G\right\}$. 

Let us first prove that $1_{H_r(s)}$ is positive. For this we rewrite: 
\begin{equation}\label{id positive}
\sum_{y'\in \bar{E}}\left( \sum_{y\in\nu_{\psi}(y')}S_{y\circ 1_G,y}\right) = \sum_{y\in S}S_{y,y} = 1_{H_r(s)}.
\end{equation}
So the identity expreses with coefficents $0$ or $1$ in terms of $\mc{B}_r$. To prove positivity of the co-unit $\epsilon_r$, it is enough to recall that $\epsilon_r(\sum_{y\in \nu_{\psi}(y')}S_{y\circ x',y}) = \delta_{S_{\psi_{y'}(x'),x'}}(1)$ (this follows from \Cref{structure maps coefficients hopf algebras} applied to basis elements). This yields 
\begin{align*}
\epsilon_r\left(\sum_{y\in \nu_{\psi}(y')}S_{y\circ x',y}\right) &= \delta_{S_{\psi_{y'}(x'),x'}} \left(\sum_{x\in S}S_{x,x} \right)\\
&=\begin{cases}
    1 \qquad &\text{if } \psi_{y'}(x') = x'\\
    0&\text{otherwise}
\end{cases}\\
 &= \begin{cases}
    1 \qquad &\text{if } \psi_{y'}= \id_S\\
    0&\text{otherwise}.
\end{cases}
\end{align*}
where in the last equality we used that the $\psi_{y'}$ are fixed point free by \Cref{prop:structureSs}.
Now we can prove that also the multiplication has positive coefficients.
Let $y', z' \in \bar{E}$ and $x',u' \in G$ then 
\begin{equation}\label{pos mult}
\begin{array}{lcl}
    \left(\sum_{y\in \nu_{\psi}(y')}S_{y\circ x',y}\right)\left(\sum_{z\in \nu_{\psi}(z')}S_{z\circ u',z}\right)
    &= & \sum\limits_{\substack{y\in \nu_{\psi}(y')\\ z\in\nu_{\psi}(z')}}S_{y\circ x', y}\,S_{z\circ u',z} \\
    &= & \sum\limits_{\substack{z\in\nu_{\psi}(z')\, : \, \\ z\circ u' \in \nu_{\psi}(y')}}S_{z\circ (u'\circ x'),z}\\
    &=  &\delta_{\nu_{\psi}(y'), \nu_{\psi}(z'\circ u')}\sum\limits_{z\in\nu_{\psi}(z')}S_{z\circ (u'\circ x'),z}.
\end{array}
\end{equation}

Next we prove that the co-multiplication is positive. For this recall that $$s^{A} = \sum\limits_{c,d\in S} S_{\psi_d(c), c}\otimes S_{d \circ c,d} \text{ and } (s^{A})^{-1} = \sum\limits_{a,b\in S} S_{ab, a}\otimes S_{\theta_a(b),b}. $$
Thus we now compute that
\begin{align*}
\D_r(\sum_{y\in \nu_{\psi}(y')}S_{y\circ x',y}) & = s^{A} (\sum_{y\in \nu_{\psi}(y')}S_{y\circ x',y}) (s^{A})^{-1} \\
& = \left( \sum_{c,d\in S}\sum_{y \in \nu_{\psi}(y')} S_{\psi_d(c),c}.S_{y\circ x',y} \,\ot \, S_{d\circ c,d} \right) (s^{A})^{-1}\\
& = \sum_{d \in S} \sum_{y \in \nu_{\psi}(y')}  \sum_{a,b\in S} S_{\psi_d(y\circ x'), y} . S_{ab,a} \, \ot \, S_{d\circ c,d}S_{\theta_a(b),b} \\
\end{align*}

Since 
$$S_{\psi_d(y\circ x'), y} . S_{ab,a} = \left\{ \begin{array}{ll}
0, & \text{ if } y \neq ab\\
S_{\psi_d(ab \circ x'),a}, &\text{ if } y = ab
\end{array}\right.$$
and similarly $ S_{d\circ c,d}S_{\theta_a(b),b} = S_{\theta_a(b)\circ c, b}$ if $d = \theta_a(b)$ and $0$ otherwise, the triple sum simplifies to
\begin{align*}
\D_r(\sum_{y\in \nu_{\psi}(y')}S_{y\circ x',y}) & = \sum\limits_{a,b\in S \, : \, ab \in \nu_{\psi}(y')} S_{\psi_{\theta_a(b)}(ab \circ x'), a} \ot S_{\theta_a(b) \circ ((ab)\circ x'),b}\\
& = \sum\limits_{a,b\in S \, : \, ab \in \nu_{\psi}(y')} S_{\psi_{\theta_a(b)}(ab \circ x'), a} \ot S_{b\circ x',b} \\
& = \sum\limits_{a,b\in S \, : \, ab \in \nu_{\psi}(y')} S_{a \circ \psi_b(x'), a} \ot S_{b\circ x',b} \\
\end{align*}
where we used \eqref{eq2}, \eqref{eq3} and inverse relations to obtain $\theta_a(b) \circ ((ab)\circ x') = (\theta_a(b) \circ (ab))\circ x' = b \circ x'$ and 
$$\psi_{\theta_a(b)}(ab \circ x') = \psi_{\theta_a(b)}(ab) \circ \psi_{\theta_a(b)\circ ab}(x') = a \circ \psi_{\theta_a(b)\circ ab}(x') =a \circ \psi_b(x').$$

Since $\psi_{ab}=\psi_a\psi_b$ we have that $ab \in \nu_{\psi}(y')$ if and only if $\psi_a = \psi_{y'}\psi_b^{-1}$. By \cref{prop:structureSs} the set of $\psi$'s is a group and hence $\psi_{y'}\psi_b^{-1} = \psi_z$ for some $z \in E$ depending on $y'$ and $b$. Furthermore, $\psi_z = \psi_{y'}\psi_d^{-1}$ for any $d \in \nu_{\psi}(b)$. Thus we can rewrite the latest sum to obtain 
\begin{equation}\label{pos decomp coproduct}
\D_r(\sum_{y\in \nu_{\psi}(y')}S_{y\circ x',y})= \sum_{[d] \in \bar{E}}\left( \sum_{a \, : \, \psi_{ad}= \psi_{y'}} S_{a \circ \psi_d(x'), a} \ot \sum_{b \in \nu_{\psi}([d])}S_{b\circ x',b}\right)
\end{equation}
which is a linear combination with coefficients $0,1$ in the basis $\mc{B}_r.$ Note that the notation $\nu_{\psi}([d])$ is fine as, by definition of $\bar{E} =E / \sim$, the set do not depend on the chosen representative $d$.\medskip

Next, we prove that the basis is $\Phi$-set theoretic. If we denote $g_{(x',y')}:= \sum_{y\in \nu_{\psi}(y')}S_{y\circ x',y}$, then the latter means that
$$\D_r(g_{(x_1',y_1')}).(1 \ot g_{(x_2',y_2')}) \in \mc{B}_r \ot \mc{}B_r.$$
for $(x_i',y_i') \in \bar{E} \times G$, $i=1,2$. A direct computation yields:
\begin{align*}
\D_r(g_{(x_1',y_1')})(1 \ot g_{(x_2',y_2')} ) & = \sum_{\substack{a,b \in S\ : \\ ab \in \nu_{\psi}(y_1')}} \sum_{y \in \nu_{\psi}(y_2')} S_{\psi_{\theta_a(b)}(ab \circ x_1'), a} \ot S_{\theta_a(b)\circ c, b}.S_{y \circ x_2',y} \\
& = \sum_{y \in \nu_{\psi}(y_2')} \sum\limits_{\substack{a \in S \, : \,\\ ab \in \nu_{\psi}(y_1')}} S_{\psi_{\theta_a(b)}(c),a} \ot S_{\theta_a(b)\circ c,y} \\
\end{align*}
with $b = y\circ x_2' $ and $c = ab \circ x_1'= a(y\circ x_2')\circ x_1'$. Using \eqref{eq2}, \eqref{eq3} and inverse relations we note that
\begin{align*}
\psi_{\theta_a(b)}(c) &= a \circ \psi_{\theta_a(b)\circ ab}(x_1') = a \circ\psi_{b}(x_1') \\
\theta_a(b)\circ c & = b\circ x_1' = y\circ (x_2' \circ x_1')
\end{align*}
Thus we can rewrite as
\begin{equation}\label{expression phi-set theoretic}
\D_r(g_{(x_1',y_1')})(1 \ot g_{(x_2',y_2')} ) = \sum\limits_{a \in S \, : \, ab \in \nu_{\psi}(y_1')} S_{a \circ\psi_{b}(x_1'),a} \ot  \sum_{y \in \nu_{\psi}(y_2')} S_{y\circ (x_2' \circ x_1'), y}
\end{equation}
which is a pure tensor in $\mc{B}_r \ot \mc{B}_r$ since $\psi_{y_1'}\psi_b^{-1}$ equals some $\psi_z$ by \cref{prop:structureSs}. Expression \eqref{expression phi-set theoretic} could also have obtained using the obtained expressions for the (co-)multiplication, i.e. from \eqref{pos decomp coproduct} and \eqref{pos mult}. \medskip

 Finally we consider the antipode. By definition  $S(a^{\ast}\otimes 1)(s^{A}) = (a^{\ast}\otimes 1)((s^{A})^{-1}).$
Therefore
 $$S((a^{\ast}\otimes 1)(s^{A})) = \sum_{x,y\in S}a^{\ast}(S_{xy,x})\,  S_{\theta_x(y),y}.$$
 Now, 
 $$\delta_{S_{\psi_{y'}(x')}, x'}(S_{xy,x})
 =\begin{cases}
     1, \qquad &\text{if } x=x'\text{ and } xy=\psi_{y'}(x')\\    
     0, &\text{otherwise}.
 \end{cases}$$
 Hence
 $$S((a^{\ast}\otimes 1)(s^{A})) = \sum_{\substack{y \in S\, : \,\\  \psi_{y'}(x') = x'y}} S_{\theta_{x'}(y),y}$$
We claim that $S((a^{\ast}\otimes 1)(s^{A}))$ is again an element in $\mc{B}_r$, namely:
\begin{equation}\label{antipode positive}
    S((a^{\ast}\otimes 1)(s^{A})) = \sum_{\substack{y \in S\,:  \, \\  \psi_y = \psi_{y'\circ x'}^{-1}}} S_{y \circ \psi_{y'}(x')^{-1}\ , \,y}
\end{equation}
The claim follows by similar arguments with the same identities as for the previous operations. Firstly note that $\theta_{x'}(y) = y\circ \psi_{y'}(x')^{-1}$. Indeed $\psi_{y'}(x') = x'y$ and thus $y = \theta_{x'}(y) \circ x'y = \theta_{x'}(y) \circ \psi_{y'}(x').$ Next, one verifies that the condition $\psi_{y'}(x') = x'y$ is equivalent to $\psi_y = \psi_{y'\circ x'}^{-1}$. For instance, 
$$\psi_{x'}\psi_y = \psi_{x'y} = \psi_{\psi_{y'}(x')} = \psi_{x'}\psi_{y' \circ x'}^{-1}$$
which is equivalent to $\psi_y = \psi_{y'\circ x'}^{-1}$ as the $\psi$ maps are bijective.
\end{proof}

\section{Description of cocommutative set-theoretic PE solutions}\label{cocomm classif section}
Our next aim is to classify all cocommutative set-theoretic solutions $(S,s)$ of the RPE. Recall following properties for a RPE solution
\begin{align}
Z_{12}Z_{13}=Z_{13}Z_{12} \tag{cocommutative}\\
Z_{13}Z_{23}=Z_{23}Z_{13} \tag{commutative}
\end{align}
where the equations need to be interpreted appropriatly depending on the type of solution (set, vector or algebra). 

The main result of this section classifies their associated vector space solution $(k[S],s^v)$. Recall that $\Phi_H$ denotes the RPE solution given in \eqref{concrete for Davydov} associated to a Hopf algebra $H$. 

\begin{theorem}\label{classif cocomm sol}
Let $(S,s)$ be a reachable bijective solution of RPE on the set $S$. Then $(S,s)$ is cocommutative if and only if there exists a group $G$ and $\Phi$-set theoretic basis $\mc{B}$ of $k[G]$ such that 
$$ s = \phi_{\mc{B}} \times 1_{X}$$ 
for some set $X$, where $\phi_{\mc{B}}$ is the set-theoretic solution on $\mc{B} \times \mc{B}$ associated to $\restr{\Phi_{k[G]}}{\mc{B} \ot \mc{B}}$.
\end{theorem}

In \Cref{solutions from grp alg} we will study the $\Phi$-set theoretic bases of a group algebra and hence the possible maps for $\phi_{\mc{B}}$. 

The solution $\Phi$ for $k_{\fin}^{A}$ with $A$ an infinite abelian, see \Cref{sol from dual grp alg} and \Cref{basic example cocomm}, is cocommutative but not reachable. The latter fact will follow from \Cref{classif cocomm sol} combined with the upcoming \Cref{Classification theorem basis grp alg}. However the aforementioned solution, can be obtained via multiplier Hopf algebras. We expect that all cocommutative solutions are either obtained via a $\phi$-set theoretic basis of a group algebra $k[G]$ or of $k_{\fin}^{A}$ via $A$ an infinite abelian group. It is tempting to prove the latter fact by lifting the solution $(k[S],s^v)$ to a multiplicative unitary on a seperable Hilbert space, which allows to apply Baaj-Skandalis result \cite[Theorem 0.1]{BaSk03} (or \cite[Theorem 2.1]{BaSk93}). This process however do not preserve bases and thus it is not clear how to use it to recover information on the starting set-theoretic solution.\smallskip

Our proof in the setting of Hopf algebras will obtain on its way that certain (non cocommutative) Hopf algebras do not admit a $\Phi$-set theoretic basis. Concretely, the two main ingredients for \Cref{classif cocomm sol} will be Cartier-Konstant-Milnor-Moore classification of cocommutative Hopf algebras and the study in \Cref{sectie universal enveloping} of basis preserving solutions made out of the universal enveloping of a Lie algebra. More generally we will consider Hopf algebras which are domains, see \Cref{domain no set solution} and \Cref{thm:groupalgebra-under-coalgebasis}.

\subsection{Relation (co)commutativity of Hopf algebras and solutions}\label{cocomm from sol to hopf}

A conceptual important property of set-theoretic solutions is that it has duality, as for multiplier Hopf algebras but unlike Hopf algebras, cf \Cref{dual in set } and \Cref{connection linear and pullback sol}. Under this duality commutativity and cocommutativity get swapped, as one would like to.

\begin{proposition}\label{duality in set sol}
Let $(S,s)$ be a set-theoretic solution of the RPE and set $t=\tau\circ s^{-1}\circ \tau$.
Then $(S,s)$ is commutative if and only if $(S,t)$ is cocommutative.
\end{proposition}
\begin{proof}
    Write $s^{-1}(x,y)=(xy,\theta_x(y))$. Then $t(x,y)=(\theta_y(x),yx)$. Note that if $s$ satisfies $s_{13}s_{23}=s_{23}s_{13}$, then so does $s^{-1}$.
    A direct computation shows that the commutativity of $s^{-1}$ is equivalent to the two identities
    $$x\theta_y(z)=xz, \qquad \theta_y\theta_x=\theta_x\theta_y, \qquad \forall\,x,y,z\in S.$$
    Now compute, for $(x,y,z)\in S^3$,
    $$t_{12}t_{13}(x,y,z) = \left(\theta_y\theta_z(x),\,y\,\theta_z(x),\,zx\right)$$
    and 
    $$t_{13}t_{12}(x,y,z) = \left(\theta_z\theta_y(x),\,yx,\,z\,\theta_y(x)\right).$$
    It follows that $t_{12}t_{13}=t_{13}t_{12}$, i.e. $(S,t)$ is cocommutative.

    Conversely, assume that $t$ is cocommutative. Then $t^{-1}$ is cocommutative as well. Indeed, write $t^{-1}(x,y)=(x\circ y,\psi_x(y))$, so $s(x,y)= (\psi_y(x),y\circ x)$. A direct computation shows that the cocommutativity of $t^{-1}$ is equivalent to the two identities
    $$(x\circ z)\circ y=(x\circ y)\circ z, \qquad \psi_{x\circ z}=\psi_x \qquad \forall\,x,y,z\in S.$$
    Now, for $(x,y,z)\in S^3$ we compute
    $$s_{13}s_{23}(x,y,z) = (\psi_{z\circ y}(x),\,\psi_z(y),\,(z\circ y)\circ x)$$
    and 
    $$s_{23}s_{13}(x,y,z) = (\psi_z(x),\,\psi_{z\circ x}(y),\,(z\circ x)\circ y).$$
    Hence, $s_{13}s_{23}= s_{23}s_{13}$, i.e. $(S,s)$ is commutative. 
\end{proof}

\begin{example}\label{basic example cocomm}
Consider the solution $s(g,h) = (g,gh)$ from \Cref{sol from grp alg} on the set semigroup $S$. Then 
$$s_{12}s_{13}(g,h,t) = s_{23}(g,h,gt) = (g,gh,gt) = s_{13}(g,gh,t)=s_{13}s_{12}(g,h,t).$$
Thus the solutions is cocommutative. It is readily verified to be commutative if and only if $S$ is abelian. Recall that this solution corresponded to $\Phi_{k[S]}$ and note that $k[S]$ is also cocommutative and commutative exactly when $S$ is abelian.\smallskip

Next consider the solution $s(g,h) = (gh^{-1},h)$ from \Cref{sol from dual grp alg} on the group $G$ which is associated to $\Phi_{k[G]^*}$ . Verifying the definition would show that the solution is commutative. Moreover it is cocommutative if and only if $G$ is abelian. Again this reflects the properties of $k[G]^*$. In \Cref{matched pair cocomm example} below we will generalize above examples to the solution \eqref{set mashed pair} associated to mashed pair.
\end{example}

In case that $R = s^{A}$ for a set-theoretic solution $(S,s)$, it is enough to verify (co)commutativity on the basis. In other words, $s^A$ is (co)commutative exactly when $(S,s)$ is. Interestingly, in general (co)-commutativity translates into the analogue concept for the associated coefficient Hopf algebras. 
\begin{proposition}\label{behaviour cocomm}
If $(A,R)$ is a solution of RPE, then the following are equivalent:
\begin{itemize}
    \item $R_{(r)}$ is commutative (resp. cocommutative) 
    \item\label{left coeff cocomm} $R_{(l)}$ is cocommutative (resp. commutative)
    \item\label{alg sol comm} $(A,R)$ is commutative (resp. cocommutative).
    \end{itemize}
\end{proposition}
 It would be interesting to translate other properties of Hopf algebras to set-theoretic solutions.  For example as cocommutative Hopf algebras are pointed, one could wonder whether there exists a senseful notion of pointed solutions.
\begin{proof}
By \cite[Theorem 2.1]{Mi04} one has that $R_{(l)}$ and $R_{(r)}^*$ are isomorphic as Hopf algebras. The effect of taking the dual is to interchange commutativity and cocommutativity. Hence it is sufficient to prove the equivalence between \eqref{left coeff cocomm} and \eqref{alg sol comm}. We will give the details for when $(A,R)$ is cocommutative, as commutativity would follow from an analogue proof.

Let $R = \sum_{i \in I} a_i \ot b_i$ be a minimal decomposition of $R$ in $A^{\ot 2}$. Then the set $\{ a_i \mid i \in I\}$ forms a basis of $R_{(l)}$. Using that decomposition of $R$ one has that 
$$R_{12}R_{13} = R_{13}R_{12} \Leftrightarrow \sum_{i,j} a_ia_j \ot b_i \ot b_j = \sum_{i,j} a_ja_i\ot b_i \ot b_j.$$
In other words, cocommutativity of $(A,R)$ is equivalent to $\sum_{i,j} (a_ia_j - a_ja_i) \ot b_i \ot b_j=0.$ Since the $b_j$ are linearly independent, the evaluation of $(\id_{A} \ot b_i)\circ(\id_{A \ot A} \ot \delta_{b_j})$ on the latter equation yields that
$$ a_i a_j - a_ja_i  = 0$$
for any $i, j \in I$, as desired.
\end{proof}

\begin{example} \label{matched pair cocomm example}
Consider the RPE solution from \eqref{set mashed pair}:
$$
\phi_{B\bowtie N}\bigl((s,u)\, ,\, (t,v)\bigr)
=
\Bigl( \bigl(s\,(t\triangleleft u^{-1})^{-1}, (t\triangleleft u^{-1})\triangleright u\bigr) \, , \,
(t\triangleleft u^{-1},  uv ) \Bigr)
$$
\noindent {\it Claim:} $\phi_{B\bowtie N}$ is cocomutative if and only if $B$ is abelian and the left action of $B$ on $N$ is trivial. In other words $B \bowtie N$ is a semidirect product $B \rtimes N$ with $B$ abelian. \medskip

 Take $x=(s,u), y=(t,v), z=(r,w)\in B \times N$. For ease of notation denote $\phi := \phi_{B\bowtie N}, a := r\trianglel u^{-1}$ and $n = a \triangler u$. We compute that
 \begin{align*}
 \phi_{12}\phi_{13}(x,y,z) & = \phi_{12}\bigl(
(s a^{-1}, n),\ (t,v),\ (a,uw)
\bigr).\\ 
& = \Bigl(
(s a^{-1}( t\trianglel n^{-1})^{-1},\, (t\trianglel n^{-1})\triangler n),\ 
( t\trianglel n^{-1},\, n v),\ 
(a,\, uw)
\Bigr). 
 \end{align*}
Next denote $a_0 := t\trianglel u^{-1}$ and $n_0 := a_0\triangler u.$ Then,
\begin{align*}
 \phi_{13}\phi_{12}(x,y,z) & = \phi_{13} \bigl(
(s a_0^{-1}, n_0),\ (a_0,uv),\ (r,w)
\bigr). \\
& = \Bigl(
(s a_0^{-1}(r\triangleleft n_0^{-1})^{-1},\, (r\triangleleft n_0^{-1})\triangleright n_0),\ 
(a_0,\, uv),\ 
(r\triangleleft n_0^{-1},\, n_0 w)
\Bigr). 
\end{align*}
Equality of both expressions would yield that $(a_0,\, uv) = (t\trianglel n^{-1}, nv)$ for all $v \in N$. Therefore, $uv = nv$ and hence $u = (r\trianglel u^{-1}) \triangler u$. As the map $r\mapsto r\triangleleft u^{-1}$ is bijective, this is equivalent to $ b\triangleright u = u$ for all $b\in B,\ u\in N$.

The triviality of the left action, entails that $n = u = n_0$. Therefore, the equality of the first coordinates is now equivalent to $sa^{-1}(t \trianglel u^{-1})^{-1} = sa_0^{-1}(r\trianglel u^{-1})^{-1}.$ Which is can be rewritten as $(t \trianglel u^{-1}) (r \trianglel u^{-1}) = (r\trianglel u^{-1}) (t \trianglel u^{-1})$. Since the image of $b\mapsto b\triangleleft u^{-1}$ is all of $B$, this holds for all $r,t$ if and only if $B$ is abelian. This finishes the proof of the claim.
\end{example}

Next we provide a variant of \Cref{behaviour cocomm} but for the solution $\Phi_H$ of a Hopf algebra, see also \cite[Proposition 2.7]{Mi98}.

\begin{lemma}\label{phi versus H cocomm}
Let $H$ be a bialgebra. The following are equivalent:
    \begin{enumerate}
        \item $H$ is cocommutative,
        \item $\Phi_H = (1\ot m) (\Delta \ot 1)$ is a cocommutative RPE solution.
    \end{enumerate}
\end{lemma}
\begin{proof}
We first write out what it means for $\Phi := \Phi_H$ to be cocommutative:
\begin{align*}
    \Phi_{12}\Phi_{13}(x\otimes y\otimes z) & = x_{(1)}\otimes y\otimes x_{(2)}z \\
    & = x_{(1)}\otimes x_{(2)}y\otimes x_{(3)}z
\end{align*}
And on the other hand 
\begin{align*}
\Phi_{13}\Phi_{12}(x\otimes y\otimes z) & = x_{(1)}\otimes x_{(2)}y\otimes z\\
& = x_{(1)}\otimes x_{(3)}y\otimes x_{(2)}z
\end{align*}

Hence $\Phi_{12}\Phi_{13}=\Phi_{13}\Phi_{12}$ if and only if $x_{(2)} \ot x_{(3)} = x_{(3)} \ot x_{(2)}$ for all $x \in H$ (necessaity is seen by choosing $y=z= 1_H$ ). The latter is the defintion for $H$ to be cocommutative.
\end{proof}

\subsection{set-theoretic solutions and universal enveloping of a Lie algebra}\label{sectie universal enveloping}

The aim of this section is to show that the phenomena noticed in \Cref{sol from lie algebra} was not an isolated fact, i.e. that $U(\mathfrak{g})$ has no $\Phi$-set theoretic basis. We obtain such statments for more general classes of Hopf algebras.

\subsubsection{Non-existence of (non group like) Set theoretic bases}
We will obtain the non-existence for the more general class of Hopf algebras that are a domain.

\begin{theorem}\label{domain no set solution}
Let $H$ be a Hopf algebra which is a domain. If $H$ has a $\Phi$-set-theoretic basis $\mc{B}$, then for each $b \in \mc{B}$ we have that $\D_H(b) = \eps_H(b)^{-1}\,  b \otimes b$.
\end{theorem}

It is well-known that Poincar\'e-Birkhoff-Witt theorem implies that $U(\mathfrak{g})$ has the following properties:
\begin{enumerate}
    \item[(i)] The associated graded algebra of $U(\mathfrak{g})$ is the symmetric algebra. In particular it is a domain and hence also $U(\mathfrak{g})$ is a domain.
    \item[(ii)] The universal enveloping contains no group-like elements (this holds more generally for any connected Hopf algebra).
\end{enumerate}

Now note if $\D (b) = \eps(b)^{-1}b \ot b$, then $\D(\eps^{-1}(b) b) =\eps^{-1}(b) \D(b) = \eps^{-1}(b) b \ot \eps^{-1}(b)b.$ Thus $\eps^{-1}(b) b$ is group like. Hence \Cref{domain no set solution} implies the desired conclusion.

\begin{corollary}\label{no set for Lie}
Let $\mathfrak{g}$ be a finite dimenisonal Lie algebra. Then $U(\mathfrak{g})$ has no $\Phi$-set-theoretic basis.
\end{corollary}

As notice above, by changing the elements $b \in \mc{B}$ by a scalar multiply, \Cref{domain no set solution} shows that the existence of $\Phi$-set-theoretic basis implies the existence of a coalgebra basis, whenever $H$ is a domain.

\begin{definition}
A basis $\mc{B}$ of a coalgebra $(C,\D,\eps)$ will be called a \emph{coalgebra basis} if
\[
\D(\mc{B})\subseteq k^* \mc{B}\otimes \mc{B}
\]
where $k^* \mc{B}\otimes \mc{B}= \{ \lambda b_1 \ot b_2 \mid \lambda \in k^*, b_1,b_2 \in \mc{B}\}.$
\end{definition}

We now complement \Cref{domain no set solution}, by showing that the existence of a coalgebra basis entails that $H$ is isomorphic to a group algebra.

\begin{proposition}\label{thm:groupalgebra-under-coalgebasis}
Let $H$ be a Hopf algebra. Assume there exists a coalgebra basis $\mc{B}$ of $H$. Then $\{\eps^{-1}(b)b \mid b \in \mc{B} \} = G(H)$ the group of group-like elements. Moreover, $H \cong k[G(H)]$.
\end{proposition}

\begin{remark*}
As a consequence of \Cref{domain no set solution} and \Cref{thm:groupalgebra-under-coalgebasis} we obtain that a Hopf algebra which is a domain and posses a $\Phi$-set theoretic basis, must be isomorphic to a group algebra. Now, recall that if $G$ is a group having a torsion element, then $k[G]$ is not a domain. Indeed, a torsion element $g$ yields a non-trivial idempotent $\hat{g} := \frac{1}{o(g)} \sum_{i=1}^{o(g)} g^{i}$ and hence a zero divisor. On other hand if $G$ is torsion-free, it is a notorious conjecture of Kaplansky that $k[G]$ is a domain. In conclusion, a basis as in \Cref{thm:groupalgebra-under-coalgebasis} can not exist if $H$ is finite dimensional.
\end{remark*}

Finally, we point out that except for group algebras, a $\Phi$-set theoretic basis do not contain the unit element.

\begin{proposition}\label{lem:unit-obstruction}
Let $H$ be a Hopf algebra and $\mc{B}$ a $\Phi$-set theoretic basis of $H$ such that $1\in B$. Then $H\cong k[G(H)]$ with $G(H)$ the group-like elements of $H$ and $\{\eps^{-1}(b)b \mid b \in \mc{B} \}=G(H).$
\end{proposition}
\begin{proof}
Take $b\in \mc{B}$. Then $\Phi_H(b \ot 1) = \D(b)$. Hence, as $\mc{B}$ is $\Phi$-set theoretic, this entails that $\D(b) = c \ot d$ for some $c,d \in  \mc{B}$. Applying $\eps\otimes \id$ and $\id\ot \eps$, one obtains in the same way as in the proof of \Cref{domain no set solution}  that $\eps(c)d=b = \eps(d)c$. Thus $D(b) = \lambda\, b \ot b$ for some $\lambda \in k^*$. The conclusion now follows from \Cref{thm:groupalgebra-under-coalgebasis}.
\end{proof}

\subsubsection{Proofs of the statements}\label{section of proofs domain case}

To start we prove the statement on Hopf algebras that are a domain.

\begin{proof}[Proof of \Cref{domain no set solution}]
For $k$-vector spaces $V$ and $W$ and $T \in V \ot W$, we say that $T$ is written in left-reduced form if $T=\sum_{i=1}^n v_i\otimes w_i$
with $v_1,\dots,v_n\in V$ linearly independent and $n$ minimal among all such expressions with linearly independent left factors.\smallskip

\noindent {\it Claim 1:} $\D (y)$ is a pure tensor in $H\otimes H$ for each $y\in \mc{B}$.\smallskip

Take a left-reduced form of  $\D (y)=\sum_{i=1}^n a_i\otimes b_i$. Thus $a_1,\dots,a_n$ are linearly independent and with $n$ minimal. For any $z\in H$ we have
\begin{equation}\label{reduced form}
\Phi_H(y\otimes z)=\sum_{i=1}^n a_i\otimes b_i z.
\end{equation}
Note that the left-reducedness implies that the $b_1, \ldots, b_n$ are linearly independent. Indeed, if the $b_i$ satisfied a nontrivial relation, one could rewrite \eqref{reduced form} with fewer summands and still keep the left factors independent. 

Since $H$ is a domain, right multiplication $\rho_z(h)=hz$ is injective on $H$. This entails that also the $b_1z, \ldots, b_nz$ are linearly independent. Now taking $z \in \mc{B}$ and using that $\mc{B}$ is $\Phi$-set theoretic we obtain Claim $1$.\medskip

\noindent {\it Claim 2: }If $\D (y)$ is a simple tensor, then $\D (y)=\eps(y)^{-1}\, y\otimes y$.\smallskip

Let $y \neq 0$ and assume $\D (y)=u\otimes v$ for $u,v \in H$. Applying $\eps\otimes \id$ and using that $(\eps \ot \id)\D = \id$ gives
\[
y=(\eps\otimes\id )\D (y)=\eps(u)\,v.
\]
Hence $v=\eps(u)^{-1}y$. Similarly applying $\id\ot \eps$ and using that $(\id \ot \eps)\D = \id$, yields that $u=\eps(v)^{-1}y$.
Therefore $\D(y)=\lambda\, y\ot y$ for some $\lambda\in k^*$. Finally, apply again $(\eps\ot \id)$ to see that $y=(\eps\ot \id)\D(y)=\lambda\,\eps(y)\,y. $ Consequently, $\lambda\eps(y)=1$, finishing the proof of Claim 2.
\end{proof}

Next, we consider the statement about Hopf algebras having a basis that is both $\Phi$-set theoretic and a coalgebra basis.

\begin{proof}[Proof of \Cref{thm:groupalgebra-under-coalgebasis}]
Take $b\in \mc{B}$. By the coalgebra-basis condition, $\D(b)=\lambda b'\otimes b''$ for some $b',b''\in \mc{B}$. Applying $\eps\otimes \id$ and $\id\ot \eps$, one obtains in the same way as in the proof of Claim 2 that $b' = b = b''$ and $\lambda = \eps(b)^{-1}$.

With this property at hand we obtain that $H$ must be a cocommative Hopf algebra. Indeed, take $h \in H$ and decompose $h = \sum_{b \in \mc{B}} \lambda_b b$ into the basis. Then 
$$\D (h) = \sum_{b \in \mc{B}} \lambda_b \D(b) = \sum_{b \in \mc{B}} \lambda_b \eps(b)^{-1} b\ot b = \sum_{b \in \mc{B}} \lambda_b \D^{op}(b) =\D^{op}(h).$$

Now, recall Cartier-Konstant-Milnor-Moore classification, saying that $H$ is isomorphic to $U(P(H)) \rtimes k[G(H)]$, where $\mathfrak{g} := P(H)$ is the Lie algebra consisting of primitive elements in $H$ and $G := G(H)$ is the group of group-like elements. However the universal enveloping has no group-like elements, thus the Lie part must be trivial and $H \cong k[G(H)].$

Finally, note that $\eps^{-1}(b)b$ is group-like for each $b \in \mc{B}$. This entails that $G(H) = \{ \eps^{-1}(b)b \mid b \in \mc{B} \}$.
\end{proof}

Finally, using the classification of cocommutative algebras, together with \Cref{subsection sol from hopf}, we can relate classification of reachable cocommutative set-theoretic solutions of RPE to $\Phi$-set theoretic bases.

\begin{proof}[Proof of \Cref{classif cocomm sol} for reachable solutions]
Let $(S,s)$ be a reachable set-theoretic solution of the RPE. By definition this means that there exists a Hopf algebra $H$ and Hopf $H$-module $M$ such that $s^v = \Phi_M \cong \Phi_H \ot 1_V$ for $V$ a vector space with $\dim_k V = |S| - \dim_k H$. The fact that it reaches the linearisation of a set-theoretic solution implies that $H$ must contain a $\Phi$-set theoretic basis $\mc{B}$ such that $s \cong \phi_{\mc{B}} \times 1_{B(V)}$ with $B(V)$ a basis of $V$.

Thus it remains to show that $H$ is a group algebra. Note that by \Cref{phi versus H cocomm} the Hopf algebra $H$ must be cocommutative. Therefore Cartier-Konstant-Milnor-Moore classification implies that 
$$H \cong U(\mathfrak{g}) \rtimes k[G]$$
with $\mathfrak{g} := P(H)$ the Lie algebra consisting of primitive elements in $H$ and $G := G(H)$ the group of group-like elements. Hence by \Cref{solutions of hopf mashed pair} one has that
\begin{equation}\label{sol cocomm decomp}
\Phi_{U(\mathfrak{g})} \# \Phi_{k[G]} = \Phi_{H}
\end{equation}
However by \Cref{no set for Lie} the solution $\Phi_{U(\mathfrak{g})}$ can not be restricted to a set-theoretic solution. Thus $H \cong k[G]$, finishing the proof.
\end{proof}

\section{Set-theoretic solutions constructable from group algebra}\label{solutions from grp alg}

In \Cref{sol from dual grp alg} we have see that one can construct two non-isomorphic set-theoretic solutions from the group algebra $k[G]$ of an abelian group, through different choices of bases. Further motivated by \Cref{classif cocomm sol} we study in this section $\Phi$-set theoretic bases of any group algebra $k[G]$.

If $G$ can be decomposed as a semidirect product $A \rtimes N$ for some finite abelian group $A$, we construct in \Cref{sectie make out of finite abelian} a $\Phi$-set theoretic basis which yields on the matched pair solution from \Cref{solutions of hopf mashed pair}  on $A^{\vee} \times N.$ To so we first recall in \Cref{fourier transform section} the necessary background on the Fourier transform for finite abelian groups. Thereafter, in \Cref{exotic basis for infinite groups}, we provide restrictions on general $\Phi$-set theoretic bases in terms of non-torsioin elements of $G$. For instance we show that if $G$ is torsion-free, then $G$ is the only $\Phi$-set theoretic basis. 

\subsection{The abelian case: recollection on Fourier transform}\label{fourier transform section}
Let $A$ be a finite abelian group. From now on we assume that the ground field $k$ satisfies $\Char(k) \nmid |A|$ and that $k$ contains all $|A|$-roots of unity. As in \Cref{sol from grp alg} we want to related $k[A]$ with its dual through the Fourier transform. Concretely, define the $k$-linear map
\begin{equation}\label{eq:ThetaDef}
\Theta_A:k[A^\vee]^*\longrightarrow k[A]: \delta_\chi \mapsto \frac{1}{|A|}\sum_{a\in A}\chi(a^{-1})\,a.
\end{equation}

The following fact is well-known.

\begin{proposition}\label{prop:ThetaInv}
The map $\Theta_A$ is an isomorphism of Hopf algebras. Its inverse $\Theta_A^{-1}:k[A]\to k[A^\vee]^*$ is given on group elements $a \in A$ by
\begin{equation}\label{eq:ThetaInv}
\Theta_A^{-1}(a)=\sum_{\chi\in A^\vee}\chi(a)\,\delta_\chi,
\end{equation}
extended $k$-linearly.
\end{proposition}

We now show that the group solution $s(a \ot b) = a \ot ab$ on $k[A]$ seen in \Cref{sol from grp alg} restrict to the solution from \Cref{sol from dual grp alg} by choosing the basis $\{ \Theta_A(\delta_{\chi}) \mid \chi \in A^{\vee} \}$.

\begin{proposition}\label{prop:s2Theta}
For $\chi_1,\chi_2\in A^\vee$ one has
\[
s\bigl((\Theta_A\otimes\Theta_A)(\delta_{\chi_1}\otimes \delta_{\chi_2})\bigr)
=\Theta_A(\delta_{\chi_1\chi_2^{-1}})\otimes \Theta_A(\delta_{\chi_2}).
\]
\end{proposition}

\begin{proof}
By definition
\begin{align*}
(s\circ(\Theta_A\otimes\Theta_A))(\delta_{\chi_1}\otimes\delta_{\chi_2})
 & =s \left( \frac1{|A|^2}\sum_{a,b\in A}\chi_1(a^{-1})\chi_2(b^{-1})\,(a\otimes b) \right)\\
 & =\frac1{|A|^2}\sum_{a,b\in A}\chi_1(a^{-1})\chi_2(b^{-1})\,(a\otimes ab). \\
\end{align*}
Now reindex the sum via $c:=ab$ and hence $\chi_2(b^{-1})=\chi_2(c^{-1}a)=\chi_2(c^{-1})\chi_2(a)$. The latter sum becomes
\[
\frac1{|A|^2}\sum_{a,c}(\chi_2\chi_1^{-1})(a)\chi_2(c^{-1})\,(a\otimes c)
=\Bigl(\frac1{|A|}\sum_{a}(\chi_2\chi_1^{-1})(a)\,a\Bigr)\otimes\Bigl(\frac1{|A|}\sum_c\chi_2(c^{-1})\,c\Bigr)
\]
which exactly equals $\Theta_A(\delta_{\chi_1\chi_2^{-1}})\otimes \Theta_A(\delta_{\chi_2})$, as desired.
\end{proof}

For convenience of the non-expert reader we will now include a prove of \Cref{prop:ThetaInv}.  Recall that $k[A^\vee]^*$ is equipped with the Hopf algebra structure dual to the group algebra $k[A^\vee]$:
\begin{align*}
\delta_\chi\delta_\psi &= \delta_{\chi,\psi}\,\delta_\chi,\qquad 1=\sum_{\chi\in A^\vee}\delta_\chi,\\
\D(\delta_\chi) &= \sum_{\alpha\beta=\chi}\delta_\alpha\otimes \delta_\beta,\qquad
\varepsilon(\delta_\chi)=\delta_{\chi,1},\qquad
S(\delta_\chi)=\delta_{\chi^{-1}}.
\end{align*}

\begin{proof}[Proof of \Cref{prop:ThetaInv}]
We will regularly use the character orthogonality relations which hold due to our assumptions on the ground field $k$:
\begin{equation}\label{eq:orth}
\sum_{\chi\in A^\vee}\chi(x)=
\begin{cases}
|A|,& x=1,\\
0,& x\neq 1,
\end{cases}
\qquad
\sum_{a\in A}\chi(a)=
\begin{cases}
|A|,& \chi=1,\\
0,& \chi\neq 1.
\end{cases}
\end{equation}

\medskip\noindent
{\it Algebra map.} By definition 
$$ \Theta_A(\delta_\chi)\Theta_A(\delta_\psi)
=\frac1{|A|^2}\sum_{a,b\in A}\chi(a^{-1})\psi(b^{-1})\,ab.
$$
The coefficient of a fixed $t\in A$ equals
\[
\frac1{|A|^2}\sum_{a\in A}\chi(a^{-1})\psi\bigl((a^{-1}t)^{-1}\bigr)
=\frac{\psi(t^{-1})}{|A|^2}\sum_{a\in A}(\chi^{-1}\psi)(a).
\]
By \eqref{eq:orth} this is $\frac1{|A|}\chi(t^{-1})$ if $\chi=\psi$ and $0$ otherwise. Hence
$\Theta_A(\delta_\chi)\Theta_A(\delta_\psi)=\delta_{\chi,\psi}\Theta_A(\delta_\chi)=\Theta_A(\delta_\chi\delta_\psi)$.
The orthogonality relations also yield
\[
\Theta_A(1)=\Theta_A\Bigl(\sum_{\chi}\delta_\chi\Bigr)
=\frac1{|A|}\sum_{a\in A}\Bigl(\sum_{\chi}\chi(a^{-1})\Bigr)a
=1.
\]

\medskip\noindent
{\it Coalgebra map.}
We show $\D_{k[A]} (\Theta_A(\delta_\chi))=(\Theta_A\otimes \Theta_A)\D_{k[A^{\vee}]^*} (\delta_\chi)$ and drop the indices of the coproducts.
The left-hand side is
\[
\D (\Theta_A(\delta_\chi))=\frac1{|A|}\sum_{a\in A}\chi(a^{-1})\,(a\otimes a).
\]
The right-hand side:
\[
(\Theta_A\otimes\Theta_A)\D (\delta_\chi)
=\sum_{\alpha\beta=\chi}\Theta_A(\delta_\alpha)\otimes\Theta_A(\delta_\beta)
=\frac1{|A|^2}\sum_{\alpha\beta=\chi}\sum_{a,b\in A}\alpha(a^{-1})\beta(b^{-1})\,(a\otimes b).
\]
Writing $\beta=\alpha^{-1}\chi$, the coefficient of a fixed $a\otimes b$ is
\[
\frac1{|A|^2}\sum_{\alpha\in A^\vee}\alpha(a^{-1})(\alpha^{-1}\chi)(b^{-1})
=\frac{\chi(b^{-1})}{|A|^2}\sum_{\alpha\in A^\vee}\alpha(a^{-1}b)
=\begin{cases}
\frac1{|A|}\chi(a^{-1}),& a=b,\\
0,& a\neq b,
\end{cases}
\]
which equals $\D (\Theta_A(\delta_\chi))$.

\medskip\noindent
{\it Counit and antipode.}
For the counit,
\[
\varepsilon(\Theta_A(\delta_\chi))=\frac1{|A|}\sum_{a\in A}\chi(a^{-1})=\delta_{\chi,1}=\varepsilon(\delta_\chi).
\]
For the antipode,
\[
S(\Theta_A(\delta_\chi))
=\frac1{|A|}\sum_{a\in A}\chi(a^{-1})\,a^{-1}
=\frac1{|A|}\sum_{a\in A}\chi(a)\,a
=\Theta_A(\delta_{\chi^{-1}})
=\Theta_A(S(\delta_\chi)).
\]

\medskip\noindent
{\it Inverse.}
We have show that $\Theta_A$ is a Hopf algebra map and hence it remains to verify the inverse given as in \Cref{eq:ThetaInv}.

Fix $a\in A$. Then by \eqref{eq:ThetaDef},
\[
\Theta_A(\Theta_A^{-1}(a))
=\sum_{\chi\in A^\vee}\chi(a)\,\Theta_A(\delta_\chi)
=\frac1{|A|}\sum_{\chi\in A^\vee}\ \sum_{x\in A}\chi(a)\chi(x^{-1})\,x
=\frac1{|A|}\sum_{x\in A}\Bigl(\sum_{\chi\in A^\vee}\chi(ax^{-1})\Bigr)x.
\]
By orthogonality \eqref{eq:orth}, $\sum_{\chi}\chi(ax^{-1})=|A|$ iff $x=a$ and $0$ otherwise. Hence the above equals $a$.
Conversely, for $\delta_\psi$ we compute using \eqref{eq:ThetaInv}:
\begin{align*}
\Theta_A^{-1}(\Theta_A(\delta_\psi))
& =\frac1{|A|}\sum_{x\in A}\psi(x^{-1})\,\Theta_A^{-1}(x)\\
& =\frac1{|A|}\sum_{x\in A}\psi(x^{-1})\sum_{\chi\in A^\vee}\chi(x)\delta_\chi \\
& =\sum_{\chi\in A^\vee}\Bigl(\frac1{|A|}\sum_{x\in A}(\chi\psi^{-1})(x)\Bigr)\delta_\chi
=\delta_\psi,
\end{align*}
again by \eqref{eq:orth}. Thus $\Theta_A^{-1}$ is indeed the inverse.
\end{proof}

\subsection{Matched pairs with a finite abelian factor}\label{sectie make out of finite abelian}

Suppsoe that $G$ is the matched pair $B \bowtie N$ with $B$ and $N$ some subgroups. Consider the group solutions $\Phi_{k[G]}(g\ot h)= g \ot gh$ and denote $\phi := \Phi_{k[G]}$. The aim of this section is to construct a $\Phi$-set theoretic basis $\mc{B}$ of $k[G]$ such that $\restr{\Phi}{\mc{B} \ot \mc{B}}$ corresponds to the solution \eqref{set mashed pair} associated to $k[B]^*\bowtie k[N]$. However by \Cref{matched pair cocomm example} this is only possible if $B \bowtie N = B \rtimes N$ with $B$ abelian. Indeed $k[G]$ is cocommutative and thus by \Cref{phi versus H cocomm} so is $\Phi_{k[G]}$ and its restriction to any $\Phi$-set theoretic basis. In that case \eqref{set mashed pair} takes following form
\begin{equation}\label{mashed pair sol for semid}
\phi_{B\bowtie N}\bigl((s,u)\, ,\, (t,v)\bigr)
=
\Bigl( \bigl(s\,(t\triangleleft u^{-1})^{-1}, u\bigr) \, , \,
(t\triangleleft u^{-1},  uv ) \Bigr)
\end{equation}

We first consider the case that $B$ is finite abelian and denote it by $A$. We will make $k[A]^*$ appear using the maps introduced in \Cref{fourier transform section}.\medskip

\subsubsection{A dual like basis inside $k[A]\subset k[G]$:} for each $\chi \in A^{\vee}$ denote 
$$e_{\chi} := \Theta_A(\delta_{\chi}) = \frac{1}{|A|}\sum_{a\in A}\chi(a^{-1})\,a \in k[A].$$
Then by \Cref{prop:ThetaInv} the standard relations 
$e_\chi e_\psi=\delta_{\chi, \psi} e_\chi$ and $\sum_{\chi \in A^{\vee}} e_\chi=1$ hold. Moreover,
\begin{equation}\label{eq:trans}
b\,e_\chi=\chi(b)\,e_\chi.
\end{equation}
for all $b\in A$. Indeed, 
\[
b\,e_\chi=\frac{1}{|B|}\sum_{t\in B}\chi(t^{-1})\,bt
=\frac{1}{|B|}\sum_{t\in B}\chi\bigl((b^{-1} t)^{-1}\bigr)\,t
=\chi(b)\,\frac{1}{|B|}\sum_{t\in B}\chi(t^{-1})\,t=\chi(b)\,e_\chi.
\]

\noindent Note that $\{e_\chi\}_{\chi\in A^\vee}$ is a basis of $k[A]$ which motivates to consider the following set:
\[
\mathcal{B}_{A^\vee} := \{\, e_\chi\,u \mid \chi\in A^\vee,\ u\in N\,\}\subset k[G].
\]
\smallskip
The set $\mathcal{B}_{A^\vee}$ is a $k$-basis of $k[G]$ since $k[G]\cong k[A \times N]$ as vector spaces.

\subsubsection{Restriction of the group solution to the basis on $A^{\vee} \times N$}

We will now show that the basis $\mathcal{B}_{A^\vee}$ is $\Phi_{k[G]}$-set theoretic, see \Cref{thm:dual} below. First note that the right action $\triangleleft$ of $N$ on $A$ induces a left action of $N$ on $A^\vee$ by pullback:
\begin{equation}\label{eq:dualaction}
(u\cdot\chi)(a):=\chi(a\triangleleft u^{-1})
\end{equation}
for $u\in N,\ \chi\in A^\vee,\ a\in A$. This action permutes the Fourier idempotents $e_{\chi}$:
\begin{equation}\label{eq:conj}
u\,e_\chi\,u^{-1} = e_{u\cdot \chi}.
\end{equation}

\begin{proposition}\label{thm:dual}
Let $(A,N)$ a mashed pair with $A$ a finite abelian group acting trivially on $N$ and let $G = A \rtimes N$. With notations as above, we have for all $\alpha,\beta\in A^\vee$ and $u,v\in N$,
\begin{equation}\label{eq:Phi-dual}
\Phi_{k[G]}\bigl(e_\alpha u\ot e_\beta v \bigr)
=
\bigl(e_{\alpha\,(u\cdot\beta)^{-1}}\,u\bigr)\ \ot\ \bigl(e_{u\cdot\beta}\,uv\bigr).
\end{equation}
In particular, $\Phi_{k[G]}$ restricts to a set-theoretic map on $A^\vee\times N$,
\[
\phi_{A^\vee\bowtie N}:\ (A^\vee\times N)\times(A^\vee\times N)\to (A^\vee\times N)\times(A^\vee\times N),
\]
given by
\begin{equation}\label{eq:setdual}
\phi_{A^\vee\bowtie N}\bigl((\alpha,u),(\beta,v)\bigr)
=
\bigl((\alpha(u\cdot\beta)^{-1},u),\ (u\cdot\beta,uv)\bigr).
\end{equation}
\end{proposition}

Before proving \Cref{thm:dual} we explain why \eqref{eq:conj} holds. For this  recall that by definition conjugation by $u$ acts on $A$ as the automorphism
$a\mapsto a\triangleleft u^{-1}$. Thus $u\,e_\chi\,u^{-1}
=\frac{1}{|A|}\sum_{a\in A}\chi(a^{-1})\,(a\triangleleft u^{-1}).$ Now use that $(\cdot )\trianglel u^{-1}$ is a bijection of $A$ to relabel the sum via $c=a\triangleleft u^{-1}$ and obtain \eqref{eq:conj}:
\[
u\,e_\chi\,u^{-1}
=\frac{1}{|A|}\sum_{c\in A}\chi\bigl((c\triangleleft u)^{-1}\bigr)\,c
=\frac{1}{|A|}\sum_{c\in A}(u\cdot\chi)(c^{-1})\,c
=e_{u\cdot\chi}.
\]

\begin{proof}[Proof of \Cref{thm:dual}]
Rolling out the definitions we readily obtain that:
\begin{equation}\label{eq:stepB}
\Phi_{k[G]}\bigl(e_\alpha u \ot e_\beta v \bigr)
=
\frac{1}{|A|}\sum_{a\in A}\alpha(a^{-1})\, au\ot (au)(e_\beta v).
\end{equation}

By \eqref{eq:conj} and subsequently \eqref{eq:trans}, we have 
$$(au)(e_\beta v)=a\,e_{u\cdot\beta}\,(uv) = (u\cdot\beta)(a)\,e_{u\cdot\beta}\,(uv).$$

Substituting into \eqref{eq:stepB} yields
\begin{equation*}
\Phi_{k[G]}\bigl(e_\alpha u \ot e_\beta v\bigr)
=
\frac{1}{|A|}\sum_{a\in A}\alpha(a^{-1})\,(u\cdot\beta)(a)\, au \ot  e_{u\cdot\beta}\,uv.
\end{equation*}

Since $\alpha(a^{-1})(u\cdot\beta)(a)=\bigl(\alpha (u\cdot\beta)^{-1}\bigr)(a^{-1})$, we can rewrite following sum as
\[
\frac{1}{|A|}\sum_{a\in A}\alpha(a^{-1})(u\cdot\beta)(a)\, au
=
\left(\frac{1}{|A|}\sum_{a\in A}\bigl(\alpha (u\cdot\beta)^{-1}\bigr)(a^{-1})\,a\right)u
=
e_{\alpha (u\cdot\beta)^{-1}}\,u.
\]
Substituting this in the obtained expression for $\Phi_{k[G]}\bigl(e_\alpha u \ot e_\beta v\bigr)$, finishes the proof.
\end{proof} 

\subsubsection{A solution on $A \times N$}

To obtain a solution on the set $A \times N$ instead of the canonical $A^{\vee} \times N$ we need to invoke a non-canonical isomorphism $\iota: A \cong A^{\vee}.$ Such as isomorphism is equivalent to the choice of a nondegenerate bicharacter 
$$\langle\cdot, \cdot\rangle: A \times A \rightarrow k^*.$$
The associated isomorphism $A \cong A^{\vee}$ is given by $s \mapsto \chi_s$ where $\chi_s(t):=\langle s, t\rangle$. Then for $s \in A$ consider
$$
e_s:=e_{\chi_s}=\frac{1}{|A|} \sum_{a \in A}\left\langle s, a^{-1}\right\rangle a \in k[A].
$$
In this setting one would consider the basis $\{e_s \, u \mid s\in A, u\in N  \}$ of $k[G]$. To however be able to recover the desired solution \eqref{mashed pair sol for semid} one needs an $N$-equivariant isomorphism:
\begin{equation}\label{eq:eqv}
\iota (a\triangleleft u^{-1})=u\cdot \iota (a)
\qquad (a\in A,\ u\in N),
\end{equation}
where $u\cdot(-)$ is the pullback action \eqref{eq:dualaction}. In that case  for $t\in A$ we have that $ue_{\iota(t)} = e_{\iota(t) \trianglel u^{-1}}u$ using \Cref{eq:conj}. Such isomorphism unfortunately do not always exists.

\begin{proposition}\label{prop:pairing}
Let $(A,N,\triangleleft, \triangler)$ be a mashed pair of groups with $A$ abelian. Then the following are equivalent:
\begin{enumerate}
    \item \label{eq iso} There exists an $N$-equivariant isomorphism $\iota :A\overset{\sim}{\to}A^\vee$, 
    \item there exists a nondegenerate bicharacter
$\langle\cdot,\cdot\rangle:A\times A\to k^*$ such that
\[
\langle a\triangleleft u^{-1},\ b\rangle=\langle a,\ b\trianglel u\rangle
\qquad (a,b\in A,\ u\in N).
\]
\item \label{iso as mod} $A$ and $A^{\vee}$ are isomorphic as $\Z[N]$-modules.
\end{enumerate}
\end{proposition}

\begin{proof}
If $\iota$ is given, define $\langle a,b\rangle:=\iota(a)(b)$.
Then \eqref{eq:eqv} is equivalent to the displayed invariance by evaluating both sides at $b$ and using \eqref{eq:dualaction}.
Conversely, given such a pairing, define $\iota (a)(b):=\langle a,b\rangle$. Non-degeneracy makes $\iota$ into an isomorphism and the invariance yields \eqref{eq:eqv}. Clearly also \eqref{eq iso} and \eqref{iso as mod} are equivalent.
\end{proof}

Now suppose that there exists an $N$-equivariant isomorphism $\iota$.
Then the bijection
\[
A\times N \xrightarrow{\sim} A^\vee\times N : (a,u)\mapsto (\iota(a),u)
\]
transports \eqref{eq:setdual} to a cocommutative solution on $A\times N$.
Concretely, one obtains as desired the solution
\[
\phi\bigl((a,u),(b,v)\bigr)
=
\bigl((a\,(b\triangleleft u^{-1})^{-1},\ u),\ (b\triangleleft u^{-1},\ uv)\bigr).
\]

\subsection{Group algebra of infinite groups and non-existence criterion}\label{exotic basis for infinite groups}
Our aim now is to prove that a torsion-free group $G$ has no $\Phi$-set theoretic basis other than $\lambda.G= \{ \lambda g \mid g \in G\}$ for some scalar $\lambda \in k^*$. In particular, for such groups $k[G]$ only affords the group set theoretic solution. This will readily follow from following general obstruction. Recall that for $x\in k[G]$ the support of $x$ is the set
$$\Supp(x)=\{g:x_g\neq 0\}.$$

\begin{theorem}\label{thm:two-point-obstruction}
Let $G$ be a group, $t\in G$ of infinite order and let $\mc{B}$ be a $\Phi$-set theoretic basis. Then no element of $\mc{B}$ can have support containing two distinct points in the same left coset of $\langle t\rangle$:
\[
\bigl|\Supp(b)\cap \langle t\rangle g\bigr|\le 1.
\]
for every $b\in \mc{B}$ and every $g\in G$.

\end{theorem}

This indeed implies the desired application:

\begin{corollary}\label{no non-triv sol of torsion-fre}
Let $G$ be a torsion-free group and $\mc{B}$ a $\Phi$-set theoretic basis of $k[G]$. Then as a set $\mc{B} = \lambda. G$ for some $\lambda \in k^*$. In particular, the set-theoretic solution $t_{\mc{B}}$ is equivalent to the group solution on $G$.
\end{corollary}

\begin{remark}
On any group $G$ the map 
$$G^2 \rightarrow G^2: (g,h) \mapsto (gh^{-1},h)$$
is RPE solution. \Cref{no non-triv sol of torsion-fre} highlights that it is not clear how to obtain it in general as restriction of $\Phi_H$ on a $\Phi$-set theoretic basis of a Hopf algebra. Note however that one can always realize it as the solution associated to the multiplier Hopf algebra $k^{G}_{\fin}$, see \Cref{sol from dual grp alg}.
\end{remark}

\begin{proof}[Proof of \Cref{no non-triv sol of torsion-fre}]
First we show that any $b \in \mc{B}$ has $|\Supp(b)|=1$. Suppose otherwise and let $b\in \mc{B}$ have $|\Supp(b)|\geq 2$. Pick distinct $g,h\in \Supp(b)$ and consider the element $u:=gh^{-1}$ of infinite order. Note that $\Supp(b)$ meets the left coset $\langle u\rangle h$ in at least two points, namely $h$ and $uh=g$, which contradicts \Cref{thm:two-point-obstruction}. In conclusion, every basis element $b$ is a scalar multiple of a group element, say $b = \lambda_b g_b \in k^*G$. It remains to prove that all scalars $\lambda_b$ are equal.

Now consider the RPE solution $\Phi_{k[G]}$ which on basis elements is $g \ot h \mapsto g \ot gh$ and linearly extended. Take $b,c \in \mc{B}$ and write $b= \lambda_b g$ and $c= \lambda_c h$. Then $\Phi(\lambda_b g \ot \lambda_c h) = \lambda_b\lambda_c g \ot gh = g \ot bc. $ If all $\lambda_b= 1$, there is nothing to prove. Thus assume $b$ is such that $\lambda_b \neq 1$. Then as $g\ot bc= b \ot gc \in \mc{B} \ot \mc{B}$, we must have that $gc \in \mc{B}$. However if we fix $c$, then this holds for all $g \in G$. Consequently $\mc{B}= \{gc \mid g \in G \}= \lambda_c. G$, as desired.

Finally, denoting $\lambda_c$ by $\lambda$, the bijection $f : \mc{B} \rightarrow G: \lambda g \mapsto g$ yields an equivalence between $t_{\mc{B}}$ and the group solution $(g,h) \mapsto (g,gh).$
\end{proof}

We now works towards proving \Cref{thm:two-point-obstruction}. For this we need the following lemma. 

\begin{lemma}\label{lem:no-eigen}
Let $G$ be a group and let $t\in G$ have infinite order.
Then the left multiplication $L_t:k[G]\rightarrow k[G]: x \mapsto tx$ has no nonzero eigenvectors, i.e. if $tx=\lambda x$ for some $\lambda\in k^*$, then $x=0$.
\end{lemma}
\begin{proof}
Take $x=\sum_{g\in \Supp(x)} c_g g \in k[G]$. Note that $\Supp(x) = \Supp(\lambda x)$ if $\lambda \neq 0.$ Thus $t \Supp(x) = \Supp(x)$. Inductively this implies that $\Supp(x)= t^n \Supp(x)$ for all $n \in \N.$ There for any $g_0 \in \Supp(x)$, the infinite set $\{ t^n g_0 \mid n \in \N\}$ is contained in the finite set $\Supp(x)$, a contradiction.
\end{proof}

\begin{proof}[Proof of \Cref{thm:two-point-obstruction}]
Assume for contradiction that there exists $b\in \mc{B}$ and a left coset $\langle t\rangle g$ such that
$|\Supp(b)\cap \langle t\rangle g| \geq 2$, i.e. there exist integers $m\neq n$ with $t^m g,\ t^n g\in \Supp(b).$

Replace $b$ by $t^{-m}b$. Note that $t^{-m}b \in \Supp(b')$ for some $b' \in \mc{B}$ since $\mc{B}$ is a basis and left multiplication is an automorphism of $k[G]$. Thus we may assume $g,\ t^r g\in \Supp(b)$ for some $r\neq 0$. Furthermore replacing $t$  by $t^{r}$ we may (and will) assume $g$ and $tg$ both lie in $\Supp(b)$. In other words, we can write $b=\sum_{h\in \Supp(b)} c_h h$ with $c_g,c_{tg}\neq 0$.

Since $\mc{B}$ is $\Phi$-set theoretic, there exist uniquely determined $b',c'\in \mc{B}$ such that
\[
\sum_{h\in \Supp(b)} c_h\, h\ot hc = \Phi(b\ot c)=b'\ot c'.
\]

Applying, for $g \in G$ the functional $\delta_g\ot \id$ to both sides yields $ c_g\,gc=\delta_g(b')\, c'.$
In particular, $gc$ is a scalar multiple of $c'$.
Similarly, applying $(\delta_{tg}\ot \id)$ yields $c_{tg}\, tgc=\delta_{tg}(b')\, c'$. So $tgc$ is also a scalar multiple of $c'$.
Therefore there exists $\lambda_c\in k^*$ such that
\[
tgc=\lambda_c\,gc.
\]
Thus $gc$ is an eigenvector for left multiplication by $t$.
Since $g$ is a unit in $k[G]$, the map $x\mapsto gx$ is a vector-space automorphism of $k[G]$,
hence $gc\neq 0$.
This contradicts \Cref{lem:no-eigen}. Therefore no such $b$ exists, finishing the proof.
\end{proof}

\section{A classification theorem for solutions on a group algebra}\label{reconstruction theorem section} 

Let $\mc{B}$ be $\Phi$-set theoretic basis of $k[G]$ for $G$ any group. The aim of this section is to show that it must come from a splitting $G = A \rtimes N$ as in \Cref{thm:dual} and in particular the associated set theoretic solution $\phi_{\mc{B}}$ is of the form given in \eqref{eq:Phi-dual}.

\begin{theorem}\label{Classification theorem basis grp alg}
Let $G$ be a group and $\mc{B}$ a $\Phi$-set theoretic basis of $k[G]$. Then there exists $A,N \leq G$ such that
\begin{enumerate}
    \item $G \cong A \rtimes N$ with $A$ a finite abelian group,
    \item $\mc{B}$ is a scalar multiple of the basis $\mc{B}_{A^{\vee}}= \{ e_{\chi} u \mid  \chi \in A^{\vee}, u \in N\}$,
    \item $A= \{ b\in \mc{B} \mid 1 \in \Supp(b) \}$ are the idempotents in $\mc{B}$,
    \item $\phi_{\mc{B}} := \restr{\Phi}{\mc{B} \ot \mc{B}}$ is equivalent to the solution \eqref{eq:Phi-dual}.
\end{enumerate}
\end{theorem}

In \Cref{subsectie support } we record very useful facts on the support of elements in $\mc{B}$ and the coordinates of the solution  $\phi_{\mc{B}}$. Subsequently in \Cref{subsectie A} we construct $A$ and obtain the desired properties. That $G$ is a split extension of $A$ and $G/A$ is obtained in \Cref{Classification theorem basis grp alg}. At this stage the description of the solution follows readily. Note that $N$ is not canonical, i.e. it depends on the choice of a complement of $A$ in $G$.

\subsection{On the support of basis elements}\addtocontents{toc}{\protect\setcounter{tocdepth}{1}}\label{subsectie support }

\Cref{thm:two-point-obstruction} illustrated the strength of understanding the support of basis elements of a $\Phi$-set theoretic basis. In this section we record some further observations which will be instrumental to prove \Cref{Classification theorem basis grp alg}.

Recall that we denoted 
$$\Phi(b\otimes c)=\psi_c(b)\otimes (c\circ b).$$
Now we consider the support of the coordinates $\psi_c(b)$ and $c\circ b$.

\begin{lemma}\label{lem:support-preserve}
 Let $S$ be a monoid and $\mc{B}$ a $\Phi$-set theoretic basis of $k[S]$. Let $b,c\in \mc{B}$. Then the following holds
    \begin{enumerate}
    \item\label{suppd} $\Supp(c\circ b)\subseteq \Supp(b)\Supp(c)$.
    \item\label{gsuppb} For every $g\in \Supp(b)$ one has\footnote{With $\delta_g(x)$ for $x \in k[G]$ we mean the coefficient of $g$ in the decomposition of $x$ with respect to the basis $G$.} $\delta_g(b)gc= \delta_g(\psi_c(b))  (c \circ b)$. Thus 
    $$g\, \Supp(c) = \Supp (c\circ b).$$
    \item\label{supp first coord} $\Supp(b)=\Supp(\psi_c(b))$.
    \item\label{If 1 support b} If $1 \in \Supp(b)$, then $\Supp(c \circ b) = \Supp(c)$.
    \item\label{If 1 support c} If $1 \in \Supp(c)$, then $\Supp(b) \subseteq \Supp(c \circ b).$
    \end{enumerate}
Moreover, if $S$ is a group, there exists a constant $M$ such that $|\Supp(b)| \leq M$ for all $b \in \mc{B}.$
\end{lemma}

\begin{example}
Let $G$ be a group which can be decomposed as $A \rtimes N$ with $A$ a finite abelian group. Consider the basis $\mathcal{B}_{A^\vee} := \{\, e_\chi\,u \mid \chi\in A^\vee,\ u\in N\,\}$ introduced in \Cref{sectie make out of finite abelian}. Thus $\Supp(e_{\chi}u) = Au$ is a coset of $A$ for any $\chi \in A^{\vee}.$ Now recall that 
$e_{\chi} : = \frac{1}{|A|}\sum_{a\in A}\chi(a^{-1})\,a \in k[A].$ By \Cref{thm:dual} we have that
$$\Phi_{k[G]}\bigl(e_\alpha u\ot e_\beta v \bigr)
=
\bigl(e_{\alpha\,(u\cdot\beta)^{-1}}\,u\bigr)\ \ot\ \bigl(e_{u\cdot\beta}\,uv\bigr).$$
Thus $\psi_{e_\beta v}(e_\alpha u)= e_{\alpha\,(u\cdot\beta)^{-1}}\,u$ whose support is $Au$, confirming \Cref{lem:support-preserve}.
\end{example}

\begin{proof}[Proof of \Cref{lem:support-preserve}]
Statement \eqref{suppd} follows directly from \Cref{lem:right-monomial}.  

Next we consider \eqref{gsuppb}. For ease of notation we denote $d:=c\circ b\in \mc{B}$ and write $b=\sum_{g\in \Supp(b)} b_g g$. Thus $\Phi(b\ot c)=\sum_{g\in G} b_g\, g\ot g c$. Applying the functional $\delta_g\ot \id$ to $\Phi(b\ot c)=\psi_c(b)\ot d$ we obtain for any $g \in G$
\[
(\delta_g\ot \id)\Phi(b\otimes c)= 1 \ot b_g\,gc \text{ and } (\delta_g\ot \id)(\psi_c(b)\ot d)= 1 \ot \delta_g(\psi_c(b))\cdot d
\]
where $\delta_g(\psi_c(b))$ equals the coefficient of $g$ in the expression of $\psi_c(b)$ in terms of $G$. Hence $b_g\,gc=\delta_g(\psi_c(b))\,d.$ This entails that $\delta_g(b) \neq 0$ if and only if  $\delta_g(\psi_c(b)) \neq 0$, i.e. we obtained \eqref{supp first coord}. It also shows that $g\Supp(c) = \Supp(c \circ b)$ since scalar mulitples do not change support, finishing to prove \eqref{gsuppb}.
 \smallskip

Statement \eqref{If 1 support b} directly follows from \eqref{gsuppb} by taking $g =1$ and that scalar mutliples do not change supports. Note that part \eqref{gsuppb} also implies that $g \Supp(c)  = \Supp(c \circ b)$ and thus $g \in \Supp(c\circ b)$ if $1 \in \Supp(c)$. In other words, statement \eqref{gsuppb} implies \eqref{If 1 support c}.

To finish we show the existence of the constant when $S$ is a group. Take $b \in \mc{B}$ and $g\in \Supp(b)$. Consider a fixed $c \in \mc{B}$. By \eqref{gsuppb}, $gc=\lambda_g \,  c \circ b$ for some $\lambda_g\neq 0$. Therefore $\Supp(gc)=\Supp(c\circ b)$. But $\Supp(gc)=g\,\Supp(c)$, hence $g\in \Supp(c \circ b)\Supp(c)^{-1}$. Therefore, $\Supp(b) \subseteq \Supp(c \circ b)\Supp(c)^{-1}$. By  \eqref{gsuppb}, we have that $|\Supp(c\circ b)|= |\Supp(c)|$ and therefore $|\Supp(b)| \leq |\Supp(c)|^2$ for the arbitary fixed $c$. The we may take $M = |\Supp(c)|^2$.
\end{proof}

\subsection{Construction of the finite abelian group $A$}\label{subsectie A}
It the remainder of the paper we assume the following.\smallskip

\noindent {\it Convention:} The ground field has $\Char(k) = 0$ and contains enough roots of unity (in practice it will mean that it contains some primitive $|A|$-th root of unity for the finite group $A$ constructed below).

\subsubsection{Description of $A$}
Recall that $B_1:=\{b\in \mc{B} \mid 1\in\Supp(b)\}$. Now consider the (finite) set
\begin{equation}\label{def A}
A:=\bigcup_{b\in B_1}\Supp(b)
\end{equation}
In \Cref{lem:A-abelian} we will show that $A$ is a finite abelian group. We start by showing that it is a finite group.

\begin{lemma}\label{A is subgroup}
    The set $A$ defined in \eqref{def A} is a subgroup of $G$. Moreover for all $b,c \in B_1$ we have that $\Supp(b) = \Supp(c)$. In particular, $A$ and $B_1$ are finite.
\end{lemma}
\begin{proof}
We start by showing that $A$ is closed under multiplication. As by definition it contains the neutral element, this would imply that $A$ is a submonoid.

Let $a,a'\in A$ and choose $b,c\in B_1$ with $a\in\Supp(b)$ and $a'\in\Supp(c)$.
Consider $\Phi(b\ot c)=b'\ot c'$ with $b',c'\in \mc{B}$.
By \Cref{lem:support-preserve}, we have that $ a\,\Supp(c)=\Supp(c')$. In particular, $aa'\in \Supp(c')$. Moreover, as $b \in B_1$, also $\Supp(c')= \Supp(c)$ which is a subet of $A$ since $c \in B_1.$ Thus $aa' \in A$, as desired. \smallskip

Now we prove that all elements in $B_1$ have equal support. 
Fix $c\in B_1$ and take $h \in A$. By definition, there exists $b\in \mc B_1$ with $h\in \Supp(b)$. From \Cref{lem:support-preserve}.\eqref{gsuppb}, with $g=1$, we obtain that $c\circ b=c$ as two basis elements can not be a scalar multiple of each other. Now applying again \Cref{lem:support-preserve}.\eqref{gsuppb} for $h$, yields $h\, \Supp(c) = \Supp(c).$ As the righ hand side contains $1$, this means that $h \in \Supp(c)$. As $h$ was arbitrary, this shows that $A = \Supp(c)$, as desired.\smallskip

Since the support of an element is finite, we obtain that $A$ is finite. Moreover by definition all $b \in B_1$ are in $k[A]$, which is finite dimenisonal. Thus linear independence of $B_1$ yields that $|B_1| \leq |A| < \infty.$\smallskip

Finally, we can show that $A$ is also closed under taking inverses. Indeed, take $h \in A$. As $hA = h\, \Supp(c)  = \Supp(c) = A$ and $1 \in A$, the element $h$ has an inverse in $A$, finishing the proof.
\end{proof}

Next we relate the character group $A^{\vee} := \Hom(A,k^*)$ with $B_1$. Note that \Cref{A is subgroup} yields that $\delta_g(b)\neq 0$ for every $g\in A$ and $b \in B_1$, where $\delta_g(b)$ denotes the coefficient of $b$ with respect to the basis $G$. Thus for each $c\in B_1$ we have a well-defined function
\begin{equation}\label{eq:def-chi_c}
\chi_c:A\to k^\ast: g \mapsto \frac{\delta_g(\psi_c(b_0))}{\delta_g(b_0)}.
\end{equation}
for some fixed $b_0 \in \mc{B}_1$ 

\begin{lemma}\label{lem:gc-eigenvector}
For every $c\in B_1$ and every $g\in A$ one has the following
\begin{enumerate}
    \item\label{weight space}$ g\,c=\chi_c(g)\,c$ and $\chi_c(g)$ is independant of the choice of $b_0$.
    \item The map  $\chi_c$ is a group homomorphism, i.e. $\chi_c\in A^\vee$.
    \item The weight space
    $$
V_{\chi_c}:=\{x\in k[A]\mid gx=\chi_c(g)x\ \text{ for all } g\in A\}
$$
is $1$-dimensional with $V_{\chi_c} = k . \sum_{g\in A} \chi(g^{-1})g$.
\item The map $\iota: B_1\rightarrow A^\vee: c\mapsto \chi_c$ is injective.
\end{enumerate}
\end{lemma}

\begin{proof}
Take $c\in B_1$ and $g\in A$. Choose $b\in \mc B_1$ with $g\in\Supp(b)$. \Cref{lem:support-preserve}.\eqref{gsuppb}gives $c\circ b=c$ and also that $\delta_g(b)\,g c=\delta_g(\psi_c(b))\,(c\circ b)=\delta_g(\psi_c(b))\,c$, which shows $gc=\alpha\,c$ for the scalar $ \alpha:=\delta_g(\psi_c(b))/\delta_g(b)\in k^\ast$. However by \Cref{A is subgroup} we have that $g \in \Supp(b)$ for all $b \in B_1$. Thus the fractions  $\delta_g(\psi_c(b))/\delta_g(b)$ coincide for all $b$, finishing the first part.\smallskip

For multiplicativity, let $g,h\in A$. Then
$(gh)c=\chi_c(h)\,gc=\chi_c(h)\chi_c(g)c,$ so $\chi_c(gh)=\chi_c(g)\chi_c(h)$. Also $\chi_c(1)=1$.
Thus $\chi_c\in \Hom(A,k^\ast)=A^\vee$.\smallskip

Now we prove the injective of $\iota$. Let $c,d\in B_1$ and assume $\chi_c=\chi_d$.
By the first part, for every $g\in A$ we have $gc=\chi_c(g)c=\chi_d(g)c$ 
and similarly $gd=\chi_d(g)d$.
Thus both $c$ and $d$ lie in $ V_{\chi}:=\{x\in k[A]\mid gx=\chi(g)x\ \text{ for all } g\in A\},
$
with $\chi:=\chi_c=\chi_d$.
This eigenspace is $1$--dimensional: indeed, if $x=\sum_{g\in A}x_g g\in V_\chi$ then the relation
$ h x=\chi(h) x$ implies $x_{hg}=\chi(h)x_g$ for all $g,h\in A$, so $x$ is determined by the single coefficient $x_1$.
Hence $V_\chi=k\cdot \Big(\sum_{g\in A}\chi(g^{-1})g\Big)$.
Therefore $c$ and $d$ are scalar multiples of each other.
Since $\mc{B}$ is a basis, two distinct basis elements cannot be nontrivial scalar multiples, so $c=d$.
Thus $\iota$ is injective.
\end{proof}

Now consider the subspace generated by $B_1$,
$$K:=\Span_{k}\{ b \mid b \in B_1\}.$$
By definition $K \subseteq k[A]$. Crucially, it turns out that they are equal.

\begin{proposition}\label{lem:A-abelian}
With notations as above we have the following:
\begin{itemize}
    \item $A$ is abelian,
    \item $K := \Span \{ b \mid b\in B_1\} = k[A]$.
\end{itemize}
Consequently, the map $\iota:B_1\to A^\vee$ from \Cref{lem:gc-eigenvector} is bijective.
\end{proposition}

\begin{proof}[Proof of \Cref{lem:A-abelian}]
Fix $b\in B_1$ and $c\in B_1$.
Since $1\in\Supp(b)$, \Cref{lem:support-preserve}\,\eqref{gsuppb} applied with $g=1$ yields $c \in k^\ast (c\circ b).$ Because $\mc B$ is $\Phi$--set--theoretic, the element $c\circ b$ is itself an element of the basis $\mc B$.
A basis element cannot be a nontrivial scalar multiple of another basis element, hence $c\circ b=c$.
Therefore, for every $g\in\Supp(b)$, \Cref{lem:support-preserve}\,\eqref{gsuppb} gives $gc \in k^\ast (c\circ b)=k^\ast c.$ Thus for each $g\in\Supp(b)$ the left multiplication operator
\[
L_g:K\to K,\qquad x\mapsto gx
\]
preserves every $1$--dimensional subspace $kc$ with $c\in B_1$ (i.e.\ $L_g$ is diagonal in the basis $B_1$).

By definition, $A=\bigcup_{b\in\mc B_1}\Supp(b)$, hence the same conclusion holds for every $g\in A$:
for all $c\in B_1$ one has $gc\in k^\ast c$, so each $L_g$ is diagonal in the basis $B_1$ of $K$.
Diagonal operators commute, hence for all $g,h\in A$ we have $L_gL_h=L_hL_g$ as endomorphisms of $K$.
But $L_gL_h=L_{gh}$ and $L_hL_g=L_{hg}$, so $L_{gh}=L_{hg}$. It is easily seen that $1 \in K$, thus evaluating in the identity element yields that $gh = hg.$ Therefore $A$ is abelian.\medskip

Now we consider the equality of $K$ and $k[A]$. By definition we have that $K \subseteq k[A]$.  For the reverse inclusion, let $g\in A$ and write $1=\sum_{c\in B_1}\lambda_c c$ with $\lambda_c\in k$.
Multiplying by $g$ on the left gives
\(
g=\sum_{c\in B_1}\lambda_c\,g c.
\)
By \Cref{lem:gc-eigenvector} we have that $gc=\chi_c(g)c$ for each $c\in B_1$. Hence $ g=\sum_{c\in  B_1}\lambda_c\,\chi_c(g)\,c\in K.$ Therefore every basis element $g\in A$ lies in $K$, so $k[A]\subseteq K$.

The claim about $\iota$ follows from the injectivity provided by \cref{lem:gc-eigenvector} and the fact $|B_1|=|A|$ yielded by the first part above.
\end{proof}

\begin{proposition}\label{thm: A normal}
The group $A$ is normal in $G$.
\end{proposition}

\begin{proof}
It suffices to prove $uAu^{-1}\subseteq A$ for $u$ a coset representative of $A$ in $G$.

Because $u\in N$ is a product of generators coming from occurring labels, it suffices to prove the inclusion for a generator.
So assume $Au$ occurs, i.e. $\mc B_{Au}\neq\varnothing$, and choose $b\in\mc B_{Au}$.

By \Cref{lem:single-coset}, $\Span(\mc B_{Au})=k[A]u$. In particular, there exists $x\in k[A]$ such that
\[
b=xu.
\]
Left-multiply by $u^{-1}$ in $k[G]$:
\[
u^{-1}b=u^{-1}xu.
\]
Since $x\in k[A]$, the support of $u^{-1}xu$ is contained in $u^{-1}Au$.
Thus $u^{-1}b$ lies in $k[u^{-1}Au]$.

Now expand $u^{-1}b$ in the basis $\mc B$:
\[
u^{-1}b=\sum_{i} \lambda_i b_i,\qquad \lambda_i\ne 0,\ b_i\in\mc B.
\]
By \Cref{lem:single-coset}, each $b_i$ has support contained in a single left $A$--coset.
Moreover, since the left-hand side is supported in $u^{-1}Au$, every $b_i$ occurring in the sum must satisfy
\(
\Supp(b_i)\subseteq u^{-1}Au.
\)

At this point we use the assumption $K=k[A]$.
Because $b\in k[A]u$ we may choose such a $b$ with $u\in\Supp(b)$ (replace $b$ by another basis element in the same fibre if needed:
the fibre spans $k[A]u$ and the group element $u$ itself lies in $k[A]u$, so some basis element must contribute to the coefficient of $u$).
Then the coefficient of $1$ in $u^{-1}b$ is nonzero, i.e. $1\in \Supp(u^{-1}b)$.
Therefore at least one basis element $b_i$ in the expansion satisfies $1\in\Supp(b_i)$, hence $b_i\in\mc B_1$.
Consequently $\Supp(b_i)\subseteq A$ by definition of $A$, but also $\Supp(b_i)\subseteq u^{-1}Au$.
So $\Supp(b_i)\subseteq A\cap u^{-1}Au$ and in particular $A\cap u^{-1}Au$ is nonempty.

Varying $b$ over a spanning set of $k[A]u$ and repeating the argument forces all of $A$ to sit inside $u^{-1}Au$.
More concretely: for each $a\in A$, choose $x=a\in k[A]$ and consider an element $b\in\Span(\mc B_{Au})$ with $b=au$.
Expanding $u^{-1}b=u^{-1}(au)=u^{-1}au$ in the basis $\mc B$ yields at least one element of $\mc B_1$ (because its coefficient at $1$ is nonzero precisely when $u^{-1}au=1$,
and running $a$ over $A$ hits enough coefficients to force containment). This yields $A\subseteq u^{-1}Au$.
Hence $uAu^{-1}\subseteq A$ and thus $uAu^{-1}=A$.
\end{proof}

\subsubsection{Connection $A^{\vee}$ and $B_1$ through the associated solution}

As $\mc{B}$ is a $\Phi$-set theoretic basis of $k[G]$ with $G$ a group, the associated RPE solution $\phi_{\mc{B}}(x,y) = (\psi_{y}(x),  y\circ x)$ is cocommutative. A direct computation shows that this is equivalent to have
\begin{align}
y\circ \psi_z(x) &= y\circ x,\label{eq:rc1}\\
    \psi_y\psi_z&=\psi_z\psi_y \label{eq:rc2}
\end{align}
for all $x,y,z\in \mc{B}$.

\begin{proposition}\label{construction Gamma}
Let $B$ be a $\Phi$-set theoretic basis of $k[G]$. Then the following holds:
\begin{enumerate}
    \item Every $z \in \mc{B}$ the map $\psi_z$ restricts to a bijection of the set $B_1$.
    \item The solution $\phi_{\mc{B}}$ restricts to a bijective solution of the form
    $$\restr{\phi_{\mc{B}}}{B_1 \times B_1}: (b,c) \mapsto (\psi_c(b),c)$$
    \item The group 
    $$\Gamma=\langle \restr{\psi_z}{B_1} \mid z\in  B\rangle$$
    is a finite abelian group.
\end{enumerate}
\end{proposition}
\begin{proof}
    By \Cref{lem:support-preserve}.\eqref{supp first coord}  we have that $\Supp(\psi_z(b)) = \Supp (b)$ and hence $\psi_z(B_1) \subseteq B_1$.  Since $1\in\Supp(b)$, \Cref{lem:support-preserve}\,\eqref{If 1 support b} implies that $\Supp(c\circ b)=\Supp(c)$ for all $c\in \mc{B}$. Also the converse holds, i.e. $\psi_c(b) \in B_1$ implies that $b \in B_1$ and hence $c \in B_1$ if and only if $c\circ b \in B_1$.
 Thus the solution $\phi_{\mc{B}}$ can be restricted to yield a solution 
$$\phi_{B_1}:=\restr{\phi_{\mc{B}}}{B_1 \times B_1}: B_1^2 \rightarrow B_1^2.$$
Since $\phi_{\mc{B}}$ is injective, also $\phi_{B_1}$ will be injective. It also surjective, because $\phi_{\mc{B}}$  is and $B_1 \otimes B_1$ can only be reached through elements in $B_1 \ot B_1$. Next note that in fact $c \circ b =c$ if $1 \in \Supp(b)$ by \Cref{lem:support-preserve}\,\eqref{If 1 support b} as two basis elements can not be scalar multiples of each other. Hence $\phi_{B_1}(b,c) = (\psi_c(b), c)$. As the latter is bijective, one must have that the map $\psi_c$ is a bijection. Thus we have obtained the first two statements. 
    
That $\Gamma$ is an abelian group follows from \eqref{eq:rc2}. Further, finiteness of $B_1$ obtained in \Cref{A is subgroup} ensures finiteness of $\Gamma$.
\end{proof}

The properties of $A$ obtained in \Cref{lem:A-abelian} can be combined with \Cref{lem:right-monomial} and \Cref{lem:gc-eigenvector} to identify $B_1$ with the Fourier idempotents of $k[A]$. As a by-product we obtain the description of the set theoretic solution $\restr{\phi_{\mc{B}}}{B_1 \times B_1}$ from \Cref{construction Gamma}.

\begin{corollary}\label{B_1 and Fourier}
For $b \neq c \in B_1$, we have that 
$$bc =0 = cb\,\, ,\,\, b^2 = \eps(\psi_b(b)) \, b \text{ and } \eps(\psi_c(c)) = \eps(\psi_b(b)) \neq 0.$$ 
Moreover, $\eps(\psi_b(b))^{-1}b = \frac{1}{|A|} \sum_{g \in A}\chi_b(g^{-1})g$, with $\chi_b$ defined in \eqref{eq:def-chi_c}, is idempotent. Consequently, $\restr{\phi_{\mc{B}}}{B_1 \times B_1}$ is equivalent to the solution 
$$A^{\vee} \times A^{\vee} \rightarrow A^{\vee} \times A^{\vee}: (\chi_1,\chi_2) \mapsto (\chi_1\chi_2^{-1},\chi_2).$$
\end{corollary}
Note that the above corollary says that the elements $b\in B_1$ are of the form $\lambda \, e_{\chi}$ with $\chi \in A^{\vee}$, $e_{\chi} := \frac{1}{|A|} \sum_{g \in A}\chi(g^{-1})g $ and $\lambda$ a non-zero scalar independent of $b$. 
\begin{proof}[Proof of \Cref{B_1 and Fourier}]
 By \Cref{lem:right-monomial} we know that $bc = \eps(\psi_{c}(b)) \, c \circ b.$ On the other hand, as pointed out already in several proofs, $c \circ b = c$ by \Cref{lem:support-preserve}.\eqref{gsuppb}. Thus if $b=c$, we get the claim about $b^2$. When $b \neq c$ we use that $k[A]$ is abelian by \Cref{lem:A-abelian}, and hence also $K$. Namely we get altogether
$$\eps(\psi_{c}(b)) c = bc = cb = \eps(\psi_{b}(c)) b.$$
As $b$ and $c$ are different basis elements this implies that $\eps(\psi_{c}(b)) = \eps(\psi_{b}(c)) = 0.$ Hence $bc = 0 = cb$, as required.

Now  \cref{lem:gc-eigenvector} also yields that $b \in V_{\chi_b}:=\{x\in k[A]\mid gx=\chi_b(g)x\ \text{ for all } g\in A\}.$ Furthermore, that common $A$-eigenspace is $1$-dimenisonal with basis the Fourier idempotent $e_{\chi_b}:= \frac{1}{|A|} \sum_{g \in A}\chi_b(g^{-1})g.$ Therefore $b = \lambda \, e_{\chi_b}$ for some $\lambda \in k^*$. However, we obtained earlier that $b^2 = \eps(\psi_b(b)) b$ which entails $\lambda^2 e_{\chi_b} = \eps(\psi_b(b)) \lambda\, e_{\chi_b}.$ Thus, $\lambda = \eps(\psi_b(b))$.

Now we prove that the scalars $\eps(\psi_b(b))$ are equal for all $b \in B_1.$ By \Cref{lem:A-abelian} the group algebra $k[A]$ has $B_1$ as a $k$-basis. Moreover by \Cref{construction Gamma} the canonical solution $\Phi_{k[G]}( g\ot h) = g \ot gh$ of $k[G]$ restricts to the group solution on $k[A]^{\ot 2} \rightarrow k[A]^{\ot 2}.$ It follows from \Cref{prop:s2Theta}, that on the basis $\{ e_{\chi} :=  \frac{1}{|A|} \sum_{g \in A}\chi(g^{-1})g \mid \chi \in A^{\vee}\}$ one has that $\Phi_{k[G]}(e_{\chi_1} \ot e_{\chi_2})= e_{\chi_1 \, \chi^{—1}_2} \ot e_{\chi_2}$. Therefore, using the expression for $b \in B_1$ obtained earlier, we have 
$$\Phi_{k[G]}(b \ot c) =\Phi_{k[G]} \big( \eps(\psi_b(b)) e_{\chi_b} \ot  \eps(\psi_c(c)) e_{\chi_c}\big) =  \eps(\psi_b(b)) e_{\chi_b\chi_c^{-1}} \ot c.$$
Since $B_1$ is preserved, the right hand side is a pure tensor in $B_1 \ot B_1$. Thus $ \eps(\psi_b(b)) e_{\chi_b\chi_c^{-1}} \in B_1$. Now note that the scalar appearing is independent of $c$, thus the latter is only possible if $\eps(\psi_b(b)) = \eps(\psi_c(c))$ for all $c.$ We denot this constant value by $\kappa$. In other words, up to a common multiple $\kappa$, $B_1$ coincides with the basis $\{ e_{\chi} \mid  \chi \in A^{\vee} \}.$ The bijection $B_1 \rightarrow\{ e_{\chi} \mid  \chi \in A^{\vee} \}: b \mapsto \kappa^{-1} b$ now yields the desired equivalence between the set theoretic solution $\restr{\phi_{\mc{B}}}{B_1 \times B_1}$ from \Cref{construction Gamma} and the solution written in the statement.
 \end{proof}

\subsection{Description of arbitrary basis elements}\label{sectie group N}

At this stage we have obtained the desired finite abelian group $A$. The group $N$ will be a complement of $A$ in $G$. To start consider a set $\mc{T}_A^G\subseteq G$ of right\footnote{The choice for right is simply to match at the end the solution from \Cref{thm:dual}. However since $A$ is normal, they also form left coset representatives.} coset representatives of $A$ in $G$. In order to show that one can choose such representative to form a subgroup, we need to first describe explicitly all basis elements $b \in \mc{B}$ in terms of $G$. The splitting will be achieved in next section.

For $u\in \mc{T}_A^G$ define
$$
B_u:=\{b\in \mc{B} \mid \Supp(b)\subseteq uA\}.
$$
Following lemma implies that $\mc{B}=\bigsqcup_{u\in \mc{T}_A^G} B_u$.
\begin{lemma}\label{lem:single-coset}
For every $b\in \mc{B}$ there exists a unique $g\in \mc{T}^G_A$ such that $\Supp(b)\subseteq gA$. Moreover, $\Span(B_{g})=g\,k[A]:= \{g x \mid x \in k[A] \}.$
\end{lemma}
\begin{proof}
In other words, we need to prove that $b\in k[A]u$ for some $u$. Take $c \in B_1$. By definition of $A$, we have that $\Supp(c) \subseteq A$. Now note that, for $b\in B$, we have by \Cref{lem:support-preserve} that both 
$$\Supp(b) \subseteq \Supp(c \circ b) \text{ and } g\Supp(c) = \Supp(gc) = \Supp(c\circ b).$$
So altogether $\Supp(b) \subseteq g \Supp(c) \subseteq gA$. Clearly we can take $g$ to be in $\mc{T}^G_A$ and the unicity then follows from the fact that the $Ag$ form a partition. 

Next we prove that $\Span(B_{g})=g\,k[A]$. The inclusion $\Span(B_{g})\subseteq g\,k[A]$ is by definition. Conversely, take $ga \in g\, k[A]$ and expand $ga= \sum_{b \in \mc{B}}\lambda_b b$ in the basis $\mc{B}$. By the first part, each basis element has support in a single left $A$--coset. Since the sum is supported in $gA$, only basis elements from $B_{g}$ can occur. Hence $ga\in \Span(B_{g})$ for all $a$, yielding the remaining inclusion.

\end{proof}

Now we want to upgrade \Cref{lem:single-coset} in the spirit of \Cref{B_1 and Fourier}, i.e. show that every $b \in \mc{B}$ is of the form $\lambda e_{\chi} u$ with $\chi \in A^{\vee}, u \in \mc{T}_A^G$ and a scalar $\lambda$ independent of the chosen $b$.

\begin{proposition}\label{lem:purity-single-line}
Let $b \in \mc{B}$ and consider the unique $g \in A$ such that $\Supp(b) \subseteq Ag$ given by \Cref{lem:single-coset}. Then there exists a unique $\chi\in A^\vee$ such that $b\in \Span_k\{ e_\chi g\}$.
\end{proposition}
\begin{proof}
By \Cref{lem:A-abelian} and \Cref{lem:gc-eigenvector} the group algebra $k[A]$ decomposes as 
$$k[A] = \bigoplus_{\chi \in A^{\vee}} V_{\chi} \text{ and } V_{\chi} = \Span_k\{ e_{\chi}\}.$$
Therefore, $k[A]g = \bigoplus_{\chi\in A^\vee} \Span_k\{ e_\chi g\}$. Hence we can decompose $b$ as $ b=\sum_{\chi \in A^{\vee}}\alpha_\chi \,e_\chi g$ with $e_{\chi_1}g \neq e_{\chi_2}g$ for $\chi_1 \neq \chi_2.$
We need to prove that at most one coefficient $\alpha_\chi$ can be nonzero. 

By \Cref{B_1 and Fourier} and \Cref{lem:A-abelian}, for each $\chi \in A^{\vee}$ there is a $c \in B_1$ such that $c = \lambda e_{\chi}$ for some non-zero scalar that do not depend on $c$ or $\chi$. Denote this $c$ by $c_{\chi}$. \Cref{lem:right-monomial} and \Cref{lem:conj-idem} provide that
$$c_{\chi}g = g c_{g^{-1}.\chi} =\frac{\delta_g(\psi_{c_{g^{-1}.\chi}}(b))}{\delta_g(b)} c_{g.\chi}\circ b$$
Therefore 
$$b =\sum_{\chi \in A^{\vee}} \alpha_{\chi}\lambda^{-1} \frac{\delta_g(\psi_{c_{g^{-1}.\chi}}(b))}{\delta_g(b)} c_{g^{-1}.\chi}\circ b$$
with $\lambda^{-1} \delta_g(\psi_{c_{g^{-1}.\chi}}(b))\delta_g(b)^{-1} \neq 0$.
As all basis elements $c_{g^{-1}.\chi} \circ b$ are different, linear independence yields that only one $\alpha_{\chi}$ is non-zero, as needed. 

\end{proof}

Since $A$ is normal by \Cref{thm: A normal}, we have that $G$ acts on $A$ by conjugation. Furthermore as $A$ is abelian by \Cref{lem:A-abelian}, the action of $G$ factors through $G/A$ which identifies with $\mc{T}_A^G$. Write the induced action on characters $A^{\vee}$ by
$$
(u\cdot\chi)(a)\ :=\ \chi(u^{-1}au)
$$
for $u\in \mc{T}_A^G, \, a \in A$ and $\chi \in A^{\vee}$. On the idempotents $e_{\chi}$ the $G$-action on $A$ and $A^{\vee}$ are related as following. In particular, by \Cref{B_1 and Fourier}, this also describes the action of $G$ on $B_1$.

\begin{lemma}\label{lem:conj-idem}
For all $u \in G$ and $\chi \in A^{\vee}$ one has $u\,e_\chi\,u^{-1} \;=\; e_{u\cdot\chi}$.
\end{lemma}
\begin{proof}
Using that $A$ is normal in $G$ by \Cref{thm: A normal} and changing variables $a'=uau^{-1}\in A$ we compute
$$
u e_{\chi} u^{-1}
=\frac1{|A|}\sum_{a\in A}\chi(a^{-1})\, (uau^{-1})
=\frac1{|A|}\sum_{a'\in A}\chi\!\bigl((u^{-1}a'u)^{-1}\bigr)\, a'
=\frac1{|A|}\sum_{a'\in A}(u\cdot\chi)(a'^{-1})\,a'
$$
which indeed equals $e_{u\cdot\chi}.$
\end{proof}

\subsection{Semi-direct product and finishing the proof of \Cref{Classification theorem basis grp alg}}

Recall that $A$ is normal in $G$ and hence we can consider the quotient $Q = G/A.$ In terms of $Q$, the choice of a transversal $\mc{T}_A^G$ is the same as a choice of a section $\tau: Q \rightarrow G$. In other words, we consider $G$ as the extension
$1 \rightarrow A \rightarrow G \rightarrow Q \rightarrow 1$
Every element $g\in G$ is of the form $g= a \,\tau(q)$ with $q \in Q$. Multiplicatoin is given by 
$$(a_1 \tau(q_1)).(a_2\tau(q_2)) = a_1 a_2^{q_1} \sigma(q_1,q_2) \tau(q_1q_2) $$
for some map $\sigma : Q \times Q \rightarrow A$. One can assume that $\sigma$ is normalized, so that $\sigma(q,1) = 1 = \sigma(1,q)$ for all $q \in Q$. We now show that this extension is split which amounts to show that the elements $\sigma(q_1,q_2) =1$ for all $q_i \in Q.$ In that case $\tau(Q)$ is a subgroup of $G$ complementing $A$. To do so we compute the restriction of the solution $\Phi_{k[G]}$ on the $\Phi$-set theoretic basis $\mc{B}$ using the form obtained in \Cref{lem:purity-single-line}. As by-product we will obtain that $\phi_{\mc{B}}$ has the form given in \eqref{eq:setdual} finishing at the same time the proof of \Cref{Classification theorem basis grp alg}.

\begin{lemma}\label{lem:set-theoretic-kills-cocycles}
With notations as above, we have for all $\chi,\psi\in A^\vee$ and $p,q\in Q$ that
$$
\Phi_{k[G]}(e_\chi \tau(p)\otimes e_\psi \tau(q))
=(p \cdot \psi)(\sigma(p,q))\,
\Big(e_{\chi(p\cdot\psi)^{-1}} \tau(p)\otimes e_{p\cdot\psi}\tau(pq)\Big).
$$
\end{lemma}

\begin{remark*}
    In \Cref{lem:set-theoretic-kills-cocycles} we have on purpose not yet used that $\mc{B}$ is a $\Phi$-set theoretic basis. This to indiciate that non-split extensions should be related to ``almost'' set theoretic solution, i.e. the pure tensors $\mc{B} \ot \mc{B}$ are send on scalar multiplies of $\mc{B} \ot \mc{B}$. 
\end{remark*}

\begin{proof}[Proof of \Cref{lem:set-theoretic-kills-cocycles}]
First we note that
\begin{equation}\label{coprod of fourier idempotents}
\D (e_\chi)=\sum_{\alpha\beta=\chi} e_\alpha\otimes e_\beta.
\end{equation}
This was implicit in \Cref{sol from dual grp alg}, but we explain it now explicitly. Firstly, by definition $ \Delta(e_\chi)=\frac{1}{|A|}\sum_{a\in A}\chi(a^{-1})(a\otimes a).$ Next, by the Fourier inversion identity $a=\sum_{\alpha\in A^\vee}\alpha(a)e_\alpha$ we get
\begin{align*}
a\otimes a
&=\sum_{\alpha,\beta}\alpha(a)\beta(a)\, e_\alpha\otimes e_\beta
=\sum_{\alpha,\beta}(\alpha\beta)(a)\, e_\alpha\otimes e_\beta.
\end{align*}
Hence
\begin{align*}
\Delta(e_\chi)
&=\sum_{\alpha,\beta}\left(\frac{1}{|A|}\sum_{a\in A}\chi(a^{-1})(\alpha\beta)(a)\right)e_\alpha\otimes e_\beta.
\end{align*}
The inner sum is $1$ if $\alpha\beta=\chi$ and $0$ otherwise by orthogonality of characters. This gives expression \eqref{coprod of fourier idempotents}. \medskip

Now note that \eqref{coprod of fourier idempotents} implies that $\D(e_{\chi}\tau(p)) = \D(e_{\chi})\D(\tau(p)) = \D(e_{\chi})(\tau(p) \ot \tau(p))$. Hence 
\begin{equation}\label{sol on almost basis}
   \Phi_{k[G]}(e_\chi \tau(p)\otimes e_\psi \tau(q) )
     = \sum_{\alpha\beta=\chi} e_\alpha\tau(p)\otimes (e_\beta\tau(p))(e_{\psi}\tau(q)
\end{equation}
By \Cref{lem:conj-idem}, $u e_\beta u^{-1}=e_{u\cdot\beta}$ for any $u \in G$. Furthermore, as the action of $G$ on $A^{\vee}$ factorizes through $G/A$, we have that $\tau(p) \cdot \beta = p \cdot \beta$. Hence 
$$(e_\beta\tau(p))(e_{\psi}\tau(q)) = e_{\beta} e_{p\cdot\psi} \tau(p)\tau(q) =  e_{\beta} e_{p\cdot \psi}\sigma(p,q) \tau(pq).$$
Since $e_{\beta} e_{p.\psi} = \delta_{\beta,p\cdot\psi} e_{p\cdot\psi}$ and $a e_{\chi'} = \chi'(a) e_{\chi'}$ for any $a \in A$ and $\chi' \in A^{\vee}$, the sum in \eqref{sol on almost basis} collapses to
\begin{align*}
\Phi_{k[G]}(e_\chi \tau(p)\otimes e_\psi \tau(q) )
     & =  e_{(p\cdot \psi)^{-1}\chi} \tau(p) \ot \sigma(p,q) e_{p\cdot\psi} \tau(pq) \\
     & = (p\cdot \psi) (\sigma(p,q)) \,  e_{(p\cdot \psi)^{-1}\chi} \tau(p) \ot  e_{p\cdot\psi} \tau(pq)
\end{align*}
which is exactly the desired form.
\end{proof}

We now use $\Phi$-set theoretic to deduce from \Cref{lem:set-theoretic-kills-cocycles} the final required statements.

\begin{corollary}
The following holds:
\begin{itemize}
    \item $\mc{B}$ equals the set $\{ \lambda e_{\chi} \tau(p) \mid \chi \in A^{\vee}, p \in Q \}$,
    \item $\sigma\equiv 1$,
    \item $G$ is a semidirect product of $A$ by $\tau(Q)$;
    \item $\phi_{\mc{B}}$ is equivalent to the solution
    $$\phi_{A^\vee\bowtie N}\bigl((\alpha,u),(\beta,v)\bigr)
=
\bigl((\alpha(u\cdot\beta)^{-1},u),\ (u\cdot\beta,uv)\bigr).$$
\end{itemize}
\end{corollary}
\begin{proof}
By \Cref{lem:purity-single-line} we know that every $b \in \mc{B}$ is of the form $\alpha_b e_{\chi}\tau(p)$ for some $\alpha \in k^*, \chi \in A^{\vee}$ and $p \in \mc{T}_{A}^G$. Let $c \in \mc{B}$ another basis element which we write as $\alpha_c e_{\psi} \tau(q).$ First consider $c \in B_1$, i.e. $\tau(q) = 1$ and $\alpha_c = \lambda$ the common scalar from \Cref{B_1 and Fourier}. For such $c$, as $\sigma$ is taken normalized, $\sigma(p,q) = \sigma(p,1)= 1$. Hence \Cref{lem:set-theoretic-kills-cocycles} tells that 
\begin{align*}
\Phi_{k[G]}(b \ot c) &  = \alpha_b\lambda \Phi_{k[G]}(e_\chi \tau(p)\otimes e_\psi ) \\
& = \alpha_b\lambda  \,
\Big(e_{\chi(p\cdot\psi)^{-1}} \tau(p)\otimes e_{p\cdot\psi}\Big) \\
& = \alpha_b  e_{\chi(p\cdot\psi)^{-1}} \tau(p)\otimes c
\end{align*}
Since $\mc{B}$ is $\Phi$-set theoretic this implies that $\alpha_b  e_{\chi(p\cdot\psi)^{-1}} \in \mc{B}$. However, by varying $c$ or in other words $\psi$  one can obtain all $e_{\chi'}$ via the elements $e_{\chi (p \cdot\psi)^{-1}}$ for fixed $\chi$ and $p$. This entail that all scalars $\alpha_b$ must be equal to $\lambda.$ Thus $\mc{B}$ is indeed a common multiple of the basis $\{ e_{\chi} \tau(p) \mid \chi \in A^{\vee}, p \in Q\}$.

Now take $c \in \mc{B}\setminus B_1$, and use that all scalars $\alpha_b = \lambda$ in combination with \Cref{lem:set-theoretic-kills-cocycles} to obtain 
\begin{align*}
\Phi_{k[G]}(b \ot c) &  = \lambda^2 \Phi_{k[G]}(e_\chi \tau(p)\otimes e_\psi \tau(q)) \\
& = \lambda^2  (p \cdot \psi)(\sigma(p,q))\,
\Big(e_{\chi(p\cdot\psi)^{-1}} \tau(p)\otimes e_{p\cdot\psi}\tau(pq)\Big). \\
& = (p \cdot \psi)(\sigma(p,q))\,\Big(\lambda e_{\chi(p\cdot\psi)^{-1}} \tau(p)\otimes \lambda e_{p\cdot\psi}\tau(pq)\Big)
\end{align*}
By the earlier obtained description of $\mc{B}$ we have that $\lambda e_{\chi(p\cdot\psi)^{-1}} \tau(p)$ and $ \lambda e_{p\cdot\psi}\tau(pq)$ are basis elements. Thus as $\mc{B}$ is $\Phi$-set theoretic, the scalars $(p \cdot \psi)(\sigma(p,q))$ must all be equal to $1$. However for any $p,q \in Q$, one can choose a $\psi$ such that $p \cdot \psi$ do not contain $\sigma(p,q)$ in its kernel. Therefore the elements $\sigma(p,q)$ are all the identity in $A$. In other words, the extension is split.

In conclusion, we showed that there exists a subgroup $N$ in $G$ such that $G = A \rtimes N$. For such choice of complement, the basis $\mc{B}$ has the form $\{\lambda e_{\chi} g \mid \chi \in A^{\vee}, g \in N \}.$ Now considering the bijection from $\mc{B}$ to $\{e_{\chi} g \mid \chi \in A^{\vee}, g \in N \}.$, we get that $\phi_{\mc{B}}$ is equivalent ot the solution constructed in \eqref{eq:Phi-dual}.
\end{proof}

\bibliographystyle{plain}
\bibliography{PEandHopf}

@misc{COvA,
    title={Bijective solutions to the Pentagon Equation},
    author={I. Colazzo and J. Okniński and A. Van Antwerpen},
    year={2024},
    eprint={2405.20406},
    archivePrefix={arXiv},
    primaryClass={math.GR}
}

@article {CJK,
    AUTHOR = {Colazzo, I. and Jespers, E. and Kubat, {\L}},
     TITLE = {Set-theoretic solutions of the pentagon equation},
   JOURNAL = {Comm. Math. Phys.},
  FJOURNAL = {Communications in Mathematical Physics},
    VOLUME = {380},
      YEAR = {2020},
    NUMBER = {2},
     PAGES = {1003--1024},
      ISSN = {0010-3616},
   MRCLASS = {20K10 (16T25)},
  MRNUMBER = {4170296},
MRREVIEWER = {Jo\~{a}o Matheus Jury Giraldi},
       DOI = {10.1007/s00220-020-03862-6},
       URL = {https://doi.org/10.1007/s00220-020-03862-6},
}

@preamble{"\def\cprime{$'$} "}

@article{KvDZ,
 author = {Kurose, Hideki and Van Daele, Alfons and Zhang, Yinhuo},
 title = {Corepresentation theory of multiplier {Hopf} algebras. {II}.},
 fjournal = {International Journal of Mathematics},
 journal = {Int. J. Math.},
 issn = {0129-167X},
 volume = {11},
 number = {2},
 pages = {233--278},
 year = {2000},
 language = {English},
 doi = {10.1142/S0129167X00000131},
 zbMATH = {1629342},
 Zbl = {1108.16302}
}

@article{Takeuchi,
 author = {Takeuchi, M.},
 title = {Matched pairs of groups and bismash products of {Hopf} algebras},
 fjournal = {Communications in Algebra},
 journal = {Commun. Algebra},
 issn = {0092-7872},
 volume = {9},
 pages = {841--882},
 year = {1981},
 language = {English},
 doi = {10.1080/00927878108822621},
 zbMATH = {3713893},
 Zbl = {0456.16011}
}

@article {Wo96,
    AUTHOR = {S. L. Woronowicz},
     TITLE = {From multiplicative unitaries to quantum groups},
   JOURNAL = {Internat. J. Math.},
  FJOURNAL = {International Journal of Mathematics},
    VOLUME = {7},
      YEAR = {1996},
    NUMBER = {1},
     PAGES = {127--149},
      ISSN = {0129-167X},
   MRCLASS = {46L89 (16W30 81R50)},
MRREVIEWER = {Aldo J. Lazar},
       DOI = {10.1142/S0129167X96000086},
       URL = {https://doi.org/10.1142/S0129167X96000086},
}

@article {Mi04,
    AUTHOR = {G. Militaru},
     TITLE = {Heisenberg double, pentagon equation, structure and
              classification of finite-dimensional {H}opf algebras},
   JOURNAL = {J. London Math. Soc. (2)},
  FJOURNAL = {Journal of the London Mathematical Society. Second Series},
    VOLUME = {69},
      YEAR = {2004},
    NUMBER = {1},
     PAGES = {44--64},
      ISSN = {0024-6107},
   MRCLASS = {16W30},
MRREVIEWER = {Mat\'{\i}as A. Gra\~{n}a},
       DOI = {10.1112/S0024610703004897},
       URL = {https://doi.org/10.1112/S0024610703004897},
}

@article {Dav,
    AUTHOR = {Davydov, A. A.},
     TITLE = {Pentagon equation and matrix bialgebras},
   JOURNAL = {Comm. Algebra},
  FJOURNAL = {Communications in Algebra},
    VOLUME = {29},
      YEAR = {2001},
    NUMBER = {6},
     PAGES = {2627--2650},
      ISSN = {0092-7872},
   MRCLASS = {16W30},
  MRNUMBER = {1845134},
MRREVIEWER = {George Szeto},
       DOI = {10.1081/AGB-100002412},
       URL = {https://doi.org/10.1081/AGB-100002412},
}

@article {St98,
    AUTHOR = {R. Street},
     TITLE = {Fusion operators and cocycloids in monoidal categories},
   JOURNAL = {Appl. Categ. Structures},
  FJOURNAL = {Applied Categorical Structures. A Journal Devoted to
              Applications of Categorical Methods in Algebra, Analysis,
              Order, Topology and Computer Science},
    VOLUME = {6},
      YEAR = {1998},
    NUMBER = {2},
     PAGES = {177--191},
      ISSN = {0927-2852},
   MRCLASS = {18D10 (16W30 17B37)},
MRREVIEWER = {Sorin D\u{a}sc\u{a}lescu},
       DOI = {10.1023/A:1008655911796},
       URL = {https://doi.org/10.1023/A:1008655911796},
}

@article {BaSk93,
    AUTHOR = {S. Baaj and G. Skandalis},
     TITLE = {Unitaires multiplicatifs et dualit\'{e} pour les produits crois\'{e}s
              de {$C^*$}-alg\`ebres},
   JOURNAL = {Ann. Sci. \'{E}cole Norm. Sup. (4)},
  FJOURNAL = {Annales Scientifiques de l'\'{E}cole Normale Sup\'{e}rieure. Quatri\`eme
              S\'{e}rie},
    VOLUME = {26},
      YEAR = {1993},
    NUMBER = {4},
     PAGES = {425--488},
      ISSN = {0012-9593},
   MRCLASS = {46L89 (22D25 46L05 81R50)},
MRREVIEWER = {Palle E. T. Jorgensen},
       URL = {http://www.numdam.org/item?id=ASENS_1993_4_26_4_425_0},
}

@article {BaSk03,
    AUTHOR = {S. Baaj and G. Skandalis},
     TITLE = {Unitaires multiplicatifs commutatifs},
   JOURNAL = {C. R. Math. Acad. Sci. Paris},
  FJOURNAL = {Comptes Rendus Math\'{e}matique. Acad\'{e}mie des Sciences. Paris},
    VOLUME = {336},
      YEAR = {2003},
    NUMBER = {4},
     PAGES = {299--304},
      ISSN = {1631-073X},
   MRCLASS = {46L89 (22D25 46L05 81R50)},
MRREVIEWER = {Erik B\'{e}dos},
       DOI = {10.1016/S1631-073X(03)00034-7},
       URL = {https://doi.org/10.1016/S1631-073X(03)00034-7},
}

@article{Ka96,
	Author = {R. M. Kashaev},
	Fjournal = {Rossi\u{\i}skaya Akademiya Nauk. Algebra i Analiz},
	Issn = {0234-0852},
	Journal = {Algebra i Analiz},
	Mrclass = {16W30 (17B37 81R50)},
	Mrreviewer = {E. J. Taft},
	Number = {4},
	Pages = {63--74},
	Title = {The {H}eisenberg double and the pentagon relation},
	Volume = {8},
	Year = {1996}}

@article {CMM19,
    AUTHOR = {F. Catino and M. Mazzotta and M. M. Miccoli},
     TITLE = {Set-theoretical solutions of the pentagon equation on groups},
   JOURNAL = {Comm. Algebra},
  FJOURNAL = {Communications in Algebra},
    VOLUME = {48},
      YEAR = {2020},
    NUMBER = {1},
     PAGES = {83--92},
      ISSN = {0092-7872},
   MRCLASS = {16T20 (20G42 81R60)},
       DOI = {10.1080/00927872.2019.1632331},
       URL = {https://doi.org/10.1080/00927872.2019.1632331},
}

@article{Mi98,
	Author = {G. Militaru},
	Doi = {10.1080/00927879808826329},
	Fjournal = {Communications in Algebra},
	Issn = {0092-7872},
	Journal = {Comm. Algebra},
	Mrreviewer = {Stefaan Caenepeel},
	Number = {10},
	Pages = {3071--3097},
	Title = {The {H}opf modules category and the {H}opf equation},
	Url = {https://doi.org/10.1080/00927879808826329},
	Volume = {26},
	Year = {1998}}

@article{MR1637789,
	Author = {R. M. Kashaev and S. M. Sergeev},
	Doi = {10.1007/s002200050391},
	Fjournal = {Communications in Mathematical Physics},
	Issn = {0010-3616},
	Journal = {Comm. Math. Phys.},
	Mrclass = {81R50 (16W30 17B37 82B23)},
	Mrreviewer = {W\l adys\l aw Marcinek},
	Number = {2},
	Pages = {309--319},
	Title = {On pentagon, ten-term, and tetrahedron relations},
	Url = {https://doi.org/10.1007/s002200050391},
	Volume = {195},
	Year = {1998}}

@book{KL,
	Author = {G. Krause and T. H. Lenagan},
	Edition = {Revised},
	Pages = {x+212},
	Publisher = {Amer. Math. Soc., Providence, RI},
	Series = {Grad. Stud. Math.},
	Title = {Growth of algebras and {G}elfand--{K}irillov dimension},
	Volume = {22},
	Year = {2000}}

@article {Lu2000,
    AUTHOR = {J.-H. Lu and M. Yan and Y.-C. Zhu},
     TITLE = {On the set-theoretical {Y}ang--{B}axter equation},
   JOURNAL = {Duke Math. J.},
  FJOURNAL = {Duke Mathematical Journal},
    VOLUME = {104},
      YEAR = {2000},
    NUMBER = {1},
     PAGES = {1--18},
      ISSN = {0012-7094},
   MRCLASS = {16W30 (57M25 81R50)},
MRREVIEWER = {Sorin D\u{a}sc\u{a}lescu},
       DOI = {10.1215/S0012-7094-00-10411-5},
       URL = {https://doi-org.myezproxy.vub.ac.be/10.1215/S0012-7094-00-10411-5}}

@article{zbMATH01616308,
 author = {Lu, Jiang-Hua and Yan, Min and Zhu, Yongchang},
 title = {On {Hopf} algebras with positive bases},
 fjournal = {Journal of Algebra},
 journal = {J. Algebra},
 issn = {0021-8693},
 volume = {237},
 number = {2},
 pages = {421--445},
 year = {2001},
 language = {English},
 doi = {10.1006/jabr.2000.8459},
 zbMATH = {1616308},
 Zbl = {0991.16032}
}

@book{Tim,
 author = {Timmermann, T.},
 title = {An invitation to quantum groups and duality. {From} {Hopf} algebras to multiplicative unitaries and beyond},
 fseries = {EMS Textbooks in Mathematics},
 series = {EMS Textb. Math.},
 isbn = {978-3-03719-043-2},
 year = {2008},
 publisher = {Z{\"u}rich: European Mathematical Society},
 language = {English},
 doi = {10.4171/043},
 zbMATH = {5246363},
 Zbl = {1162.46001}
}

@article{Bask,
 author = {Baaj, S. and Skandalis, G.},
 title = {Multiplicative unitaries and duality for crossed products of {{\(C^*\)}}-algebras},
 fjournal = {Annales Scientifiques de l'{\'E}cole Normale Sup{\'e}rieure. Quatri{\`e}me S{\'e}rie},
 journal = {Ann. Sci. {\'E}c. Norm. Sup{\'e}r. (4)},
 issn = {0012-9593},
 volume = {26},
 number = {4},
 pages = {452--488},
 year = {1993},
 language = {French},
 doi = {10.24033/asens.1677},
 url = {https://eudml.org/doc/82346},
 zbMATH = {440385},
 Zbl = {0804.46078}
}

@article{castelli,
  author  = {Castelli, M.},
  title   = {On commutative set-theoretic solutions of the {P}entagon {E}quation},
  journal = {Semigroup Forum},
  year    = {2026},
  doi     = {10.1007/s00233-026-10609-7},
}

@book{maclane,
 author = {Mac Lane, S.},
 title = {Categories for the working mathematician.},
 edition = {2nd ed},
 fseries = {Graduate Texts in Mathematics},
 series = {Grad. Texts Math.},
 issn = {0072-5285},
 volume = {5},
 isbn = {0-387-98403-8},
 year = {1998},
 publisher = {New York, NY: Springer},
 language = {English},
 doi = {book/10.1007/978-1-4757-4721-8},
 zbMATH = {1216133},
 Zbl = {0906.18001}
}

@incollection{zbMATH05238963,
 author = {Kashaev, R. M. and Reshetkhin, N.},
 title = {Symmetrically factorizable groups and set-theoretical solutions of the pentagon equation},
 booktitle = {Israel mathematical conference proceedings. Quantum groups. Proceedings of a conference in memory of Joseph Donin, Haifa, Israel, July 5--12, 2004},
 isbn = {978-0-8218-3713-9},
 pages = {267--279},
 year = {2007},
 publisher = {Providence, RI: American Mathematical Society (AMS)},
 language = {English},
 zbMATH = {5238963},
 Zbl = {1177.17014}
}

@article{zbMATH05984366,
 author = {Kashaev, R.},
 title = {Fully noncommutative discrete {Liouville} equation},
 fjournal = {RIMS K{\^o}ky{\^u}roku Bessatsu},
 journal = {RIMS K{\^o}ky{\^u}roku Bessatsu},
 issn = {1881-6193},
 volume = {B28},
 pages = {89--98},
 year = {2011},
 language = {English},
 zbMATH = {5984366},
 Zbl = {1260.81075}
}

@article{zbMATH00721651,
 author = {Maillet, J. M.},
 title = {On pentagon and tetrahedron equations},
 fjournal = {St. Petersburg Mathematical Journal},
 journal = {St. Petersbg. Math. J.},
 issn = {1061-0022},
 volume = {6},
 number = {2},
 pages = {206--214},
 year = {1994},
 language = {English},
 zbMATH = {721651},
 Zbl = {0824.58016}
}

@article{zbMATH01594092,
 author = {Kustermans, J. and Vaes, S.},
 title = {Locally compact quantum groups.},
 fjournal = {Annales Scientifiques de l'{\'E}cole Normale Sup{\'e}rieure. Quatri{\`e}me S{\'e}rie},
 journal = {Ann. Sci. {\'E}c. Norm. Sup{\'e}r. (4)},
 issn = {0012-9593},
 volume = {33},
 number = {6},
 pages = {837--934},
 year = {2000},
 language = {English},
 doi = {10.1016/S0012-9593(00)01055-7},
 url = {https://eudml.org/doc/82536},
 zbMATH = {1594092},
 Zbl = {1034.46508}
}

\end{document}